\documentclass{elsarticle}
\usepackage{amssymb,amsmath, amsthm, amsfonts,stmaryrd,mathtools}
\usepackage{mathrsfs,color}
\usepackage[all]{xypic}

\biboptions{square,sort,comma,numbers}

\usepackage{undertilde}
\usepackage{mathabx}
\usepackage{enumitem}
\usepackage{tikz}
\usetikzlibrary{matrix}
\usetikzlibrary{cd}
\usepackage{verbatim}

\newtheorem{theorem}{Theorem}[section]
\newtheorem{lemma}[theorem]{Lemma}
\newtheorem{corollary}[theorem]{Corollary}
\newtheorem{proposition}[theorem]{Proposition}

\theoremstyle{definition}
\newtheorem{definition}[theorem]{Definition}
\newtheorem{remark}[theorem]{Remark}
\newtheorem{example}[theorem]{Example}

\newcommand{\Q}{\mathrel{Q}}
\newcommand{\Qa}{\mathrel{Q_{\bf A}}}

\newcommand{\meet}{\wedge}
\newcommand{\join}{\vee}
\newcommand\upset{\mathord{\uparrow}}
\newcommand\downset{\mathord{\downarrow}}
\newcommand\Rm{{\bf RM^{t}}}

\newcommand{\twm}{\sqcap}
\newcommand{\twj}{\sqcup}
\newcommand{\twi}{\rightrightharpoons}
\newcommand{\twt}{\circ}

\newcommand{\la}{\langle}
\newcommand{\ra}{\rangle}

\newcommand{\Zo}{{\bf S}}
\newcommand{\Ze}{\mathord{{\bf S}\setminus\{0\}}}

\newcommand\les{\mathrel{\leq^\sharp}}
\newcommand\lesx{\mathrel{\leq_{\bf X}^\sharp}}
\newcommand\lesy{\mathrel{\leq_{\bf Y}^\sharp}}
\newcommand\lesa{\mathrel{\leq_{{\bf A}_*}^\sharp}}
\newcommand\lesb{\mathrel{\leq^{\bowtie}}}
\newcommand\slesb{\mathrel{<^{\bowtie}}}
\newcommand\pbow{\varphi^{\bowtie}}
\newcommand\inv{\mathord{\sim}}
\newcommand\ntwt{\bullet}
\newcommand\ntwi{\Rightarrow}
\newcommand\xia{\xi_{\bf A}}
\newcommand\xib{\xi_{\bf B}}
\newcommand\mux{\mu_{\bf X}}
\newcommand\muy{\mu_{\bf Y}}
\newcommand{\comp}{{\mathsf{c}}}

\newcommand{\grph}{\mathrm{grph}}

\usepackage{lipsum}
\makeatletter
\def\ps@pprintTitle{%
 \let\@oddhead\@empty
 \let\@evenhead\@empty
 \def\@oddfoot{}%
 \let\@evenfoot\@oddfoot}
\makeatother

\begin{document}

\begin{frontmatter}

\title{Categories of Models of ${\bf R}$-mingle}

\author{Wesley Fussner}
\ead{wesley.fussner@du.edu}
\author{Nick Galatos\corref{cor1}}
\ead{nikolaos.galatos@du.edu}
\address{Department of Mathematics, University of Denver, 2390 S. York St., Denver, CO 80208, USA}

\cortext[cor1]{Corresponding author}

\begin{abstract}
We give a new Esakia-style duality for the category of Sugihara monoids based on the Davey-Werner natural duality for lattices with involution, and use this duality to greatly simplify a construction due to Galatos-Raftery of Sugihara monoids from certain enrichments of their negative cones. Our method of obtaining this simplification is to transport the functors of the Galatos-Raftery construction across our duality, obtaining a vastly more transparent presentation on duals. Because our duality extends Dunn's relational semantics for the logic ${\bf R}$-mingle to a categorical equivalence, this also explains the Dunn semantics and its relationship with the more usual Routley-Meyer semantics for relevant logics.
\end{abstract}


\end{frontmatter}

\section{Introduction}

This study concerns a constellation of categories closely tied to semantics for the relevance logic ${\bf R}$-mingle. At the center of this constellation, the Sugihara monoids form the equivalent algebraic semantics for $\Rm$ (i.e., ${\bf R}$-mingle equipped with Ackermann constants \cite{AndBel1}).  Sugihara monoids have received extensive attention in the literature (see, e.g., \cite{AndBel1, BlokDziobiak, MM,OR,Raftery}), and are known to be equivalent to several neighboring categories (see \cite{GR1,GR2,Urquhart}). These categories are hence all pairwise equivalent, and the interplay between these equivalences is the object of this inquiry. Consequently, we are less concerned with the existence of the equivalences than the \emph{form} which they take. Our attention is therefore focused on the nature of the functors witnessing the equivalences. Scrutiny of these functors reveals how relationships among the categories considered may be transported to different regions of the constellation. This yields, \emph{inter alia}, a categorically-adequate relational semantics for $\Rm$ and an analogue of the twist product construction on dual spaces.

This work stems in part from the authors' efforts to explicate Dunn's Kripke-style semantics for ${\bf R}$-mingle \cite{Dunn}. Dunn's semantics stands out from the more widely-known Routley-Meyer semantics for relevance logics (see \cite{RM1,RM2,RM3}) because it employs a binary, rather than ternary, accessibility relation. We explain this state of affairs by introducing a topological duality for the Sugihara monoids that underwrites Dunn's semantics in the same way that the Esakia duality \cite{Esa} underwrites the Kripke semantics for intuitionistic and modal logics.

After summarizing some necessary background information, Section \ref{sec:twist} lays the groundwork for constructing this duality. We refine the equivalence depicted in \cite{GR1,GR2} between the Sugihara monoids and their enriched negative cones. The latter algebras form a class of relative Stone algebras augmented by a nucleus and a designated constant, and we show that enriching relative Stone algebras by only a designated constant is adequate to achieve categorical equivalence. We show also that this is tantamount to considering relative Stone algebras with a designated filter forming a Boolean algebra. In light of the latter fact, we call such algebras \emph{relative Stone algebras with Boolean constant}. As in \cite{GR1,GR2}, the functors witnessing the equivalence of this section are variants of the negative cone and twist product constructions. However, unlike the functors used in \cite{GR1,GR2}, the functors introduced in this section tie Sugihara monoids much more closely to their involutive lattice reducts, which proves indispensable in the sequel.

Section \ref{sec:booldual} introduces necessary background on the Priestley and Esakia dualities, and develops a duality for relative Stone algebras with Boolean constant and their bounded analogues. It also explains the connection between the duality introduced here and the Bezhanishvili-Ghilardi duality for Heyting algebras equipped with nuclei \cite{BezGhi}.

In Section \ref{sec:natdual}, we recall some facts about natural duality theory and the Davey-Werner natural duality \cite{DavWer} between Kleene algebras and certain structured topological spaces that we call \emph{Kleene spaces}. We then extend the Davey-Werner duality to algebras without lattice bounds, and introduce a class of special Kleene spaces that we call Sugihara spaces.

Section \ref{sec:sugidual} uses the results of the previous two sections to develop a topological duality for Sugihara monoids. This duality is anchored in the modified version of the Davey-Werner duality introduced in the previous section, and stands to the Davey-Werner duality in much the same way that the Esakia duality stands to Priestley's duality for bounded distributive lattices. In particular, we show that the category of Sugihara monoids is dually equivalent to the category of Sugihara spaces.

Section \ref{sec:refcon} introduces a covariant equivalence between certain categories of structured topological spaces. In particular, it gives an explicit connection between the duality of the previous section and Urquhart's well-known duality for relevant algebras \cite{Urquhart} that we call the \emph{reflection construction}. Because the Urquhart duality extends the Routley-Meyer semantics to a categorical equivalence in the same way our duality extends Dunn's semantics to a categorical equivalence, the reflection construction explains the connection between the Dunn and Routley-Meyer semantics for ${\bf R}$-mingle. The reflection construction also amounts to a translation of the functors of Section \ref{sec:twist} to dual spaces, giving a version of the twist product construction on the duals of algebras. This presentation of the twist product turns out to be vastly simpler than its manifestation on the algebraic side of the duality, opening the door to the possibility of generalizing the construction to wider contexts.

\begin{figure}
\begin{center}
\begin{tikzcd}[row sep=2.8cm, column sep=2.8cm]
\textsf{SM}_{\bot} \arrow[shift left =0.5ex]{d}{(-)_{\bowtie}}
\arrow[shift left =0.5ex]{r}{(-)_*}
\arrow[shift left =0.5ex]{dr}{(-)_+}
& \textsf{SRS} \arrow[shift left=0.5ex]{d}{(-)_{\bowtie}}
\arrow[shift left=0.5ex]{l}{(-)^*} \\
\arrow[shift left=0.5ex]{u}{(-)^{\bowtie}}
\textsf{bGA} \arrow[shift left =0.5ex]{r}{(-)_*}
& \textsf{SS}
\arrow[shift left =0.5ex]{ul}{(-)^+}
\arrow[shift left=0.5ex]{l}{(-)^*} 
\arrow[shift left=0.5ex]{u}{(-)^{\bowtie}}
\end{tikzcd}
\end{center}

\caption{The diagram above depicts several of the equivalences considered in this study. The left-hand side gives the equivalence between the category of bounded Sugihara monoids \textsf{SM}$_{\bot}$ and the category \textsf{bGA} of G\"odel algebras enriched with a Boolean constant given in Section \ref{sec:twist}. The bottom of the diagram refers to the Esakia duality for G\"odel algebras with Boolean constant articulated in Section \ref{sec:booldual}, whereas the top of the diagram refers to Urquhart's duality for relevant algebras, as specialized to bounded Sugihara monoids. The diagonal equivalence is an Esakia-style duality for bounded Sugihara monoids developed in Section \ref{sec:sugidual}. The right-hand side of the diagram alludes to the dual version of the equivalence on the algebraic side of the square, given in the work on topological twist products in Section \ref{sec:refcon}. All the equivalences except those involving the category of Sugihara relevant spaces \textsf{SRS} have analogues for algebras without universal lattice bounds as well.}
\label{fig:summary}
\end{figure}
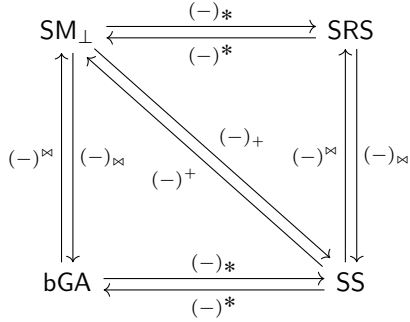

\section{Twist product representations for Sugihara monoids}\label{sec:twist}

We first recall some facts about commutative residuated lattices that are necessary to our investigation. For general reference on commutative residuated lattices and the proofs of the propositions alluded to here, we refer the reader to \cite{GJKO} and \cite{HRT}.

\subsection{Commutative residuated lattices} A \emph{commutative residuated lattice} (CRL) is an algebra $(A,\meet,\join,\cdot,\to,t)$ such that $(A,\meet,\join)$ is a lattice, $(A,\cdot, t)$ is a commutative monoid, and for all $a,b,c\in A$,
$$a\cdot b\leq c \iff a\leq b\to c$$
\noindent Note that the neutral element $t$ is sometimes denoted in the literature by $1$ or $e$.

A CRL need not enjoy bounds with respect to its underlying lattice order, but in the event that a CRL ${\bf A}$ possesses a lower bound $\bot$, it is bounded above as well. In fact, the upper bound of such a CRL is definable via the term $\bot\to\bot$. We thus refer to the expansion of a CRL ${\bf A}$ by a constant symbol $\bot$ designating a lower bound as a \emph{bounded} CRL. This expansion is term-equivalent to an expansion of ${\bf A}$ by constant symbols designating both the least and greatest elements of ${\bf A}$.

When ${\bf A}$ is an algebra with a (bounded) CRL reduct, we will denote its carrier by $A$ and its (bounded) lattice reduct by $\mathbb{A}$.

A CRL is called:

\begin{itemize}
\item \emph{integral} if the monoid identity is the greatest element with respect to its lattice order,
\item \emph{distributive} if its lattice reduct is a distributive lattice,
\item \emph{idempotent} if it satisfies the identity $x\cdot x = x$,
\item \emph{semilinear} if it is a subdirect product of totally-ordered CRLs.
\end{itemize}

\noindent The class of CRLs axiomatized by any (possibly empty) subset of the above conditions forms a variety. The following summarizes some significant quasiidentities that hold in these varieties.

\begin{proposition}\label{prop1}

Let ${\bf A} = (A,\meet,\join,\cdot,\to,t)$ be a CRL. Then ${\bf A}$ satisfies

\begin{enumerate}
\item $a\cdot (a\to b) \leq b$
\item $a\cdot (b\join c) = (a\cdot b) \join (a\cdot c)$
\item $a\to (b\meet c) = (a\to b)\meet (a\to c)$
\item $(a\join b)\to c = (a\to c)\meet (b\to c)$
\item $(a\cdot b)\to c = a\to (b\to c) = b\to (a\to c)$
\item $a\leq b \implies a\cdot c\leq b\cdot c$
\item $a\leq b \implies c\to a\leq c\to b$
\item $a\leq b \implies b\to c\leq a\to c$
\end{enumerate}

\end{proposition}

\begin{proposition}\label{prop2}

Let ${\bf A} = (A,\meet,\join,\cdot,\to,t)$ be a semilinear CRL. Then ${\bf A}$ satisfies

\begin{enumerate}
\item $t\leq (a\to b)\join (b\to a)$
\item $a\cdot (b\meet c) = (a\cdot b)\meet (a\cdot c)$
\item $a\to (b\join c) = (a\to b)\join (a\to c)$
\item $(a\meet b)\to c = (a\to c)\join (b\to c)$
\end{enumerate}

\end{proposition}
Note that a CRL is semilinear if and only if it is distributive and satisfies the identity (1) of Proposition \ref{prop2}.

A CRL for which $\cdot$ coincides with $\meet$ is called a \emph{Brouwerian algebra}, and an expansion of a Brouwerian algebra by a least element $\bot$ is called a \emph{Heyting algebra}. Brouwerian and Heyting algebras are among the most thoroughly-studied of all CRLs, and are integral, distributive, and idempotent. The semilinear Brouwerian algebras are called \emph{relative Stone algebras}, and the semilinear Heyting algebras are called \emph{G\"odel algebras}. We denote respectively by \textsf{Br}, \textsf{HA}, \textsf{RSA}, and \textsf{GA} the categories of Brouwerian algebras, Heyting algebras, relative Stone algebras, and G\"odel algebras. Here and whenever else we consider a category whose objects are algebras, we assume that the morphisms are the algebraic homomorphisms without additional comment. The following summarizes some useful algebraic properties of objects of these categories.

\begin{proposition}

Let ${\bf A}$ be an object of \textsf{Br}, \textsf{HA}, \textsf{RSA}, or \textsf{GA}. Then ${\bf A}$ satisfies the identities

\begin{enumerate}
\item $a\to a = t$
\item $a\meet (a\to b) = a\meet b$
\item $b\leq a\to b$
\end{enumerate}

\end{proposition}

\begin{proposition}[{{\cite[Lemma 4.1]{GR2}}}]\label{prop:topjoin}
Let ${\bf A}$ be an object of \textsf{RSA} and let $a,b\in A$. Then the following are equivalent.
\begin{enumerate}
\item $a\to b=b$ and $b\to a=a$.
\item $a\join b = t$.
\end{enumerate}
\end{proposition}

A \emph{nucleus} on a CRL ${\bf A}$ is a closure operator $N\colon {\bf A}\to {\bf A}$ satisfying the identity
$$Na\cdot Nb \leq N(a\cdot b)$$
\noindent One canonical way of defining a nucleus is given by the following.

\begin{example}\label{nucexample}

Let ${\bf A}=(A,\meet,\join,\to,t)$ be a Brouwerian algebra and let $d\in A$. Then the map $N\colon {\bf A}\to {\bf A}$ defined by $Na = d\to a$ is a standard example of a nucleus on ${\bf A}$.

\end{example}

Every CRL may be associated with an integral CRL via the \emph{negative cone} construction. Given a 
CRL ${\bf A} = (A,\meet,\join,\cdot,\to,t)$, let $A^- = \{a\in A : a\leq t\}$ be its collection of \emph{negative elements} and define the negative cone of ${\bf A}$ to be the algebra ${\bf A}^- = (A^-, \meet,\join,\cdot, \to^-, t)$, where $a\to^- b = (a\to b)\meet t$. Then ${\bf A}^-$ is a CRL, and it is obviously integral.

\subsection{Sugihara monoids and their negative cones}

A unary operation $\neg$ on a CRL that satisfies the identities $\neg\neg x=x$ and $\neg x\to y = y\to\neg x$ is called an \emph{involution}, and an expansion of a CRL by an involution is called an \emph{involutive} CRL. A \emph{Sugihara monoid} is a distributive, idempotent, involutive CRL. We denote the category of Sugihara monoids by $\textsf{SM}$, and the category of bounded Sugihara monoids by \textsf{SM}$_\bot$. The Sugihara monoids form a variety, and as Dunn proved in \cite{AndBel1}, they are semilinear.

The Sugihara monoids form the equivalent algebraic semantics (in the sense of \cite{BlokPigozzi}) for the relevance logic $\Rm$ of {\bf R}-mingle as formulated with Ackermann constants. The Sugihara monoids satisfying the identity $\neg t =t$ are called \emph{odd}, and the odd Sugihara monoids (with bounds) form the equivalent algebraic semantics of the logic {\bf IUML$^*$} ({\bf IUML}, respectively) of \cite{MM}.

We consider some examples of significant Sugihara monoids that will be useful in the sequel.

\begin{example}\label{oddex}

Define an algebra $\Zo = (\mathbb{Z}, \meet,\join,\cdot, \to, 0, -)$, where $\meet$ and $\join$ give the lattice operations of the usual order on the integers $\mathbb{Z}$, $-$ is the usual additive inversion on the integers, $\cdot$ is given by

\[ x\cdot y = \begin{cases} 
      x & |x|>|y| \\
      y & |x|<|y| \\
      x\meet y & |x|=|y| \\
   \end{cases}
\]\\
\noindent and $\to$ is given by

\[ x\to y = \begin{cases} 
      (-x)\join y & x\leq y \\
      (-x)\meet y & x\not\leq y \\
   \end{cases}
\]\\

\noindent Then $\Zo$ is a Sugihara monoid. $\Zo$ is obviously odd.

\end{example}

\begin{example}

Define a Sugihara monoid $\Ze = (\mathbb{Z} \setminus\{0\},\meet,\join,\cdot, \to, 1, -)$, where each of the non-nullary operations are defined as in Example \ref{oddex}. Then $\Ze$ is a Sugihara monoid with monoid identity $1$. $\Ze$ is not odd.

\end{example}

\begin{example}

Let $n$ be a positive integer. If $n=2m+1$ is odd, the set $\{-m,\dots, -1,0,1,\dots, m\}$ is the universe of a subalgebra of $\Zo$ having $n$ elements. If $n=2m$ is even, then $\{-m,\dots,-1,1,\dots m\}$ is the universe of a subalgebra of $\Ze$ having $n$ elements. We denote the $n$-element Sugihara monoid so defined by ${\bf S}_n$. The Sugihara monoid ${\bf S}_n$ is odd if and only if $n$ is odd.

\end{example}

\begin{figure}
\begin{center}
\begin{tikzpicture}
    \node[label=right:\tiny{$\la 2,2\ra$}] at (0,0) {$\bullet$};
    \node[label=left:\tiny{$\la 1,1\ra$}] at (0,-0.5) {$\bullet$};
    \node[label=left:\tiny{$\la 0,1\ra$}] at (-0.5,-1)  {$\bullet$};
    \node[label=right:\tiny{$\la 1,-1\ra$}] at (0.5,-1) {$\bullet$};
    \node[label=right:\tiny{$\la 0,-1\ra$}] at (0,-1.5) {$\bullet$};
    \node[label=left:\tiny{$\la -1,1\ra$}] at (-1,-1.5) {$\bullet$};
    \node[label=right:\tiny{$\la -1,-1\ra$}] at (-0.5,-2) {$\bullet$};
    \node[label=right:\tiny{$\la -2,-2\ra$}] at (-0.5,-2.5) {$\bullet$};

    \draw (0,0) -- (0,-0.5);
    \draw (0,-0.5) -- (-0.5,-1);
    \draw (0,-0.5) -- (0.5,-1);
    \draw (-0.5,-1) -- (0,-1.5);
    \draw (0.5,-1) -- (0,-1.5);
    \draw (-0.5,-1) -- (-1,-1.5);
    \draw (-1,-1.5) -- (-0.5,-2);
    \draw (0.5,-1) -- (-0.5,-2);
    \draw (-0.5,-2) -- (-0.5,-2.5);
\end{tikzpicture}
\end{center}
\caption{Hasse diagram for ${\bf E}$}
\label{fig:HasseE}
\end{figure}
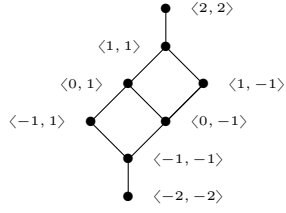

\begin{example}\label{parex}
We may define a nonlinear example on the subuniverse of ${\bf S}_5\times{\bf S}_4$ given by $$E = \{\la -2,-2\ra,\la -1,-1\ra,\la -1,1\ra,\la 0,-1\ra, \la 0,1\ra, \la 1,-1\ra, \la 1,1\ra, \la 2,2\ra\}.$$ $E$ forms the carrier of a subalgebra ${\bf E}$ of  ${\bf S}_5\times {\bf S}_4$, whose Hasse diagram is given in Figure \ref{fig:HasseE}.
\end{example}

With these examples in mind, we recall the following well-known fact (see, e.g., \cite{OR}).

\begin{proposition}\label{prop:Sugiharagenerators}
The Sugihara monoids are generated as a quasivariety by $\{\Zo, \Ze\}$.
\end{proposition}

The central result of \cite{GR2} establishes that \textsf{SM} is equivalent to the category \textsf{EnSM$^-$} of enriched negative cones of Sugihara monoids, which we define presently. The objects of \textsf{EnSM$^-$} are algebras ${\bf A} = (A, \meet, \join, \to, t, N, f)$, where $(A,\meet,\join,\to, t)$ is a relative Stone algebra, $N$ is a nucleus on {\bf A}, and $f\in A$, all satisfying the universal conditions
\begin{align*}
a\join (a\to f) &= t\\
N (Na\to a) &= t\\
Na = t \iff &f\leq a
\end{align*}
\noindent Similarly define \textsf{EnSM$^-_\bot$} to be the category whose objects are expansions of objects of \textsf{EnSM$^-$} by a least element $\bot$ and whose morphisms are those of \textsf{EnSM$^-$} preserving the constant $\bot$.

The covariant functors $C$ and $S$, defined as follows, witness the equivalence of \textsf{EnSM$^-$} and \textsf{SM}. First, define the functor $C\colon \textsf{SM}\to \textsf{EnSM$^-$}$ for a Sugihara monoid ${\bf A} =(A,\meet,\join,\cdot,\to,t,\neg)$ of \textsf{SM} by $C({\bf A}) = ({\bf A}^-, N, \neg t)$, where $N$ is the nucleus on ${\bf A}^-$ defined by $Na = (a\to t)\to t$. For a morphism $h\colon {\bf A}\to {\bf B}$ of \textsf{SM}, define $C(h)\colon C({\bf A})\to C({\bf B})$ by $C(h) = h\restriction_{A^-}$, the restriction of $h$ to the collection of negative elements of ${\bf A}$.

To obtain the reverse functor, for an object ${\bf A} =(A,\meet,\join,\to,t, N, f)$ of \textsf{EnSM$^-$} define
$$\Sigma({\bf A}) = \{\langle a,b\rangle \in A\times A : a\join b = t \text{ and } Nb=b\}.$$
\noindent Define the functor $S\colon \textsf{EnSM$^-$}\to \textsf{SM}$ on objects ${\bf A} =(A,\meet,\join,\to,t, N, f)$ of \textsf{EnSM$^-$} by $S({\bf A}) = (\Sigma ({\bf A}), \twm,\twj,\twt,\twi, \langle t,t \rangle, \neg)$, where 
\begin{align*}
\langle a,b\rangle \twm \langle c,d \rangle &= \langle a\meet c, b\join d \rangle\\
\langle a,b\rangle \twj \langle c,d\rangle &= \langle a\join c, b\meet d\rangle\\
\langle a,b\rangle \twt \langle c,d\rangle &= \langle ((a\to d)\meet (c\to b))\to (a\meet c), N((a\to d)\meet (c\to b))\rangle\\
\langle a,b\rangle \twi \langle c,d\rangle &= \langle (a\to c)\meet (d\to b), N(((a\to c)\meet (d\to b))\to (a\meet d))\rangle\\
\neg \langle a,b \rangle &= \langle a,b\rangle \twi \langle f,t \rangle \\
&= \la (a\to f)\meet b, N(((a\to f)\meet b)\to a) \ra
\end{align*}

\noindent For a morphism $h\colon {\bf A}\to {\bf B}$ of \textsf{EnSM$^-$}, define a morphism  $S(h)\colon S({\bf A})\to S({\bf B})$ of \textsf{SM} by $S(h)\la a,b \ra = \la h(a), h(b) \ra$. Under these definitions, the functors $C$ and $S$ yield a (covariant) equivalence between the categories \textsf{EnSM$^-$} and \textsf{SM}. Moreover, this equivalence may be extended to the bounded algebras arising from objects of \textsf{EnSM$^-$} and \textsf{SM}, giving an equivalence between the corresponding categories of bounded algebras \textsf{EnSM$^-_\bot$} and \textsf{SM}$_\bot$. $C$ and $S$ are extended as follows in order to obtain the latter equivalence. If $({\bf A},\bot)$ is an object of \textsf{SM}$_\bot$, extend the definition of $C$ by associating to $({\bf A}, \bot)$ the object $(C({\bf A}), \bot)$ of \textsf{EnSM$^-_\bot$}. Likewise, if $({\bf A},\bot)$ is an object of \textsf{EnSM$^-_\bot$}, extend $S$ by associating with $({\bf A},\bot)$ the Sugihara monoid $S({\bf A})$ with designated lower-bound $la \bot,t\ra$.

In \cite{GR2}, the functor $C$ was called the \emph{nuclear negative cone functor}. On the other hand, the functor $S$ is a variant of the twist product construction, originally introduced by Kalman in \cite{Kalman} (but see also, e.g., \cite{Fidel, Kracht, O1, O2, O3, TsinWille} for a sample of the rapidly-growing literature on twist products). It is noteworthy that the involution arising from $S$ does not coincide with the usual twist product involution $\la a,b\ra\mapsto \la b,a\ra$, although it does when the equivalence depicted above is restricted to \emph{odd} Sugihara monoids, as chronicled in \cite{GR1}. This mismatch between the usual twist product involution and the involution arising from $S$ proves undesirable for the applications that follow, so we first recast the construction from \cite{GR2} in order to restore the simple involution $\la a,b\ra\mapsto \la b,a\ra$. This requires further scrutiny of the algebraic structure of the variety \textsf{EnSM$^-$}.

\subsection{Algebras with Boolean constant}

Let ${\bf A}$ be a Brouwerian algebra. We call a lattice filter $F$ of ${\bf A}$ a \emph{Boolean filter} if $F$, considered as a lattice with the operations inherited from ${\bf A}$, is a Boolean lattice (i.e., a complemented, bounded, distributive lattice). Note that we admit the one-element Boolean lattice as a potential Boolean filter, and under this convention every Brouwerian algebra has at least one Boolean filter (i.e., $\{t\}$, where $t$ is the greatest element of the Brouwerian algebra).

\begin{lemma}\label{lem:boolneg}
Let ${\bf A} = (A,\meet,\join,\to,t)$ be a Brouwerian algebra, $F$ be a Boolean filter of ${\bf A}$ with least element $f$, and $a\in F$. Then the complement of $a$ in $F$ is precisely $a\to f$.
\end{lemma}

\begin{proof}
Note that $a\to f\geq f$ gives that $a\to f\in F$. Since $a\in F$ as well, this shows that $a\meet (a\to f)\in F$. But $a\meet (a\to f)\leq f$, and as $f$ is the least element of $F$, it follows that $a\meet (a\to f)=f$. On the other hand, since $F$ is a Boolean filter and $a\in F$, $a$ has a complement $c$ in $F$. This gives that $a\meet c\leq f$, so by residuation we get $c\leq a\to f$. Then $t=a\join c\leq a\join (a\to f)$, so $a\join (a\to f)=t$. It follows that $a\to f$ is the complement of $a$ in $F$.
\end{proof}

\begin{proposition}\label{bool}
Let ${\bf A} = (A,\meet,\join,\to,t)$ be a Brouwerian algebra and let $f\in A$. Then the following are equivalent.
\begin{enumerate}
\item $a\join (a\to f)=t$ for all $a\in\upset f$.
\item $a\join (a\to f)=t$ for all $a\in A$.
\item $\upset f$ is a Boolean lattice.
\end{enumerate}
\end{proposition}

\begin{proof}
First, we show that (1) implies (3), so suppose that $a\join (a\to f)=t$ for all $a\in\upset f$. Let $a\in\upset f$. Then $a\meet (a\to f)\leq f$, so as $a\to f\geq f$ yields $a,a\to f\in\upset f$ this gives $a\meet (a\to f)=f$. On the other hand, $a\join (a\to f)=t$ by hypothesis. This shows that each $a\in\upset f$ has a complement (namely, $a\to f$), and hence that $\upset f$ is a Boolean filter.

Second, we show that (3) implies (2). Suppose that $\upset f$ is a Boolean filter, and let $a\in A$. Then since $a\to f\geq f$, we have that $a\join (a\to f)\in\upset f$ and hence has a complement  in $\upset f$, and this complement is $(a\join (a\to f))\to f$ by Lemma \ref{lem:boolneg}. Observe that
\begin{align*}
t &= (a\join (a\to f))\join ((a\join (a\to f))\to f)\\
&= (a\join (a\to f))\join ((a\to f)\meet ((a\to f)\to f))\\
&\leq (a\join (a\to f))\join f\\
&\leq a\join (a\to f)
\end{align*}
This gives that $a\join (a\to f)=t$ as desired.

Since (2) implies (1) trivially holds, this gives the result.
\end{proof}

In light of Proposition \ref{bool}, we call an expansion of a Brouwerian algebra (Heyting algebra) ${\bf A}$ by a designated constant $f$ satisfying $a\join (a\to f) = t$ a \emph{Brouwerian algebra with Boolean constant} (respectively, \emph{Heyting algebra with Boolean constant}). For the present purposes, our interest is focused on the semilinear members of these classes. We thus denote the category of relative Stone algebras with Boolean constant by \textsf{bRSA}. Likewise, we denote the category of G\"odel algebras with Boolean constant by \textsf{bGA}. For brevity, we respectively call the objects of these categories bRS-algebras and bG-algebras.

In spite of the defining condition $Na=t \iff f\leq a$, the objects of \textsf{EnSM$^-$} turn out to form a variety. The subdirect irreducibles in this variety are characterized by the comments on pp. $3207$ and $3192$ of \cite{GR2} as follows.

\begin{proposition}

An object $(A,\meet,\join,\to,t,N,f)$ of \textsf{EnSM$^-$} is subdirectly irreducible iff it is totally ordered, $\{a\in A : a < t\}$ has a greatest element, and one of the following holds:

\begin{enumerate}

\item $f=t$ and $N$ is the identity function on $A$, or
\item $f$ is the greatest element of $\{a\in A : a <t\}$, $Nf=t$, and $Na=a$ whenever $a\neq f$.

\end{enumerate}

\end{proposition}

\noindent By arguing on generating algebras for the variety, we obtain the following.

\begin{lemma}\label{nterm}

\textsf{EnSM$^-$} satisfies the identity $Na = f\to a$.

\end{lemma}

\begin{proof}

It suffices to check the identity $Na=f\to a$ on subdirectly irreducible algebras, so let ${\bf A}=(A,\meet,\join,\to,t,N,f)$ be a subdirectly irreducible algebra in \textsf{EnSM$^-$}. If $f=t$ and $N$ is the identity function on $A$, then the result trivially follows since $f\to a = t\to a = a$ for any $a\in A$ and $Na=a$ for any $a\in A$.

In the remaining case, ${\bf A}$ is a chain and $N$ satisfies
\[ Na = \begin{cases} 
      t & x=f,t \\
      a & a\neq f,t \\
   \end{cases}
\]
\noindent Note also that in any totally-ordered Brouwerian algebra,
\[ x\to y = \begin{cases} 
      t & x\leq y \\
      y & x\not\leq y \\
   \end{cases}
\]
\noindent We may therefore compute
\[ f\to a = \begin{cases} 
      t & f\leq a \\
      a & f\not\leq a \\
   \end{cases}
\]

\noindent Since $t$ covers $f$ in the present case, we have that $f\leq a$ iff $a=f$ or $t$, which gives the result.
\end{proof}

\begin{proposition}\label{termeq}

\textsf{EnSM$^-$} is term-equivalent to \textsf{bRSA}, and \textsf{EnSM}$^-_\bot$ is term-equivalent to \textsf{bGA}.

\end{proposition}

\begin{proof}
Lemma \ref{nterm} shows that in any object ${\bf A} = (A,\meet,\join,\to,t,N,f)$ of \textsf{EnSM$^-$}, $N$ is definable in the $(\meet,\join,\to,t,f)$-reduct of ${\bf A}$. Since the $(\meet,\join,\to,t,f)$-reduct of such an object ${\bf A}$ of \textsf{EnSM$^-$} satisfies $a\join (a\to f)=t$ by definition, such an ${\bf A}$ is a bRS-algebra.

On the other hand, suppose that ${\bf A} = (A,\meet,\join,\to,t,f)$ is a bRS-algebra, and define $N\colon A\to A$ by $Na=f\to a$. Then $N$ is a nucleus by Example \ref{nucexample}. Moreover, observe that for any $a\in A$,
\begin{align*}
N(Na\to a) &= f\to ((f\to a)\to a)\\
&= (f\to a) \to (f\to a)\\
&= t
\end{align*}

\noindent so the identity $N(Na\to a)=t$ holds.

To see that the condition $Na=t$ iff $f\leq a$ holds, observe that
\begin{align*}
Na=t &\iff f\to a = t\\
&\iff t\leq f\to a\\
&\iff f\leq a
\end{align*}

\noindent It follows that every bRS-algebra is the $(\meet,\join,\to,t,f)$-reduct of an object of \textsf{EnSM$^-$}, so that \textsf{EnSM$^-$} is term-equivalent to \textsf{bRSA}. The term-equivalence of \textsf{EnSM}$^-_\bot$ and \textsf{bGA} follows by the same argument.
\end{proof}

The previous proposition shows that the addition of the nucleus to the signature is extraneous in the definition of \textsf{EnSM$^-$}. In order to obtain an equivalence between \textsf{SM} and the (enriched) negative cones of its members, we therefore need only consider expansions of the negative cones by a single designated constant rather than a designated constant and a nucleus. In particular, \textsf{SM} is categorically equivalent to \textsf{bRSA} and \textsf{SM}$_\bot$ is categorically equivalent to \textsf{bGA}. We modify the functors $C$ and $S$ described above to obtain this equivalence as follows. Define $S\colon \textsf{bRSA}\to\textsf{SM}$ in the same way as before, but replacing instances of $N$ in the definitions of $\twt$ and $\twi$ with $Na=f\to a$. Define a functor $(-)_{\bowtie}\colon\textsf{SM}\to\textsf{bRSA}$ by ${\bf A}_{\bowtie} = ({\bf A}^-, \neg t)$. Then replacing $S$ and $C$ with the new $S$ and $(-)_{\bowtie}$ produces an equivalence of categories between \textsf{SM} and \textsf{bRSA}. Similar remarks apply to \textsf{SM}$_\bot$ and \textsf{bGA}.

For the treatment to follow, it is desirable that we replace $S$ by a different functor that situates the equivalence more naturally among existing work on twist products. For a bRS-algebra ${\bf A} = (A,\meet,\join,\to,t,f)$, define\footnote{Here we borrow notation from the twist product construction. This should not, however, be confused with what is sometimes referred to in the literature as the \emph{full} twist product.} $$A^{\bowtie} = \{\la a,b\ra\in A\times A : a\join b =t \text{ and } a\meet b \leq f\}$$

\noindent Moreover, for $\la a,b\ra,\la c,d\ra\in A\times A$, define $\la a,b\ra\twm \la c,d\ra = (a\meet c, b\join b)$ and $\la a,b\ra \twj \la c,d\ra = \la a\join c, b\meet d\ra$ as in the definition of $S$. Then $(A\times A, \twm, \twj)$ is a lattice.

\begin{lemma}\label{reslem0}

Let ${\bf A} = (A,\meet,\join,\to,t,f)$ be a bRS-algebra. Then $\Sigma ({\bf A})$ and $A^{\bowtie}$ are universes of sublattices of $(A\times A,\twm,\twj)$.

\end{lemma}

\begin{proof}
Suppose that $\la a, b\ra,\la c, d\ra\in A\times A$ satisfy $a\join b = c\join d = t$. Then by distributing,
\begin{align*}
(a\meet c)\join (b\join d) &= ((a\join b) \meet (c\join b))\join d\\
&= (t\meet (c\join b))\join d\\
&= t
\end{align*}

\noindent and a symmetric argument shows that $(a\join c)\join (b\meet d)=t$ as well.

Next suppose $\la a, b\ra,\la c, d\ra\in A\times A$ with $Nb=b$ and $Nd=d$, where the nucleus $Nx=f\to x$ is defined as above. Then $N(b\meet d)=b\meet d$ by Proposition \ref{prop1}(3), and $N(b\join d)=b\join d$ by Proposition \ref{prop2}(3).

Finally, suppose that $\la a, b\ra,\la c, d\ra\in A\times A$ with $a\meet b\leq f$ and $c\meet d\leq f$. Then
\begin{align*}
(a\meet c)\meet (b\join d) &= (a\meet c\meet b)\join (a\meet c\meet d)\\
&\leq (f\meet c) \join (f\meet a)\\
&\leq f
\end{align*}
\noindent and symmetrically $(a\join c) \meet (b\meet d)\leq f$ as well.

The first and second paragraphs show that $\Sigma ({\bf A})$ is closed under $\twm$ and $\twj$, whereas the first and third paragraphs show that $A^{\bowtie}$ is closed under $\twm$ and $\twj$. This gives the result.
\end{proof}

For a bRS-algebra ${\bf A} = (A,\meet,\join,\to,t,f)$, define a map $\overline{\delta}_{\bf A}\colon A\times A\to A\times A$
by $$\overline{\delta}_{\bf A} \la a, b\ra = \la a,f\to b\ra=\la a,Nb\ra$$

\begin{lemma}

$\overline{\delta}_{\bf A}$ is a lattice endomorphism of $(A\times A, \twm, \twj)$.

\end{lemma}

\begin{proof}
Let $\la a, b\ra,\la c, d\ra\in A\times A$. Then by a calculation with Proposition \ref{prop2}(3) gives $\overline{\delta}_{\bf A} (\la a, b\ra\twm \la c, d\ra)=\overline{\delta}_{\bf A}\la a, b\ra\twm \overline{\delta}_{\bf A}\la c, d\ra$, and an analogous computation with Proposition \ref{prop1}(3) gives $\overline{\delta}_{\bf A} (\la a, b\ra\twj \la c, d\ra)=\overline{\delta}_{\bf A}\la a, b\ra\twj \overline{\delta}_{\bf A}\la c, d\ra$.
\end{proof}

Suppose that $\la a, b\ra\in A\times A$ satisfies $a\join b=t$. Then since $f\to b\geq b$, we have also that $a\join (f\to b) = t$. Moreover, the second coordinate of the pair $\overline{\delta}_{\bf A} \la a, b\ra = \la a,f\to b\ra = \la a,Nb\ra$ is an $N$-closed element of ${\bf A}$ since $N$ is idempotent. These considerations show that $\overline{\delta}_{\bf A} [A^{\bowtie}]\subseteq \Sigma ({\bf A})$, and we may thus define a lattice homomorphism $\delta_{\bf A}\colon (A^{\bowtie}, \twm,\twj)\to (\Sigma ({\bf A}), \twm,\twj)$ by $\delta_{\bf A} = \overline{\delta}_{\bf A}\restriction_{A^{\bowtie}}$.

\begin{lemma}\label{iso}

$\delta_{\bf A}$ is a lattice isomorphism with inverse given by
$$\delta_{\bf A}^{-1}\langle a,b\rangle = \langle a, b\meet (a\to f)\rangle.$$
\end{lemma}
\begin{proof}
To see that $\delta_{\bf A}$ is a lattice isomorphism, it suffices to show that $\delta_{\bf A}$ is a bijection. For proving $\delta_{\bf A}$ is one-to-one, suppose that $\la a,b\ra,\la c,d\ra\in A^{\bowtie}$ with $\delta_{\bf A} \la a,b\ra=\delta_{\bf A} \la c,d\ra$. Then $\la a,f\to b\ra=\la c,f\to d\ra$, so $a=c$ and $f\to b=f\to d$. Then $f\to b\leq f\to d$, so by residuation $f\meet b = f\meet (f\to b)\leq d$. Observe that since $\la a,b\ra\in A^{\bowtie}$ we have $a\meet b\leq f$ and $a\join b = t$, and by distributivity $(a\join f)\meet (b\join f) = (a\meet b)\join f = f$. Moreover, $(a\join f)\join (b\join f) = t\join f =t$. This shows that $a\join f$ and $b\join f$ are complements in the Boolean lattice $\upset f$. Since $\la a,d\ra\in A^{\bowtie}$ as well, an identical argument shows that $a\join f$ and $d\join f$ are also complements in $\upset f$. Because complements are unique in a Boolean lattice, this gives $b\join f = d\join f$. Because $b\meet f\leq d$,
\begin{align*}
b &= b\meet (b\join f)\\
&= b\meet (d\join f)\\
&= (b\meet d) \join (b\meet f)\\
&\leq (b\meet d)\join d\\
&= d
\end{align*}
\noindent so that $b\leq d$. An identical argument shows that $d\leq b$, so $b=d$. This proves $\delta$ is one-to-one.

To see that $\delta_{\bf A}$ is onto, let $\la a,b\ra\in \Sigma ({\bf A})$. Then $a\join b = t$ and $b=f\to b$. Observe that $a\meet b\meet (a\to f) = a\meet f\meet b\leq f$, and also by distributivity
\begin{align*}
a\join (b\meet (a\to f)) &= (a\join b)\meet (a\join (a\to f))\\
&= t\join t\\
&= t
\end{align*}
\noindent so $a\join (b\meet (a\to f))=t$. This gives that $\la a,b\meet (a\to f)\ra \in A^{\bowtie}$. Note also that
\begin{align*}
f\to (b\meet (a\to f)) &= (f\to b)\meet (f\to (a\to f))\\
&= (f\to b)\meet ((f\meet a)\to f))\\
&= (f\to b)\meet t\\
&= f\to b\\
&= b
\end{align*}
\noindent It follows that $\delta_{\bf A} \la a, b\meet (a\to f)\ra = \la a,b\ra$, so $\delta_{\bf A}$ is onto. The computation above also shows that the inverse of $\delta_{\bf A}$ is given by $\la a,b\ra\mapsto \la a,b\meet (a\to f)\ra$ as claimed.
\end{proof}

Owing to the fact that $(\Sigma ({\bf A}),\twm,\twj)$ is the reduct of a residuated lattice determined by the action of $S$ on ${\bf A}$, the isomorphism $\delta_{\bf A}$ allows us to endow $A^{\bowtie}$ with a residuated multiplication by transport of structure. In more detail, the proof of Lemma \ref{iso} shows that $\delta_{\bf A}$ has an inverse $\delta^{-1}_{\bf A}$ defined by 
$$\delta_{\bf A}^{-1}\langle a,b\rangle = \langle a, b\meet (a\to f)\rangle$$
\noindent Define binary operations $\ntwt$ and $\ntwi$ on $A^{\bowtie}$ by
\begin{align*}
\la a,b\ra\ntwt \la c,d\ra &= \delta_{\bf A}^{-1}(\delta_{\bf A} \la a,b\ra \twt \delta_{\bf A} \la c,d\ra)\\
\la a,b\ra \ntwi \la c,d\ra &= \delta_{\bf A}^{-1}(\delta_{\bf A} \la a,b\ra\twi \delta_{\bf A} \la c,d\ra)
\end{align*}
Written explicitly, the operation $\ntwt$ is given by $\la a,b\ra \ntwt \la c,d\ra = \la s,t\ra$, where
$$s=((a\meet f)\to d)\meet [((c\meet f)\to d)\to (a\meet c)]$$
and
$$t=((a\meet f)\to d)\meet ((c\meet f)\to d)\meet (s\to f).$$
On the other hand, the operation $\ntwi$ is given by $\la a,b\ra\ntwi\la c,d\ra=\la w,v\ra$, where
$$w=(a\to c)\meet ((f\meet d)\to b)$$
and
$$v=[(f\meet (a\to c)\meet (d\to b))\to (a\meet (f\to d))]\meet (w\to f)$$
With these operations, we immediately obtain the following.

\begin{proposition}\label{resilat}
If ${\bf A} = (A,\meet,\join,\to,t,f)$ is a bRS-algebra. Then \newline $(A^{\bowtie}, \twm, \twj, \ntwt, \ntwi, \la t,f\ra)$ is a CRL.
\end{proposition}

In fact, the CRL $(A^{\bowtie}, \twm, \twj, \ntwt, \ntwi, \la t,f\ra)$ may be enriched with a natural involution $\inv$ given by $\inv \la a,b\ra = \la b,a\ra$. Since $\la a,b\ra\in A^{\bowtie}$ obviously implies $\la b,a\ra\in A^{\bowtie}$, $\inv$ is a well-defined binary operation on $A^{\bowtie}$. We will show that the addition of $\inv$ makes $(A^{\bowtie}, \twm, \twj, \ntwt, \ntwi, \la t,f\ra)$ a Sugihara monoid. For this, we require the following lemma.

\begin{lemma}\label{lem:invtech}
$\langle a,b\rangle\in A^{\bowtie}$ implies $(a\to f)\meet (f\to b)=b$.
\end{lemma}

\begin{proof}
Let $\langle a,b\rangle\in A^{\bowtie}$. Then $a\meet b\leq f$ and $a\join b =t$. The inequality $a\meet b\leq f$ gives $b\leq a\to f$ by residuation, whence $b=b\meet (f\to b)\leq (a\to f)\meet (f\to b)$. On the other hand, Proposition \ref{prop:topjoin} together with $a\join b=t$ yields $a\to b=b$. Notice that $a\meet (a\to f)\meet (f\to b) \leq f\meet (f\to b)\leq b$, and residuation then gives $(a\to f)\meet (f\to b)\leq a\to b = b$. This proves the claim.
\end{proof}

\begin{proposition}\label{prop:inv}
Let ${\bf A}$ be an object of $\textsf{bRSA}$. Then for all $\la a,b\ra\in A^{\bowtie}$, $\neg\delta_{\bf A}\la a,b\ra = \delta_{\bf A}(\inv \la a,b\ra)$, and hence $\delta_{\bf A}$ is an isomorphism of $\textsf{SM}$.
\end{proposition}

\begin{proof}
Let $\la a,b\ra\in A^{\bowtie}$. Then $a\join b = t$ gives $a\to b = b$ and $b\to a=a$ by Proposition \ref{prop:topjoin}, and $(a\to f)\meet (f\to b) = b$ by Lemma \ref{lem:invtech}. Using these facts, observe that
\begin{align*}
\neg \delta_{\bf A} \la a,b \ra &= \neg \langle a,f\to b \ra\\
&= \la a,f\to b\rangle \twi \langle f,t \ra\\
&= \la (a\to f)\meet (t\to (f\to b)), f\to [((a\to f)\meet (t\to (f\to b))\to (a\meet t)]\ra\\
&= \la (a\to f)\meet (f\to b), f\to [((a\to f)\meet (f\to b))\to a\ra\\
&= \la b, f\to (b\to a)\ra\\
&= \la b, f\to a\ra\\
&= \delta_{\bf A}(\inv \la a,b \ra)
\end{align*}
The above shows that $\delta_{\bf A}$ preserves $\inv$ as well as the CRL operations. $\delta_{\bf A}$ is hence an isomorphism in \textsf{SM} for each object ${\bf A}$ in \textsf{bRSA}.
\end{proof}

Given a bRS-algebra ${\bf A}$, the above shows that the Sugihara monoid $S({\bf A})$ is isomorphic to $(A^{\bowtie}, \twm, \twj, \ntwt, \ntwi, \la t,f\ra,\inv)$. The involution $\inv$ is much simpler than the involution given in the definition of $S({\bf A})$, but this simplicity comes at the price of complicating the monoid operation and its residual.

Informed by these remarks, we define a functor $(-)^{\bowtie}\colon \textsf{bRSA}\to\textsf{SM}$ as follows. For an object ${\bf A} = (A,\meet,\join,\to,t,f)$ of \textsf{bRSA}, define ${\bf A}^{\bowtie}$ to be the Sugihara monoid $(A^{\bowtie}, \twm, \twj, \ntwt, \ntwi, \langle t,f \rangle, \inv)$. If $h\colon {\bf A}\to {\bf B}$ is a morphism in \textsf{bRSA}, define $h^{\bowtie}\colon {\bf A}^{\bowtie}\to {\bf B}^{\bowtie}$ by $h^{\bowtie}\langle a,b\rangle = \langle h(a), h(b)\rangle$. 

\begin{lemma}

Let $h\colon {\bf A}\to {\bf B}$ is a morphism in \textsf{bRSA}. Then $h^{\bowtie}$ is a morphism in \textsf{SM}.

\end{lemma}

\begin{proof}
Let $h\colon {\bf A}\to {\bf B}$ be a morphism in \textsf{bRSA}. From the results of \cite{GR2} it follows that the map $S(h)\colon S({\bf A})\to S({\bf B})$ defined by $S(h)\la a,b\ra = \la h(a),h(b)\ra$ is a morphism in \textsf{SM}. Observe that for any $\la a,b\ra\in A^{\bowtie}$,
\begin{align*}
S(h)(\delta_{\bf A}\la a,b\ra) &= S(h)\la a,f^{\bf A}\to b\ra\\
&= \la h(a),h(f^{\bf A}\to b)\ra\\
&= \la h(a),h(f^{\bf A})\to h(b)\ra\\
&= \la h(a),f^{\bf B}\to h(b)\ra\\
&= \delta_{\bf B} \la h(a), h(b)\ra\\
&= \delta_{\bf B} (h^{\bowtie}\la a,b\ra)
\end{align*}
\noindent It follows that $h^{\bowtie}=\delta_{\bf B}^{-1}\circ S(h)\circ \delta_{\bf A}$, hence is the composition of morphisms in \textsf{SM}.
\end{proof}

\begin{lemma}

$(-)^{\bowtie}$ is functorial.

\end{lemma}

\begin{proof}
Let $g\colon {\bf A}\to{\bf B}$ and $h\colon {\bf B}\to {\bf C}$ be morphisms in \textsf{bRSA}. Notice that the functoriality of $S$ yields
\begin{align*}
(h\circ g)^{\bowtie} &= \delta_{\bf C}^{-1}\circ S(h\circ g)\circ \delta_{\bf A}\\
&= \delta_{\bf C}^{-1}\circ S(h)\circ S(g)\circ \delta_{\bf A}\\
&= \delta_{\bf C}^{-1}\circ S(h)\circ \delta_{\bf B}\circ \delta_{\bf B}^{-1}\circ S(g)\circ \delta_{\bf A}\\
&= h^{\bowtie}\circ g^{\bowtie},
\end{align*}
\noindent and it is obvious that $(-)^{\bowtie}$ preserves the identity map.
\end{proof}

Having established the functoriality of $(-)^{\bowtie}$, it remains to show that it provides a reverse functor for $(-)_{\bowtie}\colon\textsf{SM}\to\textsf{bRSA}$.

\begin{lemma}

Let ${\bf A}$ be an object of \textsf{bRSA}. Then ${\bf A}\cong ({\bf A}^{\bowtie})_{\bowtie}$.

\end{lemma}

\begin{proof}
Observe that ${\bf A}^{\bowtie}\cong S({\bf A})$ via $\delta_{\bf A}$, and by the results of \cite{GR2}, $S({\bf A})_{\bowtie}\cong {\bf A}$. It follows that $({\bf A}^{\bowtie})_{\bowtie}\cong {\bf A}$.
\end{proof}

\begin{lemma}

Let ${\bf A}$ be an object of \textsf{SM}. Then ${\bf A}\cong ({\bf A}_{\bowtie})^{\bowtie}$.

\end{lemma}

\begin{proof}
By \cite{GR2} and $\delta_{{\bf A}_{\bowtie}}$, ${\bf A}\cong S({\bf A}_{\bowtie})\cong ({\bf A}_{\bowtie})^{\bowtie}$.
\end{proof}

\begin{lemma}

There is a bijection from $\textsf{bRSA}({\bf A},{\bf B})$ to $\textsf{SM}({\bf A}^{\bowtie},{\bf B}^{\bowtie})$.

\end{lemma}

\begin{proof}
Note that the \textsf{bRSA}-morphisms from ${\bf A}$ to ${\bf B}$ are in bijective correspondence with the \textsf{SM}-morphisms from $S({\bf A})$ to $S({\bf B})$. Moreover, given a morphism $h\colon S({\bf A})\to S({\bf B})$, the map $h\mapsto \delta_{\bf B}^{-1}\circ h\circ \delta_{\bf A}$ gives a bijection between the \textsf{SM}-morphisms from $S({\bf A})$ to $S({\bf B})$ and those from ${\bf A}^{\bowtie}$ to ${\bf B}^{\bowtie}$, which proves the result.
\end{proof}

Combining the lemmas above, we obtain

\begin{theorem}

The functors $(-)^{\bowtie}$ and $(-)_{\bowtie}$ witness the equivalence of \textsf{bRSA} and \textsf{SM}.

\end{theorem}

A consequence of the above is that $(-)^{\bowtie}$ and $S$ are both adjoints of the functor $(-)_{\bowtie}$, hence that $(-)^{\bowtie}$ and $S$ are isomorphic functors. In light of this result, we may dispense with the functor $S$ entirely, opting instead to express the equivalence in terms of the functor $(-)^{\bowtie}$ and its more familiar involution.

\begin{example}\label{ex:parex2}
Consider the Sugihara monoid ${\bf E} = (E,\meet,\join,\cdot,\to,\la 0,1\ra, \neg)$ of Example \ref{parex}. The enriched negative cone of ${\bf E}$ is given by the bRS-algebra ${\bf E}_{\bowtie}$, where $f=\neg \la 0,1\ra = \la -0,-1\ra = \la 0, -1\ra$, and has Hasse diagram
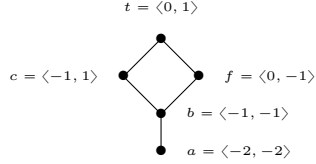
\begin{center}
\begin{tikzpicture}
    \node[label=\tiny{$t=\la 0,1\ra$}] at (-0.5,-1)  {$\bullet$};
    \node[label=right:\tiny{$f=\la 0,-1\ra$}] at (0,-1.5) {$\bullet$};
    \node[label=left:\tiny{$c=\la -1,1\ra$}] at (-1,-1.5) {$\bullet$};
    \node[label=right:\tiny{$b=\la -1,-1\ra$}] at (-0.5,-2) {$\bullet$};
    \node[label=right:\tiny{$a=\la -2,-2\ra$}] at (-0.5,-2.5) {$\bullet$};

    \draw (-0.5,-1) -- (0,-1.5);
    \draw (-0.5,-1) -- (-1,-1.5);
    \draw (-1,-1.5) -- (-0.5,-2);
    \draw (-0.5,-2) -- (0,-1.5);
    \draw (-0.5,-2) -- (-0.5,-2.5);
\end{tikzpicture}
\end{center}
The nucleus $N\colon {\bf E}_{\bowtie}\to {\bf E}_{\bowtie}$ defined by $Nx=f\to x$ is given by $Nt=Nf=t$, $Nb=Nc=c$, and $Na=a$. Therefore,
\begin{align*}
\Sigma ({\bf E}_{\bowtie}) &= \{\la x,y\ra\in E^-\times E^- : x\join y = t \text{ and } Ny=y\}\\
&= \{\la a, t\ra, \la b, t\ra, \la c, t\ra, \la f, t\ra, \la t, t\ra, \la t, a\ra, \la t,c\ra, \la f, c\ra\}
\end{align*}
On the other hand, representing ${\bf E}$ with the functor $(-)^{\bowtie}$ gives
\begin{align*}
({\bf E}_{\bowtie})^{\bowtie} &= \{\la x,y\ra\in E^-\times E^- : x\join y = t \text{ and } x\meet y\leq f\}\\
&= \{\la a, t\ra, \la t, a\ra, \la b, t\ra, \la t, b\ra, \la t, f\ra, \la f, t\ra, \la f,c\ra, \la c, f\ra\}
\end{align*}
The Hasse diagrams for $S({\bf E}_{\bowtie})$ and $({\bf E}_{\bowtie})^{\bowtie}$ are respectively
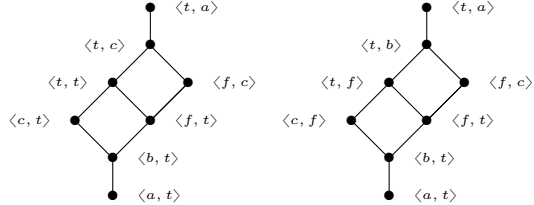
\begin{center}
\begin{tikzpicture}
    \node[label=right:\tiny{$\la t,a\ra$}] at (0,0) {$\bullet$};
    \node[label=left:\tiny{$\la t,c\ra$}] at (0,-0.5) {$\bullet$};
    \node[label=left:\tiny{$\la t,t\ra$}] at (-0.5,-1)  {$\bullet$};
    \node[label=right:\tiny{$\la f,c\ra$}] at (0.5,-1) {$\bullet$};
    \node[label=right:\tiny{$\la f,t\ra$}] at (0,-1.5) {$\bullet$};
    \node[label=left:\tiny{$\la c,t\ra$}] at (-1,-1.5) {$\bullet$};
    \node[label=right:\tiny{$\la b,t\ra$}] at (-0.5,-2) {$\bullet$};
    \node[label=right:\tiny{$\la a,t\ra$}] at (-0.5,-2.5) {$\bullet$};

    \draw (0,0) -- (0,-0.5);
    \draw (0,-0.5) -- (-0.5,-1);
    \draw (0,-0.5) -- (0.5,-1);
    \draw (-0.5,-1) -- (0,-1.5);
    \draw (0.5,-1) -- (0,-1.5);
    \draw (-0.5,-1) -- (-1,-1.5);
    \draw (-1,-1.5) -- (-0.5,-2);
    \draw (0.5,-1) -- (-0.5,-2);
    \draw (-0.5,-2) -- (-0.5,-2.5);
\end{tikzpicture}
\begin{tikzpicture}
    \node[label=right:\tiny{$\la t,a\ra$}] at (0,0) {$\bullet$};
    \node[label=left:\tiny{$\la t,b\ra$}] at (0,-0.5) {$\bullet$};
    \node[label=left:\tiny{$\la t,f\ra$}] at (-0.5,-1)  {$\bullet$};
    \node[label=right:\tiny{$\la f,c\ra$}] at (0.5,-1) {$\bullet$};
    \node[label=right:\tiny{$\la f,t\ra$}] at (0,-1.5) {$\bullet$};
    \node[label=left:\tiny{$\la c,f\ra$}] at (-1,-1.5) {$\bullet$};
    \node[label=right:\tiny{$\la b,t\ra$}] at (-0.5,-2) {$\bullet$};
    \node[label=right:\tiny{$\la a,t\ra$}] at (-0.5,-2.5) {$\bullet$};

    \draw (0,0) -- (0,-0.5);
    \draw (0,-0.5) -- (-0.5,-1);
    \draw (0,-0.5) -- (0.5,-1);
    \draw (-0.5,-1) -- (0,-1.5);
    \draw (0.5,-1) -- (0,-1.5);
    \draw (-0.5,-1) -- (-1,-1.5);
    \draw (-1,-1.5) -- (-0.5,-2);
    \draw (0.5,-1) -- (-0.5,-2);
    \draw (-0.5,-2) -- (-0.5,-2.5);
\end{tikzpicture}
\end{center}
Observe that the representations $S({\bf E}_{\bowtie})$ and $({\bf E}_{\bowtie})^{\bowtie}$ differ by only three pairs, including the monoid identity.
\end{example}

\section{Duality for algebras with a Boolean constant}\label{sec:booldual}

As an initial step to producing dualities for the Sugihara monoids and their bounded expansions, we construct dualities for the equivalent categories \textsf{bRSA} and \textsf{bGA}. Because of their close relationship to the category of Heyting algebras, dualities for \textsf{bRSA} and \textsf{bGA} may be obtained as elaborations of the well-known Esakia duality. These elaborations have much in common with Bezhanishvili and Ghilardi's duality for Heyting algebras equipped with nuclei \cite{BezGhi}, and we also explore points of contact with this duality theory. As a preliminary to obtaining dualities for \textsf{bRSA} and \textsf{bGA}, we recall some facts about the Priestley and Esakia dualities.

\subsection{Priestley and Esakia duality} A structure $(X,\leq,\tau)$ is called a \emph{Priestley space} if $(X,\leq)$ is a poset, $(X,\tau)$ is a compact topological space, and for each $x,y\in X$ satisfying $x\not\leq y$ there exists a clopen up-set $U$ with $x\in U$ and $y\notin U$. A Priestley space $(X,\leq,\tau)$ is called an \emph{Esakia space} if for each clopen set $U$ the down-set $\downset U$ is clopen as well. Given binary relational structures $(X,R_1)$ and $(Y,R_2)$, a function $\varphi\colon (X,R_1)\to (Y,R_2)$ is called a \emph{p-morphism} if it satisfies
\begin{enumerate}
\item for all $x,y\in X$, $xR_1y$ implies $\varphi(x)R_2\varphi(y)$, and
\item for all $x\in X$ and $z\in Y$, $\varphi(x)R_2z$ implies there exists $y\in X$ such that $xR_1y$ and $\varphi(y)=z$.
\end{enumerate}
If $(X,\leq_1,\tau_1)$ and $(Y,\leq_2,\tau_2)$ are Esakia spaces, then a continuous p-morphism $\varphi\colon (X,\leq_1)\to (Y,\leq_2)$ is called an \emph{Esakia map} or \emph{Esakia function}. We denote the category of Priestley spaces with continuous isotone maps by \textsf{PS}, and the category of Esakia spaces with Esakia maps by \textsf{ES}. For convenience, we denote also the category of bounded distributive lattices with bounded lattice homomorphisms by \textsf{DL}.

Given a bounded distributive lattice $\mathbb{A} = (A,\meet,\join,\bot,\top)$, we denote by $A_*$ its collection of prime filters. $A_*$ may be endowed with a topology $\tau_{\mathbb{A}}$ that is generated by the subbase $\{\sigma(a) : a\in A\}\cup\{\sigma(a)^\comp : a\in A\}$, where for each $a\in A$ we have $\sigma(a)=\{x\in A_* : a\in x\}$. Ordered by subset inclusion and equipped with this topology, $A_*$ becomes a Priestley space. We denote this Priestley space by $\mathbb{A}_*=(A_*,\subseteq,\tau_\mathbb{A})$. On the other hand, given a Priestley space $\mathbb{X}=(X,\leq,\tau)$, we denote by $X^*$ the collection of clopen up-sets of $\mathbb{X}$. This collection is closed under unions and intersections, and hence $\mathbb{X}^*=(X^*,\cap,\cup,\emptyset,X)$ is a bounded distributive lattice. The maps $(-)_*$ and $(-)^*$ may be extended to functors between \textsf{DL} and \textsf{PS} by defining their action on morphisms as follows. First, if $h\colon\mathbb{A}\to\mathbb{B}$ is a morphism of \textsf{DL}, we define $h_*\colon {\mathbb{B}}_*\to{ \mathbb{A}}_*$ by $h_*(x)=h^{-1}[x]$. Then $h_*$ is a \textsf{PS}-morphism. Likewise, if $\varphi\colon \mathbb{X}\to\mathbb{Y}$ is a morphism of \textsf{PS}, we define $\varphi^*\colon \mathbb{Y}^*\to\mathbb{X}^*$ by $\varphi^*(U)=\varphi^{-1}[U]$. Then $\varphi^*$ is a \textsf{DL}-morphism. Priestley showed in \cite{Pr1,Pr2} that the functors $(-)_*$ and $(-)^*$ witness a dual equivalence of categories between \textsf{DL} and \textsf{PS}.

A Heyting algebra ${\bf H} = (H,\meet,\join,\to,t,\bot)$ is, \emph{inter alia}, a bounded distributive lattice. Its distributive lattice reduct $\mathbb{H}$ therefore has a Priestley dual $\mathbb{H}_*$, and it turns out that $\mathbb{H}_*$ is an Esakia space. On the other hand, given an Esakia space $\mathbb{X} = (X,\leq,\tau)$, we may define a binary operation $\to$ on $X^*$ by
$$U\to V = \{x\in X : \upset x\cap U\subseteq V\}$$
The expansion $(\mathbb{X}^*,\to)$ turns out to be a Heyting algebra. Moreover, when $h$ is an \textsf{HA}-morphism, the dual $h_*$ is an Esakia map. Likewise, when $\varphi$ is an \textsf{ES}-morphism, the dual $h^*$ is a Heyting algebra homomorphism when $\to$ is defined as before. This entails that the restrictions of the functors $(-)_*$ and $(-)^*$ to \textsf{HA} and \textsf{ES} yield a dual equivalence of categories. Esakia discovered this duality independently of Priestley, and first articulated it in \cite{Esa}.

Priestley and Esakia dualities may also be formulated for algebras with a distinguished top element, but lacking a distinguished bottom element, as follows. We say that a structure $(X,\leq,\top,\tau)$ is a \emph{pointed Priestley space} if $(X,\leq,\tau)$ is a Priestley space and $\top$ is the greatest element of $(X,\leq)$, and that $(X,\leq,\top,\tau)$ is a \emph{pointed Esakia space} if it is a pointed Priestley space and $(X,\leq,\tau)$ is an Esakia space. Given pointed Priestley spaces $(X,\leq_1,\top_1,\tau_1)$ and $(Y,\leq_2,\top_2,\tau_2)$, we say that a continuous monotone map $\varphi\colon (X,\leq_1,\top_1,\tau_1)\to (Y,\leq_2,\top_2,\tau_2)$ is a \emph{pointed Priestley map} if $\varphi(\top_1)=\top_2$. We define the notion of \emph{pointed Esakia map} similarly. The category of pointed Priestley spaces with pointed Priestley maps will be denoted \textsf{pPS}, and the category of pointed Esakia spaces with pointed Esakia maps by \textsf{pES}.

Given a top-bounded distributive lattice $\mathbb{A}$ without distinguished bottom, we say that $x\subseteq A$ is a \emph{generalized prime filter} if $x$ is a prime filter or $x=A$. In this situation, we denote by $\mathbb{A}_*$ the pointed Priestley space of generalized prime filters of $\mathbb{A}$. If $\mathbb{X}$ is a pointed Priestley space, we denote by $\mathbb{X}^*$ the top-bounded distributive lattice of \emph{nonempty} clopen up-sets of $\mathbb{A}$. With these modifications, $(-)_*$ and $(-)^*$ give a dual equivalence of categories between the category of top-bounded distributive lattices and \textsf{pPS}. The same modifications witness a dual equivalence of categories between \textsf{Br} and \textsf{pES}. For a detailed treatment of the extension of the Esakia duality to Brouwerian algebras, we refer the reader to \cite{JR}.

For simplicity of notation, we will use $(-)_*$ and $(-)^*$ to denote both the functors witnessing the Priestley duality (with or without bottom elements) and their restrictions witnessing the Esakia duality (for either Heyting algebras or Brouwerian algebras). In the sequel, we will use the same notation for Urquhart's duality for relevant algebras, which is constructed based on Priestley duality. In all of these cases, we rely on context to distinguish between these meanings.

A poset $(P,\leq)$ is called a forest if $\upset x$ is a chain for each $x\in P$. It is well-known (see, e.g., \cite{CP1}) that a Heyting algebra ${\bf A}$ is a G\"odel algebra if and only if $(A_*,\subseteq)$ is a forest. For a relative Stone algebra ${\bf A}$, the addition of a new bottom element $\bot$ to ${\bf A}$ yields a G\"odel algebra with carrier $A\cup\{\bot\}$, and $(A_*,\subseteq)$ is precisely $((A\cup\{\bot\})_*, \subseteq)$. Thus a Brouwerian algebra ${\bf A}$ is a relative Stone algebra if and only if the corresponding pointed Esakia space is a forest with greatest element (i.e., a tree). The dualities discussed above may thus be restricted to obtain dualities for the G\"odel algebras (respectively, relative Stone algebras) by considering only those Esakia spaces whose underlying order is a forest (respectively, pointed Esakia spaces whose underlying order is a tree).

\subsection{Esakia duality for \textsf{bRSA} and \textsf{bGA}} We next extend the Esakia duality for Brouwerian algebras to obtain a dual equivalence of \textsf{bRSA} with the category of structured topological spaces that we define presently.

\begin{definition}\label{def:bRSspace}
A structure $(X,\leq,D,\top,\tau)$ is called a \textsf{bRS}-space if
\begin{enumerate}
\item $(X,\leq,\top,\tau)$ is a pointed Esakia space,
\item $(X,\leq)$ is a forest, and
\item $D$ is a clopen subset of $X$ consisting of designated $\leq$-minimal elements.
\end{enumerate}
Given \textsf{bRS}-spaces $(X,\leq_X,D_X,\top_X,\tau_X)$ and $(Y,\leq_Y,D_Y,\top_Y,\tau_Y)$, a map $\varphi$ from $(X,\leq_X,D_X,\top_X,\tau_X)$ to $(Y,\leq_Y,D_Y,\top_Y,\tau_Y)$ is called a \emph{\textsf{bRSS}-morphism} if
\begin{enumerate}
\item $\varphi$ is a pointed Esakia map from $(X,\leq_X,\top_X,\tau_X)$ to $(Y,\leq_Y,\top_Y,\tau_Y)$,
\item $\varphi[D_X]\subseteq D_Y$, and
\item $\varphi[D_X^\comp]\subseteq D_Y^\comp$.
\end{enumerate}
We denote the category of \textsf{bRS}-spaces with \textsf{bRSS}-morphisms by \textsf{bRSS}.
\end{definition}

The equivalence of \textsf{bRSA} and \textsf{bRSS} is witnessed by augmented versions of the functors $(-)_*$ and $(-)^*$. For an object ${\bf A} = (A,\meet,\join,\to,t,f)$ of \textsf{bRSA}, define ${\bf A}_* = ((A,\meet,\join,\to,t)_*, \sigma(f)^\comp)$. For an object $(X,\leq,D,\top,\tau)$ of \textsf{bRSS}, define $(X,\leq,D,\top,\tau)^* = ((X,\leq,\top,\tau)^*,D^\comp)$. $(-)_*$ and $(-)^*$ are defined for morphisms exactly as in the duality for Brouwerian algebras.

\begin{lemma}\label{lem:dual1}
Let ${\bf A} = (A,\meet,\join,\to,t,f)$ be an object of \textsf{bRSA}. Then ${\bf A}_*$ is an object of $\textsf{bRSS}$.
\end{lemma}

\begin{proof}
${\bf A}_*$ is a pointed Esakia space whose underlying order is a forest by the duality for Brouwerian algebras as applied to relative Stone algebras. It thus suffices to show that $\sigma(f)^\comp$ is a clopen subset of $A_*$ consisting of $\subseteq$-minimal elements. That $\sigma(f)^\comp$ is clopen follows as it is a basic clopen set. To see that $\sigma(f)^\comp$ consists of minimal elements, let $y\in \sigma(f)^\comp$ and suppose that $x\in A_*$ with $x\subseteq y$. Let $a\in y$. Then $(a\to f)\join a = t\in x$, so by the primality of $x$ either $a\in x$ or $a\to f\in x$. If $a\to f\in x$, then $a\to f\in y$. This gives $a\meet (a\to f)\in y$. But $a\meet (a\to f)\leq f$ and $y$ upward-closed gives $f\in y$, which is a contradiction to the choice of $y$. It follows that $a\in x$, so that $y\subseteq x$. Since $x\subseteq y$ as well, this shows that $x=y$ and thus $y$ is $\subseteq$-minimal.
\end{proof}

\begin{lemma}\label{lem:dual2}
Let ${\bf X} = (X,\leq,D,\top,\tau)$ be an object of $\textsf{bRSS}$. Then ${\bf X}^*$ is an object of \textsf{bRSA}.
\end{lemma}

\begin{proof}
${\bf X}^*$ is a relative Stone algebra by the duality for Brouwerian algebras, so we need only show that $D^\comp$ is a clopen up-set of ${\bf X}$ and that for any clopen up-set $U\subseteq X$, $U\cup (U\to D^\comp)=X$. $D$ clopen immediately yields that $D^\comp$ is clopen. To see that $D^\comp$ is an up-set, let $x\in D^\comp$ and $y\in X$ with $x\leq y$. If $y\in D$ held, then the minimality of the elements of $D$ would give $x=y$ and hence $x\in D$, a contradiction. Therefore $y\in D^\comp$, so $D^\comp$ is an up-set.

Now let $U\subseteq X$ be a clopen up-set and let $x\in X$. If $x\notin U$, then we claim that $x\in U\to D^\comp=\{y\in X : \upset y\cap U\subseteq D^\comp\}$, so suppose that $y\in \upset x\cap U$. It suffices to show that $y$ is not minimal. Observe that $x\leq y$ and $y\in U$, so $x\notin U$ gives $x\neq y$. Thus $y$ is not $\leq$-minimal, which gives $x\in U\to D^\comp$. It follows that $x\in U\cup (U\to D^\comp)$, so that $U\cup (U\to D^\comp)=X$, proving the claim.
\end{proof}

\begin{lemma}\label{lem:dual3}
Let $h\colon {\bf A} \to {\bf B}$ be a morphism of \textsf{bRSA}. Then $h_*\colon {\bf B}_*\to {\bf A}_*$ is a morphism of $\textsf{bRSS}$.
\end{lemma}

\begin{proof}
The duality for Brouwerian algebras gives that $h_*$ is a morphism of \textsf{pES}. We must show that $h_*[\sigma(f^{\bf B})]\subseteq \sigma(f^{\bf A})$ and $h_*[\sigma(f^{\bf B})^\comp]\subseteq \sigma(f^{\bf A})^\comp$.

Firstly, let $x\in h_*[\sigma(f^{\bf B})]$. Then there exists $y\in\sigma(f^{\bf B})$ such that $x=h_*(y)$. Since $h(f^{\bf A})=f^{\bf B}\in y$, it follows that $f^{\bf A}\in h^{-1}[y]=h_*(y)=x$, so $x\in\sigma(f^{\bf A})$. This gives $h_*[\sigma(f^{\bf B})]\subseteq \sigma(f^{\bf A})$.

Secondly, let $x\in h_*[\sigma(f^{\bf B})^\comp]$. Then there exists $y\in \sigma(f^{\bf B})^\comp$ such that we have $x=h_*(y)=h^{-1}[y]$. Were if the case that $f^{\bf A}\in x$, then $f^{\bf B}=h(f^{\bf A})$ would give that $f^{\bf B}\in y$, contradicting $y\notin\sigma(f^{\bf B})$. Thus $f^{\bf A}\notin x$, and it follows that $h_*[\sigma(f^{\bf B})^\comp]\subseteq \sigma(f^{\bf A})^\comp$.
\end{proof}

\begin{lemma}\label{lem:dual4}
Let $\varphi\colon {\bf X} \to {\bf Y}$ be a morphism of $\textsf{bRSS}$. Then $\varphi^*\colon {\bf Y}^*\to {\bf X}^*$ is a morphism of \textsf{bRSA}.
\end{lemma}

\begin{proof}
$\varphi^*$ is a morphism of \textsf{Br} by the duality for Brouwerian algebras. We must show $\varphi^*(D_Y^\comp)=D_X^\comp$.

Since $\varphi$ is a \textsf{bRSS}-morphism, it follows that $\varphi[D_X]\subseteq D_Y$ and $\varphi[D_X^\comp]\subseteq D_Y^\comp$. From the latter, it follows that $D_X^\comp\subseteq \varphi^{-1}[\varphi[D_X^\comp]]\subseteq\varphi^{-1}[D_Y^\comp]$, so we have $D_X^\comp\subseteq \varphi^*(D_Y^\comp)$.

On the other hand, $D_X\subseteq \varphi^{-1}(\varphi[D_X])\subseteq \varphi^{-1}[D_Y]$ follows from the other condition, so by taking complements
$$D_X^\comp\supseteq X\setminus\varphi^{-1}[D_Y]=\varphi^{-1}[Y]\setminus\varphi^{-1}[D_Y]=\varphi^{-1}[D_Y^\comp]=\varphi^*(D_Y^\comp).$$
The result follows.
\end{proof}

\begin{lemma}\label{lem:dual5}
Let ${\bf A}$ be an object of \textsf{bRSA}. Then $({\bf A}_*)^*\cong {\bf A}$.
\end{lemma}

\begin{proof}
By the Esakia duality for relative Stone algebras, $\sigma\colon A\to (A_*)^*$ is an isomorphism between the $(\meet,\join,\to,t)$-reducts of $A$ and $(A_*)^*$. It thus suffices to show that this map preserves the constant $f$. Thus the result follows from observing that $f^{({\bf A}_*)^*} = A_*\setminus (\sigma(f^{\bf A})^\comp) = \sigma(f^{\bf A})$.
\end{proof}

\begin{lemma}\label{lem:bRSSiso}
Let ${\bf X}$ and ${\bf Y}$ be objects of \textsf{bRSS}, and let $\varphi\colon {\bf X}\to {\bf Y}$ be a \textsf{pES}-isomorphism. Then $\varphi$ is an isomorphism of \textsf{bRSS} if and only if $\varphi[D_X]=D_Y$.
\end{lemma}

\begin{proof}
Suppose first that $\varphi$ is an isomorphism of \textsf{bRSS}. Then $\varphi$ has an inverse morphism in \textsf{bRSS}. Among other things, that $\varphi$ is an isomorphism in \textsf{pES} entails that $\varphi$ is an isomorphism of posets and hence a bijection. Moreover, $\varphi[D_X]\subseteq D_Y$ and $\varphi[D_X^\comp]\subseteq D_Y^\comp$ hold by definition. Since $\varphi$ is a bijection, taking complements in the latter inclusion gives $D_Y\subseteq\varphi[D_X^\comp]^\comp=\varphi[D_X]$, and thus $\varphi[D_X]=D_Y$.

For the converse, suppose that $\varphi[D_X]=D_Y$. Since $\varphi$ is an isomorphism of \textsf{pES}, $\varphi$ is a bijection and its set-theoretic $\varphi^{-1}$ inverse corresponds with its inverse in \textsf{pES}. The fact that $\varphi$ is a bijection gives $\varphi[D_X^\comp]=\varphi[D_X]^\comp=D_Y^\comp$, and this implies that $\varphi$ is a morphism in $\textsf{bRSS}$. On the other hand, $\varphi[D_X]=[D_Y]$ implies $\varphi^{-1}[D_Y]=D_X$ and $\varphi[D_X^\comp]=D_Y^\comp$ implies $\varphi^{-1}[D_Y^\comp]=D_X^\comp$, so $\varphi^{-1}$ is a morphism in \textsf{bRSS} as well. This gives that $\varphi$ is an isomorphism in \textsf{bRSS} and the claim is proven.
\end{proof}

\begin{lemma}\label{lem:dual6}
Let ${\bf X}$ be an object of \textsf{bRSS}. Then $({\bf X}^*)_*\cong {\bf X}$.
\end{lemma}

\begin{proof}

Let $\varphi\colon X\to (X^*)_*$ be defined by $\varphi(x) = \{U\in X^* : x\in U\}$. The Esakia duality for relative Stone algebras gives that $\varphi$ is an isomorphism of \textsf{pES}. We will show that $\varphi$ is also an isomorphism of \textsf{bRSS}, and it suffices to show that $\varphi[D]=\sigma(D^\comp)^\comp=\{p\in X^* : D^\comp\notin p\}$ by Lemma \ref{lem:bRSSiso}.

Suppose first that $p\in\varphi[D]$. Then there exists $x\in D$ such that $p=\varphi(x)$, i.e., $p=\{U\in X^* : x\in U\}$. Since $x\notin D^\comp$, we have $D^\comp\notin p$. Thus $p\in\sigma(D^\comp)^\comp$ and $\varphi[D]\subseteq\sigma(D^\comp)^\comp$.

For the reverse inclusion, let $p\in\sigma(D^\comp)^\comp$. Then $D^\comp\notin p$. If there were $x\in D^\comp$ with $\varphi(x)=p$, it would follow that $D^\comp\in \{U\in X^* : x\in U\} =\varphi(x) = p$, a contradiction. Therefore $p\notin\varphi[D^\comp]$. Since $\varphi$ is a bijection $\varphi[D^\comp]=\varphi[D]^\comp$, so $p\notin \varphi[D]^\comp$. This implies $p\in\varphi[D]$, whence $\sigma(D^\comp)^\comp\subseteq\varphi[D]$. It follows that $\varphi[D]=\sigma(D^\comp)^\comp$, proving the claim.
\end{proof}

\begin{theorem}\label{thm:bRSAdual}
\textsf{bRSA} is dually equivalent to \textsf{bRSS}.
\end{theorem}

\begin{proof}
This follows immediately from Lemmas \ref{lem:dual1}, \ref{lem:dual2}, \ref{lem:dual3}, \ref{lem:dual4}, \ref{lem:dual5}, and \ref{lem:dual6}. Naturality follows from the proof that the functors $(-)_*$ and $(-)^*$ give an equivalence between \textsf{pES} and the \textsf{Br}.
\end{proof}

The duality exhibited above may be easily extended to provide a duality for bG-algebras as well. This extension amounts to dropping the top element from the language of \text{bRSS}.
\begin{definition}\label{def:BGspace}
A structure $(X,\leq,D,\tau)$ is called a \textsf{bG}-space if
\begin{enumerate}
\item $(X,\leq,\tau)$ is an Esakia space,
\item $(X,\leq)$ is a forest, and
\item $D$ is a clopen subset of $X$ consisting of $\leq$-minimal elements.
\end{enumerate}
Given \textsf{bG}-spaces ${\bf X} = (X,\leq_X,D_X,\tau_X)$ and ${\bf Y} = (Y,\leq_Y,D_Y,\tau_Y)$, a map $\varphi$ from ${\bf X}$ to ${\bf Y}$ is called a \emph{\textsf{bGS}-morphism} if
\begin{enumerate}
\item $\varphi$ is an Esakia map from $(X,\leq_X,\tau_X)$ to $(Y,\leq_Y,\tau_Y)$,
\item $\varphi[D_X]\subseteq D_Y$, and
\item $\varphi[D_X^\comp]\subseteq D_Y^\comp$.
\end{enumerate}
We denote the category of \textsf{bG}-spaces with \textsf{bGS}-morphisms by \textsf{bGS}.
\end{definition}

\begin{theorem}\label{thm:bGAdual}
\textsf{bGA} is dually equivalent to \textsf{bGS}.
\end{theorem}

\begin{proof}
This follows as in the proof of Theorem \ref{thm:bRSAdual}, replacing any mention of the Esakia duality for relative Stone algebras in the proofs of the relevant lemmas by the Esakia duality for G\"odel algebras.
\end{proof}

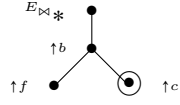
\begin{figure}
\begin{center}
\begin{tikzpicture}
    \node[label=left:\tiny{${E_{\bowtie}}_*$}] at (0,0) {$\bullet$};
    \node[label=left:\tiny{$\upset b$}] at (0,-0.5) {$\bullet$};
    \node[label=left:\tiny{$\upset f$}] at (-0.5,-1)  {$\bullet$};
    \node[label=right:\tiny{$\upset c$}] at (0.5,-1) {\textcircled{$\bullet$}};

    \draw (0,0) -- (0,-0.5);
    \draw (0,-0.5) -- (-0.5,-1);
    \draw (0,-0.5) -- (0.5,-1);
\end{tikzpicture}
\end{center}
\caption{Hasse diagram for $({\bf E}_{\bowtie})_*$}
\label{fig:HasseEbrsdual}
\end{figure}

\begin{example}\label{ex:parex3}
The \textsf{bRS}-algebra ${\bf E}_{\bowtie}$ of Example \ref{ex:parex2} has dual space $({{\bf E}_{\bowtie}})_*$, whose Hasse diagram is given in Figure \ref{fig:HasseEbrsdual}. The elements of the designated subset are circled.
\end{example}

\subsection{bG-algebras as Heyting algebras with nuclei}

\textsf{bG}-algebras were originally formulated in \cite{GR2} in the guise of G\"odel algebras equipped nuclei, and the duality articulated here was originally discovered in the setting of Bezhanishvili and Ghilardi's duality for Heyting algebras equipped with nuclei \cite{BezGhi}. The nucleus of a \textsf{bG}-algebra is definable from the designated constant $f$ via the term $Na=f\to a$, but it is natural to ask how the duality presented here compares with that of Bezhanishvili and Ghilardi. We will see that the nucleus of a \textsf{bG}-algebra presents itself in a particularly simple and pleasant fashion on the dual space, and provides a useful perspective for thinking about \textsf{bG}-spaces.

\begin{definition}
An algebra ${\bf A}=(A,\meet,\join,\to,t,\bot,N)$ is called a \emph{nuclear Heyting algebra} if $(A,\meet,\join,\to,t,\bot)$ is a Heyting algebra and $N$ is nucleus on $(A,\meet,\join,\to,t,\bot)$. The category of nuclear Heyting algebras with Heyting algebra homomorphisms that preserve the nucleus is denoted \textsf{nHA}. 
\end{definition}
\begin{definition}\label{def:nES}
A structure $(X,\leq,R,\tau)$ is called a \emph{nuclear Esakia space} if $(X,\leq,\tau)$ is an Esakia space, and $R$ is a binary relation on $X$ satisfying
\begin{enumerate}
\item $xRz$ if and only if $(\exists y\in X)(yRy \text{ and } x\leq y\leq z)$,
\item $R[x]$ is closed for each $x\in X$, and
\item whenever $A\subseteq X$ is clopen, so is $R^{-1}[A]$.
\end{enumerate}
The category of nuclear Esakia spaces with morphisms the continuous $p$-morphisms with respect to both $\leq$ and $R$ is denoted \textsf{nES}.
\end{definition}
If ${\bf A} = (A,\meet,\join,\to,\bot,N)$ is a nuclear Heyting algebra, then define the dual ${\bf A}_* = ((A,\meet,\join,\to,t,\bot)_*, R_{\bf A})$, where $R_{\bf A}$ is the binary relation on $A_*$ defined by $xR_{\bf A}y$ if and only if $N^{-1}[x]\subseteq y$. On the other hand, for a nuclear Esakia space ${\bf X} = (X,\leq,R,\tau)$, define ${\bf X}^*=((X,\leq,\tau)^*, N_{\bf X})$, where $N_{\bf X}\colon X^*\to X^*$ is defined by $N_{\bf X}(U) = X\setminus R^{-1}[X\setminus U]$. For morphisms of \textsf{nHA} and \textsf{nES}, define $(-)_*$ and $(-)^*$ as usual. With these definitions, we have
\begin{theorem}[{{\cite[Theorem 14]{BezGhi}}}]\label{thm:nucdual}
$(-)_*$ and $(-)^*$ witness a dual equivalence of categories between \textsf{nHA} and \textsf{nES}.
\end{theorem}

For ${\bf A}=(A,\meet,\join,\to,t,\bot,f)$ a \textsf{bG}-algebra, define $N_{\bf A}\colon A\to A$ to be the nucleus given by $N_{\bf A}(a)=f\to a$. Then $(A,\meet,\join,\to,t,\bot,N_{\bf A})$ is nuclear Heyting algebra, and we aim to characterize the relation $R_{\bf A}$ on $A_*$ associated with this algebra. Toward this end, for $x\in A_*$ define $x^{-1}=N_{\bf A}^{-1}[x]$. In this terminology, for $x,y\in A_*$, $xR_{\bf A}y$ if and only if $x^{-1}\subseteq y$. We prove several technical lemmas about the operator $(-)^{-1}$.

\begin{lemma}\label{lem:prime}
Let ${\bf A} = (A,\meet,\join,\to,t,\bot,f)$ be a \textsf{bG}-algebra, and let $x\in A_*$. Then $x^{-1}\in A_*\cup\{A\}$.
\end{lemma}

\begin{proof}
Note that by Proposition \ref{prop2}, we have that $N_{\bf A}(a\meet b)=N_{\bf A}(a)\meet N_{\bf A}(b)$ and $N_{\bf A}(a\join b)=N_{\bf A}(a)\join N_{\bf A}(b)$ for all $a,b\in A$. Let $x\in A_*$. If $a,b\in x^{-1}$, then $N_{\bf A}(a),N_{\bf A}(b)\in x$ and so $N_{\bf A}(a\meet b)=N_{\bf A}(a)\meet N_{\bf A}(b)\in x$ since $x$ is a filter. This gives $a\meet b\in x^{-1}$. Moreover, if $a\in x^{-1}$ and $a\leq b\in A$, then we have that $N_{\bf A}(a)\in x$ and $N_{\bf A}(a)\leq N_{\bf A}(b)$ by the isotonicity of $N_{\bf A}$. Since $x$ is upward-closed, $N_{\bf A}(b)\in x$ and $b\in x^{-1}$. It follows that $x^{-1}$ is a filter. To see that $x^{-1}$ is either prime or improper, let $a\join b\in x^{-1}$. Then $N_{\bf A}(a)\join N_{\bf A}(b)=N_{\bf A}(a\join b)\in x$, so since $x$ is prime we have $N_{\bf A}(a)\in x$ or $N_{\bf A}(b)\in x$. It follows that $a\in x^{-1}$ or $b\in x^{-1}$.
\end{proof}
\begin{remark}
Lemma 11 of \cite{BezGhi} shows that $(-)^{-1}$ is a closure operator on the lattice of filters of ${\bf A}$, and combined with the previous lemma this shows that $(-)^{-1}$ is a closure operator on the poset $A_*\cup\{A\}$.
\end{remark}
\begin{lemma}\label{lem:tech}
Let ${\bf A} = (A,\meet,\join,\to,t,\bot,f)$ be a \textsf{bG}-algebra. Then for each $x,y\in A_*$,
\begin{enumerate}
\item If $x^{-1}\in A_*$, then $x^{-1}$ is the least $R_{\bf A}$-successor of $x$,
\item $xR_{\bf A}x$ iff $f\in x$,
\item If $x$ is an $R_{\bf A}$-successor, then $xR_{\bf A}x$,
\item If $x\subset y$, then $xR_{\bf A}y$.
\end{enumerate}
\end{lemma}

\begin{proof}
For (1), suppose $x^{-1}\in A_*$. Since $x^{-1}\subseteq x^{-1}$, $xR_Ax^{-1}$ trivially holds. Now suppose that $s\in A_*$ is an $R_A$-successor of $x$. Then $x^{-1}\subseteq s$ by definition, so $x^{-1}$ is the least $R_A$-successor of $x$.

For (2), note that the identity $N_{\bf A}(N_{\bf A}(a)\to a)=t$ together with $N_{\bf A}(a) = t$ if and only if $f\leq a$ implies $f\leq N_{\bf A}(a)\to a$ for all $a\in A$, whence by residuation $f\meet N_{\bf A}(a)\leq a$ for all $a\in A$. If $x$ is a filter and $f\in x$, then $a\in x^{-1}$ implies $N_{\bf A}(a)\in x$, and since $x$ is a filter we have that $f\meet N_{\bf A}(a)\in x$ also. Because $x$ is upward-closed, $f\meet N_{\bf A}(a)\leq a$ yields $a\in x$. Thus $x^{-1}\subseteq x$, which gives $xR_{\bf A}x$. Conversely, if $xR_{\bf A}x$ then $x$ is an $R_{\bf A}$-successor of $x$. Since $x^{-1}$ is the least $R_{\bf A}$-successor of $x$, this gives $x^{-1}\subseteq x$. But $N_{\bf A}(f)=t\in x$, so $f\in x^{-1}$ and hence $f\in x$.

For (3), suppose there exists $p$ with $pR_{\bf A}x$. Then $p^{-1}\subseteq x$. Since $p^{-1}R_Ap^{-1}$, part (b) gives $f\in p^{-1}$ and hence $f\in x$. Therefore $xR_{\bf A}x$ by part (b).

For (4), let $y\in A_*$ with $x\subset y$. Because this containment is proper, there exists $a\in y\setminus x$. By definition $a\join (a\to f)=t$, so since $a\join (a\to f)\in x$ and since $x$ is prime with $a\notin x$ we have $a\to f\in x$. This implies that $a,a\to f\in y$, whence $a\meet (a\to f)\in y$ since $y$ is a filter. But $a\meet (a\to f)\leq f$ so since $y$ is upward-closed we have $f\in y$. It follows from (2) that $yR_{\bf A}y$. Thus $y^{-1} \subseteq y$. Since $y\subseteq y^{-1}$ always, we have $y^{-1}=y$. Since $x\subseteq y$, the isotonicity of $(-)^{-1}$ gives $x^{-1}\subseteq y^{-1}=y$, so $xR_{\bf A} y$ as desired.
\end{proof}

Given an object ${\bf A}$ of \textsf{bGA}, Lemma \ref{lem:tech}(4) gives that the only points of ${\bf A}_*$ that are not $R_{\bf A}$-reflexive are minimal. Definition \ref{def:nES}(1) makes it clear that the accessibility relation of a nuclear Esakia space is determined by the order together with the non-reflexive points, which motivates the following. For a \textsf{bG}-space ${\bf X} = (X,\leq,D,\tau)$, define a binary relation $\lesx$ on $X$ by
$$\lesx = \leq\setminus\{\la x,x\ra\in X\times X : x\in D\}.$$

\begin{proposition}\label{prop:accesschar}
Let ${\bf A} = (A,\meet,\join,\to,t,\bot,f)$ be a \textsf{bG}-algebra. Then $R_{\bf A}$ coincides with $\lesa$.
\end{proposition}

\begin{proof}
Suppose first that $xR_{\bf A}y$. Then by Lemma \ref{lem:tech}(3) it follows that $yR_{\bf A}y$, and by Lemma \ref{lem:tech}(2) it follows that $f\in y$. Thus $y\in\sigma(f)$, and hence we have that $\la x,y\ra\notin\{\la z,z\ra\in A_*\times A_* : z\in \sigma(f)^\comp\}$. Because $x\subseteq y$ as a consequence of $xR_{\bf A}y$, this yields $x\lesa y$.

On the other hand, suppose that $x\lesa y$. Then $x\subseteq y$, and $\la x,y\ra$ is not in $\{\la z,z\ra : z\in \sigma(f)^\comp\}$. There are two possibilities. First, if $x\neq y$, then by Lemma \ref{lem:tech}(4) we have $xR_{\bf A}y$. Second, if $x=y\notin \sigma(f)^\comp$, then $y\in\sigma(f)$. This gives $f\in y$, and Lemma \ref{lem:tech}(2) gives $yR_{\bf A}y$. But since $x=y$, this gives $xR_{\bf A}y$. It follows that $x\lesa y$ if and only if $xR_{\bf A} y$ as desired.
\end{proof}

Proposition \ref{prop:accesschar} completely characterizes the accessibility relation arising from the nucleus $N_{\bf A}$ for a \textsf{bG}-algebra ${\bf A}$. The fact that $R_{\bf A}$ is definable in terms of the order relation $\subseteq$ and the designated subset $\sigma(f)^\comp$ reflects the fact that $N_{\bf A}$ is term-definable in the underlying \textsf{bG}-algebra. The following further underscores this fact.

\begin{proposition}\label{prop:image}
Let $(X,\leq,D,\tau)$ be a \textsf{bG}-space. Then the image of $X$ under $\lesx$ coincides with $D^\comp$.
\end{proposition}

\begin{proof}
Let $y\in\lesx[X]$. Then there exists $x\in X$ with $x\lesx y$. Then $x\leq y$, and either $x\neq y$ or $x=y\notin D$. In the first case, $y$ is not $\leq$-minimal and hence $y\notin D$. In the second case, $y\notin D$ by hypothesis. Hence $y\notin D$ and $\lesx[X]\subseteq D^\comp$.

For the reverse inclusion, let $y\in D^\comp$. Then we have that $y\leq y$, and additionally $\la y,y\ra\notin\{\la x,x\ra : x\in D\}$, so $y\lesx y$. Therefore $y\in\les[X]$ and $D^\comp\subseteq \les[X]$. Equality follows.
\end{proof}

Propositions \ref{prop:accesschar} and \ref{prop:image} allow us to understand the duality articulated here for \textsf{bGA} in the context of the Bezhanishvili-Ghilardi duality for nuclear Heyting algebras, at least on the level of objects. The condition that \textsf{bGA}-morphisms preserve the constant $f$ turns out to be more demanding than merely asking that morphisms commute with the nucleus $Na=f\to a$, so not all \textsf{nES}-morphisms between objects of \textsf{bGS} are \textsf{bGS}-morphisms. However, we obtain the appropriate morphisms if we only consider those \textsf{nES}-morphisms that preserve the designated set $D$.

\begin{proposition}
Let $(X,\leq_X,D_X,\tau_X)$ and $(Y,\leq_Y,D_Y,\tau_Y)$ be \textsf{bG}-spaces and let $\varphi\colon X\to Y$ be a \textsf{bGS}-morphism. Then $\varphi$ is a $p$-morphism with respect to $\les$.
\end{proposition}

\begin{proof}
Suppose first that $\varphi$ is a \textsf{bGS}-morphism. Then $\varphi$ is an Esakia map by definition. We first show that $\varphi$ preserves $\les$. Let $x,y\in X$ with $x\lesx y$. Then $x\leq_X y$, so as $\varphi$ preserves $\leq$ it follows that $\varphi(x)\leq_Y\varphi(y)$. Since we have $\la x, y\ra\notin \{\la z, z\ra : z\in D\}$, either $x\neq y$ or $x=y\notin D$. In the former case, $y\notin D_X$ since $y$ is not minimal, so as $\varphi[D_X^\comp]\subseteq D_Y^\comp$ it follows that $\varphi(y)\notin D_Y$. On the other hand, if $x=y\notin D_X$, then $\varphi(y)\notin D_Y$ as well. In either case, this yields that $\la\varphi(x),\varphi(y)\ra\notin \{\la z,z\ra : z\in D_Y\}$, so $\varphi(x)\lesy\varphi(y)$.

Next, suppose that $x\in X$,  $z\in Y$ with $\varphi(x)\lesy z$. Then we have that $\la\varphi(x),z\ra\notin\{(w,w) : w\in D_Y\}$, so either $\varphi(x)\neq z$ or $\varphi(x)=z\notin D_Y$. In the former case, note that $\varphi(x)\lesy z$ gives $\varphi(x)\leq_Y z$, so since $\varphi$ is an Esakia map we have that there exists $y\in X$ with $x\leq y$ and $\varphi(y)=z$. Since $\varphi(x)\neq z = \varphi(y)$, we have $x\neq y$. Together with $x\leq y$, this gives that $y$ is not minimal, and hence $y\notin D_X$. Thus $x\lesx y$ and $\varphi(y)=z$, which gives that $\varphi$ is a $p$-morphism with respect to $\les$.
\end{proof}

\begin{proposition}
Let $(X,\leq_X,D_X,\tau_X)$ and $(Y,\leq_Y,D_Y,\tau_Y)$ be \textsf{bG}-spaces and let $\varphi\colon X\to Y$ be an Esakia map that is a $p$-morphism with respect to $\les$. Then if $\varphi[D_X]\subseteq D_Y$, $\varphi$ is a \textsf{bG}-morphism.
\end{proposition}

\begin{proof}
It suffices to show that $\varphi[D_X^\comp]\subseteq D_Y^\comp$, so let $y\in\varphi[D_X^\comp]$. Then there exists $x\in D_X^\comp$ such that $\varphi(x)=y$. Since $x\in D_X^\comp$ we have that $x\lesx x$, so $\varphi(x)\lesy\varphi(x)$. Thus $\varphi(x)\lesy y$, which entails that $y\in \lesy[Y] = D_Y^\comp$ as desired.
\end{proof}

\section{Natural dualities and the Davey-Werner duality}\label{sec:natdual}

Because \textsf{SM} is equivalent to \textsf{bRSA}, the duality presented in Section \ref{sec:booldual} also provides a dual equivalence between \textsf{SM} and \textsf{bRSS}. As presented so far, this dual equivalence involves passing between a Sugihara monoid and its dual through the enriched negative cone. We will recast the duality of Section \ref{sec:booldual} in terms more native to the Sugihara monoids by identifying appropriate duals for their $(\meet,\join,\neg)$-reducts. This presentation of the duality rests on the Davey-Werner natural duality for Kleene algebras \cite{DavWer} in much the same way that Esakia duality rests on Priestley duality. Because the fuctor $S$ of \cite{GR2} presents the involution of a Sugihara monoid in a way inextricably linked to the residual operation, it is inadequate for connecting the duality of Section \ref{sec:booldual} to the Davey-Werner duality. However, the simplified presentation of the involution obtained in the algebraic work of Section \ref{sec:twist} reveals the relationship between the duality of the previous section and the Davey-Werner duality. The preliminary work of Section \ref{sec:twist} thus provides an essential ingredient in obtaining the duality for Sugihara monoids. To explicate the duality in full generality, we first develop an analogue of the Davey-Werner duality for algebras without lattice bounds. This treatment requires the review of some basic natural duality theory. Due to the vastness of the subject, our review of natural duality theory is necessarily perfunctory. We draw all background material on natural dualities from \cite{CD}, and refer the reader there for a more thorough exposition.

\subsection{Natural dualities in general}

Let $\underline{{\bf M}}$ be a finite algebra and $\mathbb{ISP}(\underline{{\bf M}})$ be the prevariety it generates. We denote by $\mathcal{A}$ the category whose objects are algebras in $\mathbb{ISP}(\underline{{\bf M}})$ and whose morphisms are algebraic homomorphisms between members of $\mathbb{ISP}(\underline{{\bf M}})$. Consider a structure $\utilde{\bf M} = (M, G, H, R, \tau)$ defined on the same underlying set $M$ as $\underline{\bf M}$, where $G$ is a set of total operations on $M$, $H$ is a set of partial operations on $M$, $R$ is a set of relations on $M$, and $\tau$ is the discrete topology on $M$. We say that $\utilde{\bf M}$ is \emph{algebraic} over $\underline{\bf M}$ if the graph of each total operation in $G$, the graph of each partial operation in $H$, and each relation in $R$ is a subalgebra of the appropriate power of $\underline{\bf M}$. In this situation, there is always an adjunction between $\mathcal{A}$ and the category of $\mathcal{X}$ defined presently. The objects of $\mathcal{X}$ are the enriched topological spaces in $\mathbb{IS}_c\mathbb{P}^+(\utilde{\bf M})$, i.e., isomorphic copies of topologically closed substructures of powers of $\utilde{\bf M}$ (excluding $\utilde{\bf M}^\emptyset$). The morphisms of $\mathcal{X}$ are continuous homomorphisms between such structures. The adjunction between $\mathcal{A}$ and $\mathcal{X}$ is given by hom-functors $E\colon \mathcal{X}\to\mathcal{A}$ and $D\colon\mathcal{A}\to\mathcal{X}$ whose action on objects is defined by
$$E({\bf X}) = \mathcal{X}({\bf X}, \utilde{\bf M})$$
$$D({\bf A}) = \mathcal{A}({\bf A}, \underline{\bf M}),$$
where $\mathcal{X}({\bf X}, \utilde{\bf M})$ is viewed as an object of $\mathcal{A}$ by inheriting structure pointwise from $\underline{\bf M}$, and likewise $\mathcal{A}({\bf A}, \underline{\bf M})$ is viewed as an object of $\mathcal{X}$ by inheriting structure pointwise from $\utilde{\bf M}$. The action of $E$ and $D$ on morphisms is defined by precomposition, i.e., for $h\colon {\bf A}\to {\bf B}$ a morphism of $\mathcal{A}$ and $\varphi\colon {\bf X}\to {\bf Y}$ a morphism of $\mathcal{X}$, we define $D(h)\colon D({\bf B})\to D({\bf A})$ and $E(\varphi)\colon E({\bf Y})\to E({\bf X})$ by
$$D(h)(x)= x\circ h$$
$$E(\varphi)(\alpha) = \alpha \circ \varphi,$$
respectively. The unit of this adjunction is the natural transformation $e$ given by evaluation, i.e., for objects ${\bf A}$ of $\mathcal{A}$, $e_{\bf A}\colon {\bf A}\to ED({\bf A})$ is defined for $a\in A$ by $e_{\bf A}(a)(x) = x(a)$. The counit is likewise defined for objects ${\bf X}$ of $\mathcal{X}$ by $\epsilon_{\bf X}\colon {\bf X}\to DE({\bf X})$ given by $\epsilon_{\bf X}(x)(\alpha)=\alpha(x)$. With the above set-up, whenever each homomorphism $e_{\bf A}$ is an isomorphism, we say that the dual adjunction $(D, E, e, \epsilon)$ is a \emph{natural duality}. We also say that the structure $\utilde{\bf M}$ \emph{dualizes $\underline{\bf M}$}. When each $\epsilon_{\bf X}$ is also an isomorphism, we say that the natural duality $(D, E, e, \epsilon)$ is \emph{full}. A duality is full precisely when it is an equivalence between the categories $\mathcal{A}$ and $\mathcal{X}$. When a natural duality $(D, E, e, \epsilon)$ associates embeddings in $\mathcal{X}$ with surjections in $\mathcal{A}$ (equivalently, embeddings in $\mathcal{A}$ with with surjections in $\mathcal{X}$) we say that the duality is \emph{strong}. Strong dualities are full, but the converse is not in general true.

Priestley duality is an example of a natural duality: The 2-element bounded distributive lattice ${\bf 2}$ plays the role of $\underline{\bf M}$, and the 2-element Priestley space whose underlying order is a chain plays role of $\utilde{\bf M}$. Formulated in these terms, the dual of a bounded distributive lattice $\mathbb{A}$ does not consist of its collection of prime filters, but instead morphisms from ${\bf A}$ into the 2-element bounded distributive lattice. This mismatch is explained by the fact that every prime filter $x$ of ${\bf A}$ may be understood as a homomorphism $h_x\colon {\bf A}\to {\bf 2}$ given by $h_x(a)=1$ if and only if $a\in x$, and conversely that each prime filter of $\mathbb{A}$ may be understood as the preimage of $1$ under some homomorphism ${\bf A}\to{\bf 2}$. Similar remarks apply to the reverse functor.

Although Esakia duality is a restriction of Priestley duality to Heyting algebras, Esakia duality is not a natural duality because there is no finite algebra $\underline{\bf M}$ generating \textsf{HA} as a prevariety. This remains true even when we restrict our attention to G\"odel algebras.

\emph{Mutatis mutandis}, all the preceding remarks apply to versions of the Priestley and Esakia duality for algebras without designated bottom elements.

\subsection{Lattices with involution and Kleene algebras}

A \emph{lattice with involution} (or \emph{$i$-lattice}) is an algebra $(A,\meet,\join,\neg)$, where $(A,\meet,\join)$ is a lattice and $\neg$ is a unary operation satisfying the identities
$$\neg\neg a = a$$
$$\neg (a\join b) =\neg a\meet\neg b$$
$$\neg (a\meet b) = \neg a \join \neg b$$
An $i$-lattice is \emph{normal} if its lattice reduct is distributive and it satisfies the identity $a\meet \neg a \leq b\join \neg b$. Kalman in \cite{Kalman} showed that the variety of normal $i$-lattices is exactly $\mathbb{ISP}({\bf \underline{L}})$, where ${\bf\underline{L}} = (\{-1,0,1\}, \meet, \join, \neg)$ is the $i$-lattice defined by $-1<0<1$, and 
\[ \neg x = \begin{cases} 
      1, & \text{if }x=-1 \\
      0, & \text{if }x=0 \\
      -1, & \text{if }x=1 \\
   \end{cases}\]\\
We denote by \textsf{IL} the category of normal $i$-lattices.

The expansion of a normal $i$-lattice by bounds $\bot$ and $\top$ for the lattice order is called a \emph{Kleene algebra}. In the presence of these bounds, for any $a$ we have that $\neg\bot = \neg (\bot\meet \neg a) = \neg\bot \join a$, whence $\neg\bot = \top$ and $\neg\top = \bot$. Kleene algebras are generated as a prevariety by the Kleene algebra ${\bf \underline{K}} = (\{-1,0,1\}, \meet, \join,\neg,-1,1)$ obtained by expanding the signature for the normal $i$-lattice ${\bf\underline{L}}$ by constant symbols for its least and greatest elements.

The relevance of normal $i$-lattices to the present study is explained by the following proposition.

\begin{proposition}\label{prop:reduct}
Let ${\bf A} = (A,\meet,\join,\cdot,\to,t,\neg)$ be a Sugihara monoid. Then ${\bf A}$ satisfies $a\meet\neg a\leq \neg t\leq t\leq b\join\neg b$, and hence $(A,\meet,\join,\neg)$ is a normal $i$-lattice.
\end{proposition}

\begin{proof}
It suffices to check that the identity $a\meet\neg a\leq \neg t\leq t\leq b\join\neg b$ holds in every Sugihara monoid. For this, by Proposition \ref{prop:Sugiharagenerators} it is enough to check that this identity holds on the generating algebras $\Zo$ and $\Ze$. Let $n,m\in \mathbb{Z}$. Then $n\meet \neg n = n\meet -n = -|n|\leq 0$ and $m\join \neg m = m\join -m = |m|\geq 0$, whence $n\meet \neg n \leq 0\leq m \join \neg m$ in $\Zo$. If $n,m\neq 0$, then $n\meet - n\leq -1\leq 1\leq m\join \neg m$ gives the identity for $\Ze$. The result follows.
\end{proof}

We may likewise obtain an analogue for bounded Sugihara monoids.

\begin{corollary}\label{cor:reduct}
Let $(A,\meet,\join,\cdot,\to,t,\neg,\bot,\top)$ be a bounded Sugihara monoid. Then $(A,\meet,\join,\neg,\bot,\top)$ is a Kleene algebra.
\end{corollary}

\subsection{The Davey-Werner duality}

In \cite{DavWer}, Davey and Werner established a natural duality for the variety of Kleene algebras. Under this duality, the alter ego for $\underline{{\bf K}}$ consists of the topological relational structure $\utilde{{\bf K}} = (\{-1,0,1\}, \leq, \Q, K_0, \tau)$, where $\leq$ is the partial order given by $-1<0$ and $1<0$, $\Q$ is the relation of comparability with respect to $\leq$ given by $x\Q y$ iff $\la x,y\ra\notin \{\la -1,1\ra,\la 1,-1\ra\}$, $K_0=\{-1,1\}$ is a set of designated minimal elements, and $\tau$ is the discrete topology on $\{-1,0,1\}$. The concrete category of isomorphic copies of closed substructures of nonempty powers of $\utilde{\bf K}$ form a dual category to the variety Kleene algebras, and may be given the following external characterization (see \cite[p. 107]{CD} and \cite{DavWer}).

\begin{figure}
\begin{center}
\begin{tikzpicture}
    \node[label=left:\tiny{$1$}] at (0,0) {$\bullet$};
    \node[label=left:\tiny{$0$}] at (0,-0.5)  {$\bullet$};
    \node[label=left:\tiny{$-1$}] at (0,-1) {$\bullet$};

    \draw (0,0) -- (0,-0.5);
    \draw (0,-0.5) -- (0,-1);
\end{tikzpicture}\hspace{0.2 in}
\begin{tikzpicture}
    \node[label=left:\tiny{$0$}] at (0,-0.5) {$\bullet$};
    \node[label=left:\tiny{$-1$}] at (-0.5,-1)  {\textcircled{$\bullet$}};
    \node[label=right:\tiny{$1$}] at (0.5,-1) {\textcircled{$\bullet$}};

    \draw (0,-0.5) -- (-0.5,-1);
    \draw (0,-0.5) -- (0.5,-1);
\end{tikzpicture}
\end{center}
\caption{Hasse diagrams for the different personalities of the object ${\bf K}$}
\label{fig:Ktilde}
\end{figure}
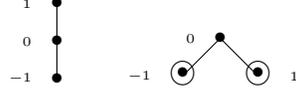

\begin{proposition}\label{prop:DWextchar}
$(X,\leq,\Q,D,\tau)$ is an isomorphic copy of a closed substructure of a nonempty power of $\utilde{{\bf K}}$ if and only if:
\begin{enumerate}
\item $(X,\leq,\tau)$ is a Priestley space,
\item $Q$ is a closed binary relation,
\item $D$ is a closed subspace,
\item for all $x\in X$, $x\Q x$,
\item for all $x,y\in X$, $x\Q y \text{ and } x\in D\implies x\leq y$,
\item for all $x,y,z\in X$, $x\Q y \text{ and } y\leq z\implies z\Q x$
\end{enumerate}
\end{proposition}
We call the structured topological spaces described above \emph{Kleene spaces}. We denote the category of Kleene algebras by \textsf{KA}, and the category of Kleene spaces with continuous structure-preserving morphisms by \textsf{KS}.

The methods used to obtain a natural duality for \textsf{KA} may be used with little modification to produce a natural duality for normal $i$-lattices.

\begin{theorem}
The variety of normal $i$-lattices is dualized by the structure $\utilde{\bf L} = (\{-1,0,1\}, \leq, Q, L_0, 0, \tau)$, where $\leq$ is the partial order given by $-1<0$ and $1<0$, $L_0$ is the unary relation $\{-1,1\}$, $Q$ is the binary relation given by $xQy$ iff $\la x,y\ra\notin \{\la -1,1\ra,\la 1,-1\ra\}$, and $0$ is a designated nullary constant symbol for the greatest element with respect to $\leq$. Moreover, this duality is strong.
\end{theorem}

\begin{proof}
We apply the NU Strong Duality Theorem \cite[Theorem 3.8]{CD} as applied to algebras with a majority term. The universes of subalgebras of $\underline{\bf L}^2$ are exactly $\{0\}$, $\Delta_{L_0}$, $\leq\cap (L_0\times L)$, $\geq\cap (L\times L_0)$, $L_0\times L$, $L\times L_0$, $L^2$, $\Delta_L$, $\leq$, $\geq$, $\Q$, $L_0\times \{0\}$, $\{0\}\times L_0$, $L\times \{0\}$, $\{0\}\times L$, and $L_0^2$. These are readily seen to be entailed by $\leq$, $L_0$, $Q$, and $\top$.

Next, we note that the partial and total homomorphisms of arity at most $1$ are given by 
$$\varphi_0\colon \{0\}\to {\underline{\bf L}} \text{ defined by } \varphi_0(0)=0$$
$$\varphi_1\colon \{-1,1\}\to {\underline{\bf L}} \text{ defined by } \varphi_1(-1)=\varphi_1(1)=0$$
$$\varphi_2\colon \{-1,1\}\to {\underline{\bf L}} \text{ defined by } \varphi_2(-1)=-1\text{ and }\varphi_2(1)=1$$
$$\varphi_3\colon {\underline{\bf L}}\to {\underline{\bf L}} \text{ defined by } \varphi_3(-1)=\varphi_3(0)=\varphi_3(1)=0$$
$$\varphi_4\colon {\underline{\bf L}}\to {\underline{\bf L}} \text{ defined by } \varphi_4(x)=x \text{ for all }x\in\{-1,0,1\}$$
\noindent The graphs of these functions are, respectively,
$$\grph(\varphi_0) = \{(0,0)\}=\{0\}\times\{0\}$$
$$\grph(\varphi_1) = \{(-1,0),(1,0)\}=L_0\times \{0\}$$
$$\grph(\varphi_2) = \{(-1,1),(1,1)\}=\Delta_{L_0}$$
$$\grph(\varphi_3) = \{(-1,0),(0,0),(1,0)\} = L\times \{0\}$$
$$\grph(\varphi_4) = \{(-1,-1),(0,0),(1,1)\} = L\times \{0\}$$
\noindent It follows from this that $\leq$, $0$, $L_0$, $Q$ entails all of the relations and partial operations listed above.

For hom-entailment, note by the $\utilde{\bf M}$-Shift Strong Duality Lemma \cite[Lemma 2.8]{CD}, we may delete $\varphi_0$, $\varphi_1$, and $\varphi_2$ since they have extensions $\varphi_3$ and $\varphi_4$. Observe that $\varphi_4$ is the identity endomorphism and is therefore hom-entailed by any set of partial operations. $\varphi_3$ is the constant endomorphism associated with $0$, and is thus entailed by $0$. The result therefore follows.
\end{proof}

\begin{theorem}\label{thm:pKSchar}
$(X,\leq,\Q,D,\top,\tau)$ is an isomorphic copy of a closed substructure of a nonempty power of $\utilde{{\bf K}}$ if and only if:
\begin{enumerate}
\item $(X,\leq,\top,\tau)$ is a pointed Priestley space,
\item $Q$ is a closed binary relation,
\item $D$ is a closed subspace,
\item for all $x\in X$, $x\Q x$,
\item for all $x,y\in X$, $x\Q y \text{ and } x\in D\implies x\leq y$,
\item for all $x,y,z\in X$, $x\Q y \text{ and } y\leq z\implies z\Q x$
\end{enumerate}
\end{theorem}

\begin{proof}
Identical to the proof of Proposition \ref{prop:DWextchar}. 
\end{proof}

We call the spaces defined in the previous theorem \emph{pointed Kleene spaces}, and denote the category of pointed Kleene spaces with continuous structure-preserving maps by \textsf{pKS}. The above theorems show that \textsf{IL} is dually equivalent to \textsf{pKS}, and we denote the functors witnessing this equivalence by $(-)_+\colon\textsf{IL}\to\textsf{pKS}$ and $(-)^+\colon\textsf{pKS}\to\textsf{IL}$. The category \textsf{pKS} plays the same role in the duality for Sugihara monoids that \textsf{PS} plays in Esakia duality. Following this analogy, for simplicity we will also use the notation $(-)_+$ and $(-)^+$ for the functors witnessing the equivalence of $\textsf{KA}$ and $\textsf{KS}$, and later for the functors of the duality for Sugihara monoids and their bounded analogues. This agrees with our convention of using $(-)_*$ and $(-)^*$ for the functors associated with the dualities for \textsf{DL}, \textsf{Br}, \textsf{HA}, \textsf{bRSA}, and \textsf{bGA} in Section \ref{sec:booldual}.

\begin{remark}\label{rem:minimal}
Suppose that $(X,\leq,Q,D,\tau)$ is a Kleene space (or, if one wishes, a pointed Kleene space) and let $x\in D$. It then follows from the axioms for Kleene spaces that $x$ is $\leq$-minimal in $X$.
\end{remark}

\section{Esakia duality for Sugihara monoids}\label{sec:sugidual}

Proposition \ref{prop:reduct} shows that each Sugihara monoid ${\bf A}$ may be associated with its normal $i$-lattice reduct via a forgetful functor $U\colon\textsf{SM}\to\textsf{IL}$. On the other hand, the Davey-Werner duality for normal $i$-lattices associates to each such reduct a pointed Kleene space $U({\bf A})_+$. By composing $U$ and $(-)_+$, we obtain a functor that associates to each Sugihara monoid the pointed Kleene space that is dual to its $i$-lattice reduct. For simplicity, we omit explicit mention of the forgetful functor $U$, and simply write the pointed Kleene space obtained in this fashion by ${\bf A}_+$.

We will identify a class of pointed Kleene spaces, which we call \emph{Sugihara spaces}, that contain the spaces arising in the aforementioned way. On the other hand, to each Sugihara space ${\bf X}$ we will associate the normal $i$-lattice ${\bf X}^+$. It turns out that each $i$-lattice arising in this fashion is the reduct of Sugihara monoid, and, moreover, determines a unique such Sugihara monoid. In this way, the functor $(-)^+$ from the Davey-Werner duality may be amended to give a functor to \textsf{SM}. The main result of this section is that the pair $(-)_+$ and $(-)^+$, appropriately modified, witness a dual equivalence of categories between \textsf{SM} and a subcategory of \textsf{pKS}.

\subsection{Sugihara spaces and \textsf{bRS}-spaces}

Before describing the duality for Sugihara monoids in detail, we introduce the pointed Kleene spaces of interest and clarify their connection to the \textsf{bRS}-spaces of Section \ref{sec:booldual}. The following isolates the appropriate class of pointed Kleene spaces for our study.

\begin{definition}
A pointed Kleene space $(X,\leq,Q,D,\top,\tau)$ is called a \emph{Sugihara space} if
\begin{enumerate}
\item $(X,\leq,\top,\tau)$ is a pointed Esakia space,
\item $\Q$ is the relation of comparability with respect to $\leq$, i.e., $\Q=\leq\cup\geq$, and
\item $D$ is open.
\end{enumerate}
Because the relation $\Q$ is understood to be comparability with respect to $\leq$, we sometime omit it and simply say that $(X,\leq,D,\top,\tau)$ is a Sugihara space. Observe that since $D$ is closed in any pointed Kleene space, the above definition entails that $D$ is clopen in a Sugihara space.
\end{definition}

These spaces bear a striking similarity to the \textsf{bRS}-spaces of Section \ref{sec:booldual}, and indeed we have the following.

\begin{lemma}\label{lem:bRStosugihara}
Let $(X,\leq,D,\top,\tau)$ be a \textsf{bRS}-space. Then $(X,\leq,\leq\cup\geq,D,\top,\tau)$ is a Sugihara space.
\end{lemma}

\begin{proof}
From the definition of \textsf{bRS}-spaces, $(X,\leq,\top,\tau)$ is a pointed Esakia space and $D$ is clopen. We need only verify the conditions listed in Theorem \ref{thm:pKSchar} to show that $(X,\leq,\leq\cup\geq,D,\top,\tau)$ is a pointed Kleene space. Note that (1) and (3) follow immediately from the preceding comments, and the order relation $\leq$ is closed in $X\times X$ for any Priestley space, and this gives (2). It remains only to show that conditions (4), (5), and (6) are satisfied. Let $\Q=\leq\cup\geq$ be the relation of comparability with respect to $\leq$.

For (4), since each $x\in X$ is comparable to itself, we have $x\Q x$.

For (5), let $x,y\in X$ with $x\Q y$ and $x\in D$. Since $x\Q y$ we have either $x\leq y$ or $y\leq x$. In the former case, $x\leq y$ holds by hypothesis. In the latter case, observe that since $D$ consists of $\leq$-minimal elements by Remark \ref{rem:minimal}, we have that $y\leq x$ and $x\in D$ implies $x=y$. Hence $x\leq y$ in either case.

For (6), let $x,y,z\in X$ with $x\Q y$ and $y\leq z$. Since $x\Q y$ we have either $x\leq y$ or $y\leq x$. In the first case, $x\leq y$ and $y\leq z$ gives $x\leq z$ by transitivity. In the second case, $y\leq x$ and $y\leq z$ gives $x,z\in\upset y$. But $(X,\leq)$ is a forest since it is the underlying poset of a \textsf{bRS}-space, so $\upset y$ is a chain. Hence $x\leq z$ or $z\leq x$, so $z\Q x$ as desired. The result follows.
\end{proof}

A converse to the above lemma also holds.

\begin{lemma}\label{lem:sugiharatobRS}
Let $(X,\leq,\Q,D,\top,\tau)$ be a Sugihara space. Then $(X,\leq,D,\top,\tau)$ is a \textsf{bRS}-space.
\end{lemma}

\begin{proof}
From the definition of Sugihara spaces, $(X,\leq,\top,\tau)$ is a pointed Esakia space and $D$ is clopen. From Definition \ref{def:bRSspace}, it remains only to show that $D$ consists of $\leq$-minimal elements and that $(X,\leq)$ is a forest.

To see that $D$ consists of minimal elements, let $y\in D$ and let $x\leq y$. From $x\leq y$ we have $y\Q x$ since $\Q$ is the relation of $\leq$-comparability. Then $y\Q x$ and $y\in D$ gives $y\leq x$ by Theorem \ref{thm:pKSchar}(2). Since $x\leq y$, antisymmetry yields $x=y$. Hence $D$ consists of minimal elements.

To see that $(X,\leq)$ is a forest, let $x\in X$ and let $y,z\in \upset x$. Note that $x\leq y$ gives $y\Q x$, and $x\leq z$ together with Theorem \ref{thm:pKSchar}(3) gives $z\Q y$. Then $z\leq y$ or $y\leq z$. It follows that $\upset x$ is a chain, and hence that $(X,\leq)$ is a forest.
\end{proof}

In light of Lemmas \ref{lem:bRStosugihara} and \ref{lem:sugiharatobRS}, \textsf{bRS}-spaces and and Sugihara spaces are tantamount to the same objects. However, conceptually they arise from quite different origins: Whereas Sugihara spaces are Davey-Werner duals of some (as yet unidentified) normal $i$-lattices, \textsf{bRS}-spaces are enriched Esakia duals of \textsf{bRS}-algebras. Our proximal goal is to develop this connection more thoroughly. 

To fix some notation, let ${\bf A} = (A,\meet,\join,\cdot,\to,t,\neg)$ be a Sugihara monoid. Define $A_+$ to be the collection of $(\meet,\join,\neg)$-morphisms from ${\bf A}$ to ${\bf \underline{L}}$. We denote by $\leq$ the partial order on $A_+$ inherited pointwise from $\utilde{\bf L}$, denote the designated subset by $A_0 = \{h\in A_+ : (\forall a\in A)(h(a)\in\{-1,1\})\}$, define $\top\colon {\bf A}\to{\bf A}$ by $\top(a)=0$ for all $a\in A$, and define $\Qa$ to to be the binary relation on $A_+$ given by $h\Qa k$ if and only if $h(a)\Q k(a)$ for all $a\in A$. Moreover, we let $\tau_{\bf A}$ be the topology on $A_+$ generated by the subbasis $\{U_{a,l} : a\in A, l\in \{-1,0,1\}\}$, where $U_{a,l} = \{h\in A_+ : h(a)=l\}$. The latter definition is motivated by the following.

\begin{lemma}[{{\cite[Lemma B.6, p. 340]{CD}}}]\label{lem:DWdualtop}
Let $A$ be an index set and consider ${\bf \underline{L}}^A$ as a topological space endowed with the product topology. For each $a\in A$ and\\ $l\in \{-1,0,1\}$, let $U_{a,l}=\{x\in {\bf \underline{L}}^A : x(a)=l\}$. Then
$$\{U_{a,l} : a\in A \text{ and } l\in \{-1,0,1\}\}$$
is a clopen subbasis for the topology on ${\bf \underline{L}}^A$.
\end{lemma}

Given an $i$-lattice ${\bf A}$, the Davey-Werner dual of ${\bf A}$ has topology induced as a subspace of ${\bf \underline{L}}^A$. Hence from the previous lemma we obtain

\begin{lemma}\label{lem:DWdualtop}
Let ${\bf A} = (A,\meet,\join,\cdot,\to,t,\neg)$ be a Sugihara monoid. Then the sets $U_{a,l}=\{h\in A_+ : h(a)=l\}$, where $l\in \{-1,0,1\}$ and $a\in A$, give a clopen subbasis for the topology on ${\bf A}_+$.
\end{lemma}

It follows that ${\bf A}_+ = (A_+,\leq,\Qa, A_0, \top, \tau_{\bf A})$ is the Davey-Werner dual of the normal $i$-lattice $(A,\meet,\join,\neg)$ as discussed above.

\begin{lemma}\label{lem:welldeffilter}
Let ${\bf A} = (A,\meet,\join,\cdot,\to,t,\neg)$ be a Sugihara monoid and let $h\in A_+$. Then $h^{-1}[\{0,1\}]\cap A^-$ is a prime filter of the enriched negative cone ${\bf A}_{\bowtie}$.
\end{lemma}

\begin{proof}
This follows immediately since $\{0,1\}$ is a prime filter of $\underline{\bf L}$ and $h$ is a lattice homomorphism.
\end{proof}

For a Sugihara monoid ${\bf A}$, define a map $\xia\colon (A_+,\leq)\to (A_{{\bowtie}*},\subseteq)$ by
$$\xia(h) = h^{-1}[\{0,1\}]\cap A^-.$$
Lemma \ref{lem:welldeffilter} shows that $\xia$ is well-defined.

\begin{lemma}\label{lem:orderpreserving}
Let ${\bf A}$ be a Sugihara monoid. Then $\xia$ is isotone.
\end{lemma}

\begin{proof}
Let $h_1,h_2\in A_+$ with $h_1\leq h_2$. If $a\in \xia(h_1)$, then $a\leq t$. Also, we have $h_1(a)\in\{0,1\}$. Since $h_1\leq h_2$, this gives $1 \leq h_1(a)\leq h_2(a)$. Thus $a\in h_2^{-1}[\{0,1\}]$, giving $a\in\xia(h_2)$. It follows that $\xia(h_1)\subseteq\xia(h_2)$.
\end{proof}

\begin{lemma}\label{lem:timage}
Let $(A,\meet,\join,\cdot,\to,t,\neg)$ be a Sugihara monoid and let $h\in A_+$. Then $h(t)\in\{0,1\}$.
\end{lemma}

\begin{proof}
By Proposition \ref{prop:reduct}, the identity $\neg t\leq t$ holds in every Sugihara monoid. Were if the case that $h(t)=-1$, we would have $h(\neg t) = \neg h(t) = 1$. But $\neg t \leq t$ gives $h(\neg t)\leq h(t)$, a contradiction. Thus $h(t)\in \{0,1\}$.
\end{proof}

\begin{lemma}\label{lem:orderreflecting}
Let ${\bf A} = (A,\meet,\join,\cdot,\to,t,\neg)$ be a Sugihara monoid. Then $\xia$ is order-reflecting.
\end{lemma}

\begin{proof}
Let $h_1,h_2\in A_+$ with $\xia(h_1)\subseteq\xia(h_2)$. Let $a\in A$. Were it the case that $h_1(a)\not\leq h_2(a)$, then either $h_2(a)=-1$ and $h_1(a)\neq -1$, or $h_2(a)=1$ and $h_1(a)\neq 1$.

In the first case, $h_1(a)\in\{0,1\}$ and by Lemma \ref{lem:timage} it follows that we have $h_1(a\meet t)=h_1(a)\meet h_1(t)\in \{0,1\}$ as well. Since $a\meet t\in A^-$, it follows that $a\meet t\in \xia(h_1)$. This gives $a\meet t\in\xia(h_2)$. But $h_2(a)=-1$ and $h_2(t)\in\{0,1\}$ gives $h_2(a\meet t) = -1$, a contradiction.

In the second case, $h_1(a)\in\{-1,0\}$ and $h_2(a)=1$. Then $h_1(\neg a)\in\{0,1\}$ and $h_2(\neg a) = -1$. Thus the second case reduces to the first case, and we arrive at a contradiction again. It follows that $h_1(a)\leq h_2(a)$, and hence that $\xia$ is order-reflecting.
\end{proof}

\begin{lemma}\label{lem:orderisomorphism}
Let ${\bf A}=(A,\meet,\join,\cdot,\to,t,\neg)$ be a Sugihara monoid. Then $\xia$ is an order isomorphism.
\end{lemma}

\begin{proof}
It suffices to show that $\xia$ is surjective. Note that the map $h$ given by $h(a)=0$ for all $a\in A$ is a $(\meet,\join,\neg)$-morphism such that $\xia(h)=A^-$. Now let $x$ be a prime filter of ${\bf A}_{\bowtie}$. Then $I=\{a\in A^- : a\notin x\}$ is a prime ideal of ${\bf A}_{\bowtie}$, being the complement of a prime filter. Also, $I$ is an ideal of ${\bf A}$. A trivial argument shows that $F=\upset_{\bf A} x = \{b\in A : a\leq b \text{ for some } a\in x\}$ is a filter of ${\bf A}$, and $F\cap I = \emptyset$. The prime ideal theorem then guarantees that there exists a prime ideal $J$ of ${\bf A}$ with $I\subseteq J$ and $F\cap J=\emptyset$. One may readily show that the set $\neg J = \{\neg a : a\in J\}$ is a prime filter of ${\bf A}$. Define a map $h\colon {\bf A}\to {\bf \underline{L}}$ by
\[ h(a) = \begin{cases} 
      1 & \text{ if } a\in \neg J\\
      0 & \text{ if } a\notin J\cup\neg J\\
     -1 & \text{ if } a\in J\\
   \end{cases}
\]
Notice that if $a,\neg a\in J$, then $J$ being an ideal gives that $a\join\neg a\in J$. Proposition \ref{prop:reduct} gives that $t\leq a\join\neg a$, so $J$ being downward-closed then gives that $t\in J$. But this is impossible since $J\cap x=\emptyset$ and $t\in x$ (as $x$ is a prime filter of ${\bf A}_{\bowtie}$). Hence for each $a\in A$, either $a\notin J$ or $\neg a\notin J$, whence $J\cap \neg J=\emptyset$. This implies that at most one of $a\in \neg J$, $a\in J$, or $a\notin J\cup\neg J$ holds. As at least one of $a\in J$, $a\in \neg J$, or $a\notin J\cup\neg J$ must hold, this yields that $h$ is a well-defined function.

By checking cases, one may verify that $h$ is an $i$-lattice homomorphism, and hence $h\in A_+$. It is easy to show that $\xia(h)=x$. Because Lemmas \ref{lem:orderpreserving} and \ref{lem:orderreflecting} show that $\xia$ is an order embedding, this proves that $\xia$ is an order isomorphism.
\end{proof}

\begin{example}
Recall the algebra ${\bf E}$ introduced in Example \ref{parex} has Hasse diagram
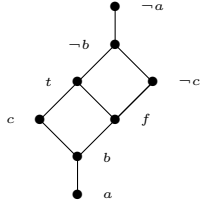
\begin{center}
\begin{tikzpicture}
    \node[label=right:\tiny{$\neg a$}] at (0,0) {$\bullet$};
    \node[label=left:\tiny{$\neg b$}] at (0,-0.5) {$\bullet$};
    \node[label=left:\tiny{$t$}] at (-0.5,-1)  {$\bullet$};
    \node[label=right:\tiny{$\neg c$}] at (0.5,-1) {$\bullet$};
    \node[label=right:\tiny{$f$}] at (0,-1.5) {$\bullet$};
    \node[label=left:\tiny{$c$}] at (-1,-1.5) {$\bullet$};
    \node[label=right:\tiny{$b$}] at (-0.5,-2) {$\bullet$};
    \node[label=right:\tiny{$a$}] at (-0.5,-2.5) {$\bullet$};

    \draw (0,0) -- (0,-0.5);
    \draw (0,-0.5) -- (-0.5,-1);
    \draw (0,-0.5) -- (0.5,-1);
    \draw (-0.5,-1) -- (0,-1.5);
    \draw (0.5,-1) -- (0,-1.5);
    \draw (-0.5,-1) -- (-1,-1.5);
    \draw (-1,-1.5) -- (-0.5,-2);
    \draw (0.5,-1) -- (-0.5,-2);
    \draw (-0.5,-2) -- (-0.5,-2.5);
\end{tikzpicture}
\end{center}
If we consider the filter $x=\{b,c,f,t\}$ of the negative cone, then in the proof of Lemma \ref{lem:orderisomorphism} we have that $I$ is $\{a\}$, $F$ is $A\setminus\{a\}$, $J$ is $\{a\}$, and $\neg J$ is $\{\neg a\}$. If instead $x=\{c,t\}$, then $I$ is $\{a,b,f\}$, $F$ is $\{c,t,\neg b,\neg a\}$, $J$ is $\{a,b,f,\neg c\}$, and $\neg J$ is $\{c,t,\neg b,\neg a\}$. Finally, if $x=\{t,f\}$, then $I$ is $\{a,b,c\}$, $F$ is $\{t,f,\neg b,\neg c,\neg a\}$, $J$ is $\{a,b,c\}$, and $\neg J$ is $\{\neg c,\neg b,\neg a\}$.
\end{example}

The isomorphism described in the foregoing lemmas turns out to provide more than an order-theoretic correspondence, as shown in the following.

\begin{lemma}\label{lem:continuous}
Let ${\bf A}=(A,\meet,\join,\cdot,\to,t,\neg)$ be a Sugihara monoid. Then $\xia$ is continuous.
\end{lemma}

\begin{proof}
It suffices to show that the inverse image of each subbasis element is open, so let $a\in A^-$. Then
\begin{align*}
\xia^{-1}[\sigma(a)] &= \xia^{-1}[\{x\in A_{{\bowtie}*} : a\in x\}]\\
&= \{h\in A_+ : a\in\xia(h)\}\\
&= \{h\in A_+ : a\in h^{-1}[\{0,1\}]\cap A^-\}\\
&= \{h\in A_+ : h(a)\in\{0,1\}\}\\
&= \{h\in A_+ : h(a)=0\}\cup\{h\in A_+ : h(a)=1\}\\
&= U_{a,0}\cup U_{a,1}
\end{align*}
Thus $\xia$ is continuous.
\end{proof}

\begin{lemma}\label{lem:priestleyiso}
Let ${\bf A}=(A,\meet,\join,\cdot,\to,t,\neg)$ be a Sugihara monoid. Then ${\bf A}_+$ and ${\bf A}_{{\bowtie}*}$ are isomorphic as Priestley spaces.
\end{lemma}

\begin{proof}
Lemma \ref{lem:orderisomorphism} shows that $\xia$ is an order isomorphism. In particular, this shows that $\xia$ is a bijection. Lemma \ref{lem:continuous} shows that $\xia$ is continuous. Continuous bijections of compact Hausdorff spaces are homeomorphisms, so it follows that $\xia$ is a homeomorphism. Thus $\xia$ is an isomorphism in \textsf{PS}.
\end{proof}

As a consequence of the above, we obtain

\begin{lemma}\label{lem:aplusesakia}
Let ${\bf A} = (A,\meet,\join,\cdot,\to,t,\neg)$  a Sugihara monoid, and denote by ${\bf A}_+ = (A_+, \leq,\Qa,A_0,\top,\tau_{\bf A})$ its Davey-Werner dual. Then $(A_+,\leq,\tau_{\bf A})$ is an Esakia space.
\end{lemma}

\begin{proof}
Every Priestley space that is \textsf{PS}-isomorphic to an Esakia space is itself an Esakia space, so the result follows from Lemma \ref{lem:priestleyiso}.
\end{proof}

\begin{lemma}\label{lem:aplusbRS}
Let ${\bf A} = (A,\meet,\join,\cdot,\to,t,\neg)$ be a Sugihara monoid and denote by ${\bf A}_+ = (A_+, \leq,Q_{\bf A},A_0,\top,\tau_{\bf A})$ its Davey-Werner dual. Then $(A_+,\leq,A_0,\top,\tau_{\bf A})$ is a \textsf{bRS}-space.
\end{lemma}

\begin{proof}
$(A_+,\leq,\tau_{\bf A})$ is an Esakia space by Lemma \ref{lem:aplusesakia}, and the fact that $(A_+,\leq)$ is a forest follows since $\xia$ is an order isomorphism and $(A_*,\subseteq)$ is a forest. It remains only to show that $A_0$ is a clopen collection of $\leq$-minimal elements. That $A_0$ consists of minimal elements holds because ${\bf A}_+$ is pointed Kleene space. To see that $A_0$ is clopen, let $x=\xia(h)=h^{-1}[\{0,1\}]\cap A^-$. Then for all $a\in x$, we have $h(a)\in\{0,1\}$. Observe that
\begin{align*}
x\in\sigma(\neg t) &\iff \neg t\in x\\
&\iff h(\neg t)\in \{0,1\}\\
&\iff h(t)\in \{0,-1\}
\end{align*}
From Lemma \ref{lem:timage}, the above shows that $x\in\sigma(\neg t)$ if and only if $h(t)=0$. Now if $h\in A_0$, then $h(a)\in\{-1,1\}$ for all $a\in A$ and thus $\xia(h)\notin\sigma(\neg t)$ by the above, so $\xia[A_0]\subseteq \sigma(\neg t)^\comp$. On the other hand, suppose that $x\in\sigma(\neg t)^\comp$. Then the above shows that $h(t)\notin\{0,-1\}$, whence $h(t)=1$. Were it the case that $h(a)=0$ for some $a\in A$, we would have $h(\neg a)=0$ and hence $h(a\join\neg a)=0$. But this is impossible since $t\leq a\join\neg a$ and $h$ is isotone, so it follows that the image of $h$ is contained in $\{-1,1\}$. This implies that $\sigma(\neg t)\subseteq \xia[A_0]$, so $\sigma(\neg t)=\xia[A_0]$. Because $\xia$ is a homeomorphism and $\sigma(\neg t)$ is clopen, it follows that $A_0$ is clopen. This proves the lemma.
\end{proof}

\begin{lemma}\label{lem:bRSisomorphism}
Let ${\bf A} = (A,\meet,\join,\cdot,\to,t,\neg)$ be a Sugihara monoid. Then $\xia$ is an isomorphism of \textsf{bRS}-spaces.
\end{lemma}

\begin{proof}
Lemma \ref{lem:priestleyiso} shows that $\xia$ is an isomorphism of Priestley spaces, and hence an Esakia function. It thus suffices to show that $\xia$ preserves the top element, the designated subset, and its complement. The greatest element of ${\bf A}_+$ is the morphism $\top\colon {\bf A}\to {\bf\underline{L}}$ defined by $\top(a)=0$. Observe that we have $\xia(\top) = \top^{-1}[\{0,1\}]\cap A^- = A^-$, which is the $\subseteq$-greatest element of ${\bf A}_{{\bowtie}*}$.

Next, we show that
$$\xia[\{h\in A_+ : (\forall a\in A)(h(a)\in\{-1,1\}\})]=\sigma(\neg t)^\comp.$$
For the forward inclusion, let $h\in A_+$ with image contained in $\{-1,1\}$. Since $h(t)\in\{0,1\}$ always holds, this gives $h(t)=1$ and hence $h(\neg t) = -1$. If $\xia(h)\in\sigma(\neg t)$, this gives $\neg t\in h^{-1}[\{0,1\}]$, which contradicts $h(\neg t) = -1$. Thus $\xia(h)\in\sigma(\neg t)^\comp$.

For the reverse inclusion, let $x\in\sigma(\neg t)^\comp$. Then $\neg t\notin x$. By the surjectivity of $\xia$, there exists $h\in A_+$ such that $\xia(h)=x$. Suppose that there exists $a\in A$ such that $h(a)=0$. By Proposition \ref{prop:reduct}, the identity $x\meet\neg x\leq\neg t\leq t\leq y\join\neg y$ holds in ${\bf A}$. In particular, this gives $a\meet\neg a\leq \neg t\leq t\leq a\join \neg a$. Since $h(\neg a)=\neg h(a) = 0$, the isotonicity of $h$ gives
$$0 = h(a\meet \neg a)\leq\neg t\leq t\leq h(a\join\neg a)=0,$$
so $h(\neg t) = h(t) = 0$. It follows that $\neg t\in h^{-1}[\{0,1\}]\cap A^- = x$, contradicting $\neg t\notin x$. Thus $h(a)\in\{-1,1\}$ for all $a\in A$, and we obtain the reverse containment.

It remains only to show that
$$\xia[\{h\in A_+ : (\exists a\in A)(h(a)=0)\}] = \sigma(\neg t).$$
But this follows immediately by taking complements since $\xia$ is a bijection.
\end{proof}

\subsection{Esakia duality for Sugihara monoids}

For a Sugihara monoid ${\bf A}$, the previous section provides an extremely close connection between ${\bf A}_+$ and the \textsf{bRS}-algebra ${\bf A}_{\bowtie*}$. At the same time, there is a close connection between \textsf{bRS}-spaces and Sugihara spaces. We exploit these connections to show that the Davey-Werner dual ${\bf A}_+$ is actually a Sugihara space. 

Given a universal set $X$ and $U,V\subseteq X$ with $U\cup V = X$, define a map $C_{U,V}\colon X\to \{-1,0,1\}$ by
\[ C_{U,V}(x) = \begin{cases} 
      1, & \text{if }x\notin V \\
      0, & \text{if }x\in U\cap V \\
      -1, & \text{if }x\notin U \\
   \end{cases}
\]\\
Observe that the well-definedness of $C_{U,V}$ hinges on $U\cup V=X$.

\begin{lemma}\label{lem:ilatticemorph1}
Let ${\bf X} = (X,\leq,\Q,D,\top,\tau)$ be a pointed Kleene space and let $U,V\subseteq X$ with $U\cup V=X$. Then $C_{U,V}$ is a pointed Kleene space morphism from ${\bf X}$ to $\utilde{\bf L}$ if and only if $U,V$ are clopen up-sets with $(X\setminus U\times X\setminus V)\cap Q =\emptyset$ and $U\cap V\subseteq D^\comp$.
\end{lemma}

\begin{proof}
Suppose first that $C_{U,V}\colon X \to \utilde{\bf L}$ is a pointed Kleene space morphism. Since $C^{-1}_{U,V}(\{0,1\})=U$ and $C^{-1}_{U,V}(\{-1,0\})=V$, both $U$ and $V$ are clopen up-sets. Suppose that $x,y\in X$ with $x\notin U$ and $y\notin V$. Then $C_{U,V}(x)=-1$ and $C_{U,V}(y)=1$, so $x\Q y$ cannot hold and $(X\setminus U\times X\setminus V)\cap Q =\emptyset$. Finally, suppose that $x\in U\cap V$. Then $C_{U,V}(x)=0\notin K_0$, so $x\notin D$. It follows that $x\in D^\comp$ and $U\cap V\subseteq D^\comp$.

For the converse, suppose that $U$ and $V$ are clopen up-sets satisfying the condition $(X\setminus U\times X\setminus V)\cap Q =\emptyset$ and $U\cap V\subseteq D^\comp$. We claim that $C_{U,V}$ is a pointed Kleene space morphism. To see that $C_{U,V}$ preserves $\leq$, suppose that $x,y\in X$ with $x\leq y$. If $C_{U,V}(y)=0$ then $C_{U,V}(x)\leq C_{U,V}(y)$ obviously holds. If $C_{U,V}(y)=1$, then $y\notin V$ and, since $V$ is an up-set, $x\notin V$ as well. This shows $C_{U,V}(x)=1$. Similarly, if $C_{U,V}(y)=-1$ then $C_{U,V}(x)=-1$. The monotonicity of $C_{U,V}$ follows.

To see that $C_{U,V}$ preserves $\Q$, let $x,y\in X$ satisying $C_{U,V}(y)=1$ and $C_{U,V}(x)=-1$. Then $y\notin V$ and $x\notin U$, whence $(x,y)\in X\setminus U\times X\setminus V$. This yields $(x,y)\notin Q$. It follows that $x\Q y$ implies $C_{U,V}(x)\Q C_{U,V}(y)$.

To see that $D$ is preserved, let $x\in D$. Then $x\notin U\cap V\subseteq D^\comp$, so $C_{U,V}(x)=-1$ or $C_{U,V}(x)=1$.

Finally, to see that $\top$ is preserved, observe that since $U,V$ are up-sets we have $\top\in U\cap V$. Then $C_{U,V}(\top) = 0$, which is the greatest element of $\utilde{\bf L}$. This proves the result.
\end{proof}

\begin{lemma}\label{lem:ilatticemorph2}
Let $\varphi\colon (X,\leq,\Q,D,\top,\tau)\to \utilde{\bf L}$ be a morphism of \textsf{pKS}. Then there exist clopen up-sets $U,V\subseteq X$ such that $\varphi=C_{U,V}$.
\end{lemma}

\begin{proof}
Put $U=\varphi^{-1}(\{0,1\})$ and $V=\varphi^{-1}(\{-1,0\})$. Then $U,V$ are clopen up-sets since they are the inverse images of clopen up-sets, and $C_{U,V}(x)=\varphi(x)$ for all $x\in X$.
\end{proof}

\begin{lemma}\label{lem:muhomomorphism}
Let $\varphi_1,\varphi_2\colon {\bf X}\to \utilde{\bf L}$ be pointed Kleene space morphisms with $\varphi_1 =C_{U_1,V_1}$ and $\varphi_2=C_{U_2,V_2}$. Then
\begin{enumerate}
\item $\neg\varphi_1 = C_{V_1,U_1}$.
\item $\varphi_1\meet\varphi_2 = C_{U_1\cap U_2, V_1\cup V_2}$, and
\item $\varphi_2\join\varphi_2 = C_{U_1\cup U_2, V_1\cap V_2}$,
\end{enumerate}
\end{lemma}

\begin{proof}
For (1), for $x\in X$ note that
\begin{align*}
\varphi_1(x)=1 &\iff C_{U_1,V_1}(x)=1\\
&\iff x\notin V_1\\
&\iff C_{V_1,U_1}(x)=-1
\end{align*}
Likewise, $\varphi_1(x)=-1$ if and only if $C_{V_1,U_1}(x)=1$. It follows from this that $\varphi_1(x)=0$ if and only if $C_{V_1,U_1}(x)=0$, and hence that $\neg\varphi_1 = C_{V_1,U_1}$.

For (2), observe that in $\utilde{\bf L}$ we have $a\meet b=1$ if and only if $a=1$ and $b=1$, and also $a\meet b = -1$ if and only if $a=-1$ or $b=-1$. Now if $x\in X$ then we have
\begin{align*}
\varphi_1(x)\meet\varphi_2(x) = 1 &\iff \varphi_1(x)=1 \text{ and } \varphi_2(x)=1\\
&\iff C_{U_1,V_1}(x)=1 \text{ and } C_{U_2,V_2}(x)=1\\
&\iff x\notin V_1 \text{ and } x\notin V_2\\
&\iff x\notin V_1\cup V_2\\
&\iff C_{U_1\cap U_2,V_1\cup V_2}(x) = 1
\end{align*}
Likewise,
\begin{align*}
\varphi_1(x)\meet\varphi_2(x) = -1 &\iff \varphi_1(x)=-1 \text{ or } \varphi_2(x)=-1\\
&\iff C_{U_1,V_1}(x)=-1 \text{ or } C_{U_2,V_2}(x)=-1\\
&\iff x\notin U_1 \text{ or } x\notin U_2\\
&\iff x\notin U_1\cap U_2\\
&\iff C_{U_1\cap U_2,V_1\cup V_2}(x) = -1
\end{align*}
It follows also that $\varphi_1(x)\meet\varphi_2(x) = 0$ if and only if $C_{U_1\cap U_2,V_1\cup V_2}(x) = 0$, which gives $\varphi_1\meet\varphi_2=C_{U_1\cap U_2,V_1\cup V_2}$.

(3) follows by an analogous argument.
\end{proof}

For a \textsf{bRS}-space ${\bf X}$, we define a function $\mu_{\bf X}\colon X^{*\bowtie}\to (X,\leq\cup\geq)^+$ by $\mu_{\bf X}(U,V) = C_{U,V}$. Provided that $\la U, V\ra\in X^{*\bowtie}$, it follows that $U\cup V = X$ and $U\cap V\subseteq D^\comp$. Moreover, if $(x,y)\in X\setminus U \times X\setminus V$, then $x\notin U$ and $y\notin V$. Since $U\cup V = X$, this gives that $y\in U$ and $x\in V$. Were it the case that $x\leq y$, then $V$ being upward-closed would give $y\in V$, a contradiction. Likewise, if $y\leq x$, then $U$ being upward-closed would give $x\in U$, another contradiction. It follows that $(X\setminus U\times X\setminus V)\cap (\leq \cup\geq) =\emptyset$, and Lemma \ref{lem:ilatticemorph1} thus gives that $\mu_{\bf X}$ is well-defined. 

\begin{lemma}\label{lem:sugiharaspaceiso}
Let ${\bf A}$ be a \textsf{bRS}-algebra. Then $({\bf A}_*,\subseteq\cup\supseteq)^+$ is isomorphic as an $i$-lattice to ${\bf A}^{\bowtie}$.
\end{lemma}

\begin{proof}
Lemma \ref{lem:bRStosugihara} asserts that $({\bf A}_*,\subseteq\cup\supseteq)$ is a pointed Kleene space, and therefore $({\bf A}_*,\subseteq\cup\supseteq)^+$ is a normal $i$-lattice.
By Lemma \ref{lem:dual5}, $({\bf A}_*)^*\cong {\bf A}$. It thus suffices to show that $({\bf A}_*,\subseteq\cup\supseteq)^+$ is isomorphic as an $i$-lattice to $({\bf A}_*)^{*\bowtie}$. Let $\mu = \mu_{{\bf A}_*}$. 

Lemma \ref{lem:muhomomorphism} shows that $\mu$ is an $i$-lattice homomorphism from $({\bf A}_*)^{*\bowtie}$ to $({\bf A}_*,\subseteq\cup\supseteq)^+$, and Lemma \ref{lem:ilatticemorph2} gives that $\mu$ is surjective. It remains only to show that $\mu$ is one-to-one, so suppose that $\la U_1,V_1\ra, \la U_2,V_2\ra\in ({\bf A}_*)^{*\bowtie}$ with $\mu(U_1,V_1)=\mu(U_2,V_2)$. Then $C_{U_1,V_1}=C_{U_2,V_2}$, so for all $x\in X$ we have
\begin{align*}
x\in U_1 &\iff C_{U_1,V_1}(x)\neq -1\\
&\iff C_{U_2,V_2}(x)\neq -1\\
&\iff x\in U_2
\end{align*}
Thus $U_1=U_2$. A similar argument shows that $V_1=V_2$, so $\la U_1,V_1\ra=\la U_2, V_2\ra$. This gives that $\mu$ is an $i$-lattice isomorphism.
\end{proof}

The stage is finally set to describe the duality for Sugihara monoids.

\begin{definition}
Given Sugihara spaces ${\bf X} = (X,\leq_X,\leq_X\cup\geq_X,D_X,\top_X,\tau_X)$ and ${\bf Y} = (Y,\leq_Y,\leq_Y\cup\geq_Y,D_Y,\top_Y,\tau_Y)$, a \textsf{bRSS}-morphism $\varphi$ from the \textsf{bRS}-space $(X,\leq_X,D,\top_X,\tau_X)$ to the \textsf{bRS}-space $(Y,\leq_Y,D_Y,\top_Y,\tau_Y)$ is called a \emph{Sugihara space morphism}. We denote the category of Sugihara spaces with Sugihara space morphisms by \textsf{pSS}, keeping with our earlier convention of naming categories of top-bounded spaces to make it clear that they are pointed.
\end{definition}

\begin{remark}
Observe that a morphism of \textsf{pSS} is automatically a morphism of \textsf{pKS} even though the preservation of the relation $\leq\cup\geq$ is not stipulated. A morphism always preserves the latter relation when it preserves $\leq$.
\end{remark}

Consider variants $(-)_+\colon\textsf{SM}\to\textsf{pSS}$ and $(-)^+\colon\textsf{pSS}\to\textsf{SM}$ of the functors from the Davey-Werner duality defined as follows. For an object ${\bf A}$ of \textsf{SM}, let ${\bf A}_+$ be the Davey-Werner dual of the $i$-lattice reduct of ${\bf A}$ as previously discussed. For a morphism $h\colon {\bf A}\to {\bf B}$ of \textsf{SM}, define $h_+\colon {\bf B}_+\to {\bf A}_+$ by precomposition $h_+(x) = x\circ h$ as usual.

On the other hand, for a Sugihara space ${\bf X} = (X,\leq,D,\top,\tau)$, let $X^+$ be the collection of pointed Kleene space morphisms from ${\bf X}$ to $\utilde{\bf L}$. Letting $\meet$, $\join$, and $\neg$ be the operations on $X^+$ inherited pointwise from the operations on $\underline{\bf L}$, the $i$-lattice $(X^+,\meet,\join,\neg)$ is the the Davey-Werner dual of ${\bf X}$. Define binary operations $\cdot$ and $\to$ for $\varphi_1=C_{U_1,V_1}$ and $\varphi_2=C_{U_2,V_2}$ maps in $X^+$ by
$$\varphi_1\cdot\varphi_2 = C_{\la U_1,V_1\ra \ntwt \la U_2,V_2\ra}$$
$$\varphi_1\to\varphi_2 = C_{\la U_1,V_1\ra \ntwi \la U_2,V_2\ra},$$
where $\ntwt$ and $\ntwi$ are the operations on the Sugihara monoid ${\bf X}^{*\bowtie}$ defined in Section \ref{sec:twist}. Then define ${\bf X}^+=(X^+,\meet,\join,\cdot,\to,C_{X,D^\comp},\neg)$. For a morphism $\varphi\colon {\bf X}\to {\bf Y}$ of \textsf{pSS}, define $\varphi_+\colon {\bf Y}^+\to{\bf X}^+$ by $\varphi(\alpha) = \alpha\circ\varphi$ as before.

\begin{remark}\label{rem:muiso}
With the above definitions, the map $\mux$ is actually a Sugihara monoid isomorphism. It is an $i$-lattice isomorphism by the proof of Lemma \ref{lem:sugiharaspaceiso}, and $\mux$ is a homomorphism with respect to $\cdot$, $\to$, and the monoid identity by the definition above.
\end{remark}

\begin{lemma}\label{lem:sugdual1}
Let ${\bf A} = (A,\meet,\join,\cdot,\to,t,\neg)$ be a Sugihara monoid. Then ${\bf A}_+$ is a Sugihara space.
\end{lemma}

\begin{proof}
Since $(A_+, \leq,A_0,\top,\tau_{\bf A})$ is a \textsf{bRS}-space by Lemma \ref{lem:aplusbRS}, by Lemma \ref{lem:bRStosugihara} is suffices to show that $Q_{\bf A}$ coincides with $\leq$-comparability.

By the Davey-Werner duality, $({\bf A}_+)^+$ is isomorphic to ${\bf A}$ as an $i$-lattice. Moreover, by the categorical equivalence developed in Section \ref{sec:twist}, $({\bf A}_{\bowtie})^{\bowtie}$ is isomorphic to ${\bf A}$ as a Sugihara monoid, hence in particular as an $i$-lattice. By Lemma \ref{lem:sugiharaspaceiso}, $({\bf A}_{\bowtie})^{\bowtie}$ is isomorphic as an $i$-lattice to $({\bf A}_{\bowtie*}, \subseteq\cup\supseteq)^+$. Thus ${\bf A}$ is isomorphic as an $i$-lattice to both $({\bf A}_*,\subseteq\cup\supseteq)^+$ and $({\bf A}_+)^+$. It follows that
$$({\bf A}_{\bowtie*},\subseteq\cup\supseteq)\cong (({\bf A}_*,\subseteq\cup\supseteq)^+)_+ \cong (({\bf A}_+)^+)_+\cong {\bf A}_+$$
as pointed Kleene spaces. Let $\varphi\colon {\bf A}_+\to (A_{\bowtie*},\subseteq\cup\supseteq)$ be a \textsf{pKS}-isomorphism. Then for $h,k\in A_+$,
\begin{align*}
hQ_{\bf A}k &\iff \varphi(h) \text{ and } \varphi(k) \text{ are } \text{$\subseteq$-comparable}\\
&\iff \varphi(h)\subseteq\varphi(k) \text{ or } \varphi(k)\subseteq\varphi(h)\\
&\iff h\leq k \text{ or } k\leq h\\
&\iff h \text{ and } k \text{ are $\leq$-comparable.}
\end{align*}
This proves that $Q_{\bf A}$ is the relation of $\leq$-comparability, and the result follows.
\end{proof}

\begin{lemma}\label{lem:sugdual2}
Let ${\bf X} = (X,\leq,D,\leq\cup\geq,\top,\tau)$ be a Sugihara space. Then ${\bf X}^+$ is a Sugihara monoid.
\end{lemma}

\begin{proof}
$(X,\leq,D,\top,\tau)$ is \textsf{bRS}-space by Lemma \ref{lem:sugiharatobRS}, and hence $(X,\leq,D,\top,\tau)^*$ is a \textsf{bRS}-algebra by the duality of Section \ref{sec:booldual}. It follows from Lemma \ref{lem:sugiharaspaceiso} that $(((X,\leq,D,\top,\tau)^*)_*,\subseteq\cup\supseteq)^+$ is isomorphic as an $i$-lattice to $(X,\leq,D,\top,\tau)^{*\bowtie}$. But $((X,\leq,D,\top,\tau)^*)_*\cong (X,\leq,D,\top,\tau)$ as a \textsf{bRS}-space, so it follows that $((X,\leq,D,\top,\tau),\leq\cup\geq)^+$ is isomorphic to $(X,\leq,D,\top,\tau)^{*\bowtie}$ as an $i$-lattice. Since the former structure is identical to the $i$-lattice reduct of ${\bf X}^+$, it follows that ${\bf X}^+$ is isomorphic as an $i$-lattice to the Sugihara monoid $(X,\leq,D,\top,\tau)^{*\bowtie}$. The definition of the operations $\to$ and $\cdot$ therefore makes the $i$-lattice reduct of ${\bf X}^+$ into a Sugihara monoid by transport of structure.
\end{proof}

\begin{lemma}\label{lem:hpluscomp}
Let ${\bf A}$ and ${\bf B}$ be Sugihara monoids and let $h\colon{\bf A}\to{\bf B}$ be a morphism in \textsf{SM}. Then $h_+ = \xia^{-1}\circ h_{\bowtie*}\circ\xib$.
\end{lemma}

\begin{proof}
Let $x\in B_+$ and $a\in A$. If $a\in A^-$, then $h_{\bowtie}(a) = h\restriction_{A^-}(a) = h(a)$ holds. Moreover, $h\restriction_{A^-}^{-1}[B^-]=A^-$. These facts give
\begin{align*}
a\in (\xia\circ h_+)(x) &\iff a\in \xia(x\circ h)\\
&\iff a\in (x\circ h)^{-1}[\{0,1\}]\cap A^-\\
&\iff (x\circ h)(a)\in\{0,1\}\text{ and }a\in A^-\\
&\iff (x\circ h_{\bowtie})(a)\in\{0,1\}\text{ and }a\in A^-\\
&\iff x(h\restriction_{A^-}(a))\in\{0,1\}\text{ and }a\in A^-\\
&\iff a\in h\restriction_{A^-}^{-1}[x^{-1}[\{0,1\}]]\cap A^-\\
&\iff a\in h\restriction_{A^-}^{-1}[x^{-1}[\{0,1\}]\cap B^-]\\
&\iff a\in h_{\bowtie*}(x^{-1}[\{0,1\}]\cap B^-)\\
&\iff a\in h_{\bowtie*}(\xib(x))\\
&\iff a\in (h_{\bowtie*}\circ\xib)(x)
\end{align*}
This shows that $\xia\circ h_+ \ h_{\bowtie*}\circ \xib$. Since $\xia$ is an isomorphism of \textsf{bRS}-spaces by Lemma \ref{lem:bRSisomorphism}, it has an inverse $\xia^{-1}$, and this yields $h_+=\xia^{-1}\circ h_{\bowtie*}\circ \xib$.
\end{proof}

\begin{corollary}\label{cor:sugdual3}
Let ${\bf A}$ and ${\bf B}$ be Sugihara monoids and let $h\colon{\bf A}\to{\bf B}$ be a morphism in \textsf{SM}. Then $h_+$ is a morphism of \textsf{pSS}.
\end{corollary}

\begin{proof}
Lemma \ref{lem:hpluscomp} shows that $h_+$ is the composition of \textsf{bRSS}-morphisms, which immediately gives the result.
\end{proof}

\begin{lemma}\label{lem:varphipluscomp}
Let ${\bf X}$ and ${\bf Y}$ be Sugihara spaces and let $\varphi\colon{\bf X}\to{\bf Y}$ be a morphism in \textsf{pSS}. Then $\varphi^+=\mux\circ\varphi^{*\bowtie}\circ\muy$.
\end{lemma}

\begin{proof}
Let $\la U,V\ra\in Y^{*\bowtie}$ and let $x\in X$. Then
\begin{align*}
((\mux\circ\varphi^{*\bowtie})(U,V))(x) &= \mux(\varphi^{*\bowtie}(U,V))(x)\\
&=\mux(\varphi^*(U),\varphi^*(V))(x)\\
&=\mux(\varphi^{-1}[U],\varphi^{-1}[V])(x)\\
&=C_{\varphi^{-1}[U],\varphi^{-1}[V]}(x)
\end{align*}
On the other hand,
\begin{align*}
((\varphi^+\circ\muy)(U,V))(x) &= \varphi^+(\muy(U,V))(x)\\
&= (C_{U,V}\circ\varphi)(x)\\
&= C_{U,V}(\varphi(x))
\end{align*}
Now note that $\varphi(x)\in U$ if and only if $x\in\varphi^{-1}[U]$, $\varphi(x)\in V$ if and only if $x\in\varphi^{-1}[V]$, and $\varphi(x)\in U\cap V$ if and only if $x\in\varphi^{-1}[U\cap V]=\varphi^{-1}[U]\cap\varphi^{-1}[V]$. This together with the definition of $C_{U,V}$ immediately gives that
$$C_{U,V}(\varphi(x)) = C_{\varphi^{-1}[U],\varphi^{-1}[V]}(x).$$
It follows that $\mux\circ\varphi^{*\bowtie}=\varphi^+\circ\muy$. Since $\muy$ is a Sugihara monoid isomorphism and hence invertible, it follows that $\varphi^+=\mux\circ\varphi^{*\bowtie}\circ\muy^{-1}$.
\end{proof}

\begin{corollary}\label{cor:sugdual4}
Let ${\bf X}$ and ${\bf Y}$ be Sugihara spaces and let $\varphi\colon{\bf A}\to{\bf B}$ be a morphism in \textsf{pSS}. Then $\varphi^+$ is a morphism of \textsf{SM}.
\end{corollary}

\begin{proof}
Lemma \ref{lem:varphipluscomp} gives that $\varphi^+$ is the composition of morphisms in \textsf{SM}, so $\varphi^+$ is a morphism of \textsf{SM}.
\end{proof}

\begin{lemma}\label{lem:sugdual5}
Let ${\bf A}$ be a Sugihara monoid. Then $({\bf A}_+)^+\cong {\bf A}$.
\end{lemma}

\begin{proof}
Note that ${\bf A}_+$ is isomorphic as a \textsf{bRS}-space to ${\bf A}_{\bowtie*}$ via $\xia$. On the other hand, $({\bf A}_{\bowtie*})^*\cong {\bf A}_{\bowtie}$ as \textsf{bRS}-algebras, and thus $({\bf A}_{\bowtie*})^{*\bowtie}\cong ({\bf A}_{\bowtie})^{\bowtie}\cong {\bf A}$ as Sugihara monoids by the equivalence of Section \ref{sec:twist}. The map $\mu_{{\bf A}_{\bowtie*}}$ is a Sugihara monoid isomorphism from $({\bf A}_{\bowtie*})^{*\bowtie}$ to $({\bf A}_{\bowtie*},\subseteq\cup\supseteq)^+$ by Remark \ref{rem:muiso}. It follows that $({\bf A}_+)^+\cong {\bf A}$ as Sugihara monoids as desired.
\end{proof}

\begin{lemma}\label{lem:sugdual6}
Let ${\bf X} = (X,\leq,\leq\cup\geq,D,\top,\tau)$ be a Sugihara space. Then $({\bf X}^+)_+\cong {\bf X}$.
\end{lemma}

\begin{proof}
${\bf X}^+$ is isomorphic as a Sugihara monoid to ${(X,\leq,D,\top,\tau)}^{*\bowtie}$ via $\mux$. Moreover, $({\bf X}^+)_+$ is isomorphic to $({\bf X}^+)_{\bowtie*}$ as a \textsf{bRS}-space via $\xi_{{\bf X}^+}$. It follows that as \textsf{bRS}-spaces, $({\bf X}^+)_+$ is isomorphic to $((X,\leq,D,\top,\tau)^{*\bowtie})_{\bowtie*}$. Since the latter space is isomorphic to $(X,\leq,D,\top,\tau)$ by the duality of Section \ref{sec:booldual} and the equivalence of Section \ref{sec:twist}, it follows that $({\bf X}^+)_+$ and $(X,\leq,D,\top,\tau)$ are isomorphic as \textsf{bRS}-spaces. The \textsf{bRSS}-isomorphism witnessing this is likewise a \textsf{pSS}-isomorphism between $({\bf X}^+)_+$ and $(X,\leq,\leq\cup\geq,D,\top,\tau)$, but the latter object is exactly ${\bf X}$. This gives the result.
\end{proof}

\begin{theorem}
\textsf{SM} is dually equivalent to \textsf{pSS}.
\end{theorem}

\begin{proof}
As the functoriality of $(-)_+$ and $(-)^+$ comes directly from the Davey-Werner duality, this follows immediately from Lemmas \ref{lem:sugdual1}, \ref{lem:sugdual2}, \ref{lem:hpluscomp},  \ref{lem:varphipluscomp}, \ref{lem:sugdual5}, \ref{lem:sugdual6}, and Corollaries \ref{cor:sugdual3} and \ref{cor:sugdual4}.
\end{proof}

Having obtained the duality between Sugihara monoids and Sugihara spaces, it remains to modify this duality for the bounded analogues of the Sugihara monoids.

\begin{definition}
A Kleene space $(X,\leq,Q,D,\tau)$ is called an \emph{unpointed Sugihara space} if
\begin{enumerate}
\item $(X,\leq,\tau)$ is an Esakia space,
\item $Q$ is the relation of comparability with respect to $\leq$, i.e., $Q=\leq\cup\geq$, and
\item $D$ is open.
\end{enumerate}
As in the case of Sugihara spaces, we sometimes simply say that $(X,\leq,D,\tau)$ is an unpointed Sugihara space, leaving $Q$ to be inferred.

A \textsf{bGS}-morphism between unpointed Sugihara spaces is called an \emph{unpointed Sugihara space morphism}, and we denote the category of unpointed Sugihara spaces with unpointed Sugihara space morphisms by \textsf{SS}.
\end{definition}

Repeating the argument above with necessary modifications for the addition of bounds, we obtain

\begin{corollary}
\textsf{SM}$_\bot$ is dually equivalent to \textsf{SS}.
\end{corollary}

\begin{figure}
\begin{center}
\begin{tikzpicture}
    \node[label=left:\tiny{$\top$}] at (0,0) {$\bullet$};
    \node[label=left:\tiny{$h_0$}] at (0,-0.5) {$\bullet$};
    \node[label=left:\tiny{$h_1$}] at (-0.5,-1)  {$\bullet$};
    \node[label=right:\tiny{$h_2$}] at (0.5,-1) {\textcircled{$\bullet$}};

    \draw (0,0) -- (0,-0.5);
    \draw (0,-0.5) -- (-0.5,-1);
    \draw (0,-0.5) -- (0.5,-1);
\end{tikzpicture}
\begin{tikzpicture}
    \node[label=left:\tiny{$h_0$}] at (0,-0.5) {$\bullet$};
    \node[label=left:\tiny{$h_1$}] at (-0.5,-1)  {$\bullet$};
    \node[label=right:\tiny{$h_2$}] at (0.5,-1) {\textcircled{$\bullet$}};

    \draw (0,-0.5) -- (-0.5,-1);
    \draw (0,-0.5) -- (0.5,-1);
\end{tikzpicture}
\end{center}
\caption{Hasse diagrams for ${\bf E}_+$ and $({{\bf E}_\bot})_+$}
\label{fig:HasseEdual}
\end{figure}

\begin{example}\label{ex:Edual}
Recall the Sugihara monoid ${\bf E}$ of Example \ref{parex}. The dual ${\bf E}_+$ of this algebra has Hasse diagram given in Figure \ref{fig:HasseEdual}, where the maps $\top,h_0,h_1,h_2$ are uniquely determined by $\top(a)=0$ for all $a\in E$, $h_0(a)=0$ for all $a$ except $\la 2,2\ra, \la -2, -2\ra$, $h_1(a)=0$ for $a=\la 0,1$ or $\ra,\la 0,-1,\ra$ and $h_2(a) = 1$ for all $a\in \upset \la -1,1\ra$ and $h_2(a)=-1$ for $a\in\downset\la 1,-1\ra$. Of these, only $h_2$ lies in the designated subset because its image does not contain $0$. If ${\bf E}_\bot$ is the expansion of ${\bf E}$ be universal lattice bounds, then its dual is given by the same Hasse diagram, but with the exclusion of the map $\top$ (this map is not a morphism in the bounded signature).
\end{example}

\subsection{Alternative formulations of the duality}

One of the greatest strengths of the Esakia duality, often lacked by natural dualities, is the pictorial character of the dual equivalence. The duality for Sugihara monoids rests on the representation of each Sugihara monoid as an algebra consisting of Kleene space morphisms, which is a less geometrically-intuitive construction. Here we recast this construction in more geometric terms in two distinct ways.

For an odd Sugihara monoid ${\bf A}$, we may realize its dual in terms of certain algebraic substructures that are ordered by containment. This representation in terms of \emph{convex prime subalgebras} has much of the pictorial flavor of the Esakia duality and its representation in terms of prime filters.

Unfortunately, when a Sugihara monoid is \emph{not} odd, the prime convex subalgebra representation proves inadequate. However, we may nevertheless obtain a more pictorial representation in terms of certain filters. In the next section, we will see that it also has points of contact with previous work on dualities for Sugihara monoids and other relevant algebras.

\begin{definition}
Let ${\bf A} = (A,\meet,\join,\cdot,\to,t,\neg)$ be an odd Sugihara monoid. A $(\meet,\join,t,\neg)$-subalgebra ${\bf C}$ of ${\bf A}$ is said to be a \emph{convex prime subalgebra} if for all $a,b,c\in A$,
\begin{enumerate}
\item If $a,c\in C$ and $a\leq b\leq c$, then $b\in C$, and
\item If $a\meet b\in C$, then $a\in C$ or $b\in C$.
\end{enumerate}
The collection of convex prime subalgebras of ${\bf A}$ is denoted $\mathcal{C}({\bf A})$.
\end{definition}

Note that if ${\bf C}$ is a convex prime subalgebra and $a\join b\in C$, then we have $\neg a\meet \neg b=\neg (a\join b)\in C$ as well. It follows that $\neg a\in C$ or $\neg b\in C$, so $a\in C$ or $b\in C$ by $\neg$-closure. Thus a convex prime subalgebra is prime with respect to $\join$ as well as $\meet$.

\begin{proposition}
Let ${\bf A}$ be an odd Sugihara monoid. Then ${\bf A}_+$ is order isomorphic to $(\mathcal{C}({\bf A}),\subseteq)$.
\end{proposition}

\begin{proof}
${\bf A}_+$ is order isomorphic to $\mathbb{A}_{\bowtie*}$ by Lemma \ref{lem:orderisomorphism}, so it suffices to show that $(\mathcal{C}({\bf A}),\subseteq)$ is order isomorphic to $(A_{{\bowtie}*},\subseteq)$. Define a map $\psi\colon\mathcal{C}({\bf A})\to A_{{\bowtie}*}$ for ${\bf C}\in\mathcal{C}({\bf A})$ by $\psi({\bf C})=C\cap A^-$. To see that $\psi({\bf C})$ is a filter, suppose that $a\in\psi({\bf C})$ and $b\in A^-$ with $a\leq b$. Then $b\in A^-$ and $a\leq b\leq t$, and by convexity $b\in C$. This gives that $b\in\psi({\bf C})$, so $\psi({\bf C})$ is upward closed.

For closure under meets, let $a,b\in\psi({\bf C})$. Then ${\bf C}$ being $\meet$-closed gives $a\meet b\in C$, and $a,b\leq t$ gives $a\meet b\leq t$. Thus $a\meet b\in \psi({\bf C})$. The primality of ${\bf C}$ gives $a\in C$ or $b\in C$.  

To see that $\psi({\bf C})$ is prime, let $a,b\in A^-$ with $a\join b\in\psi({\bf C})$. Then $a\join b\in C$ and $a\join b\leq t$. The latter gives $a\leq t$ and $b\leq t$,  so one of $a\in\psi({\bf C})$ or $b\in\psi({\bf C})$ must hold. This shows that $\psi({\bf C})$ is a prime filter of ${\bf A}_{\bowtie}$, and hence that $\psi$ is well-defined.

Obviously, $\psi$ is order-preserving since $C_1\subseteq C_2$ implies $C_1\cap A^-\subseteq C_2\cap A^-$ for any sets $C_1,C_2$. To see that $\psi$ is order-reflecting, suppose that ${\bf C}_1,{\bf C}_2\in\mathcal{C}({\bf A})$ with $\psi({\bf C}_1)\subseteq\psi({\bf C}_2)$. Let $a\in C_1$. Then $\neg a\in C_1$, and $a\meet t,\neg a\meet t\in\psi({\bf C}_1)$. This gives $a\meet t,\neg a\meet t\in\psi({\bf C}_2)$. Since $a\meet t,\neg a\meet t\in\psi({\bf C}_2)$, it follows that $a\meet t,\neg a\meet t\in C_2$. From $\neg a\meet t\in C_2$, it follows that $\neg(\neg a\meet t) = a\join\neg t\in C_2$. Because $a\meet t\leq a\leq a\join\neg t$, convexity gives $a\in C_2$. This yields $C_1\subseteq C_2$ as desired.

It remains only to show that $\psi$ is onto, so let $x\in A^-_*$. Let
$$\upset_{\bf A} x = \{a\in A : (\exists p \in x)(p\leq a)\},$$
$$\neg x = \{\neg a: a\in x\},$$
$$\downset_{\bf A} \neg x =\{a\in A : (\exists p\in \neg x)(a\leq p))\},\text{ and }$$
$$C=\upset_{\bf A} x\cap\downset_{\bf A}\neg x.$$
We claim that $C$ is the universe of a convex prime subalgebra ${\bf C}$, and that $\psi({\bf C}) = x$.

First, note that since $x\in A^-_*$ we have that $t\in x$, so $t\in C$. If $a\in C$, then there exists $p,q\in x$ such that $p\leq a\leq \neg q$. Then $q\leq \neg a\leq \neg p$, so $\neg a\in C$.

Second, suppose that $a,b\in C$. Then there exists $p_1,p_2,q_1,q_2\in x$ such that $p_1\leq a\leq \neg q_1$ and $p_2\leq b\leq \neg q_2$. This gives
$$p_1\meet p_2\leq a\meet b\leq \neg q_1\meet\neg q_2 = \neg (q_1\join q_2).$$
Since $x$ is a filter, $p_1\meet p_2,q_1\join q_2\in x$. This gives $a\meet b\in C$. On the other hand, $p_1\join p_2\leq a\join b\leq \neg q_1\join\neg q_2 = \neg (q_1\meet q_2)$ gives that $a\join b\in C$. Since $t\in x$, $t\leq t\leq \neg t = t$ gives $t\in C$, and this shows that $C$ is a $(\meet,\join,\neg,t)$-subalgebra.

To see that $C$ is convex, suppose that $a,c\in C$ and $b\in A$ with $a\leq b \leq c$. Since $a,c\in C$, there exists $p_1,p_2,q_1,q_2\in x$ with $p_1\leq a\leq\neg q_1$ and $p_2\leq c\leq\neg q_2$. This gives $p_1\leq a\leq b\leq c\leq \neg q_2$, so $b\in C$ as well. Thus ${\bf C}$ is a convex prime subalgebra.

To see that $\psi({\bf C}) = x$, suppose first that $a\in \psi({\bf C}) = C\cap A^-$. Then there exists $p,q\in x$ with $p\leq a\leq \neg q$, and $a\in A^-$. Since $x$ is upward-closed, $p\leq a$ and $p\in x$ yields $a\in x$. Hence $\psi({\bf C})\subseteq x$. On the other hand, if $a\in x$, then $a\leq a \leq t = \neg t$ gives that $a\in\psi({\bf C})$ as desired. This proves the result.
\end{proof}

Given a Sugihara monoid (or bounded Sugihara monoid) ${\bf A}$ with monoid identity $t$, define
$$I({\bf A})=\{x\in A_* : t\in A\},$$
where $A_*$ is the collection of generalized prime filters (i.e., the collection of prime filters along with $A$ itself) if ${\bf A}$ is a Sugihara monoid and $A_*$ is the collection of prime filters (excluding $A$ itself) if ${\bf A}$ has lattice bounds in its signature. By considering $I({\bf A})$, we may obtain a more pictorial representation of the dual of an arbitrary Sugihara monoid.

\begin{proposition}\label{prop:iofa}
Let ${\bf A}$ be a Sugihara monoid (with or without designated bounds). Then ${\bf A}_+$ is order isomorphic to $(I({\bf A}),\subseteq)$.
\end{proposition}

\begin{proof}
Define $\psi_{\bf A}\colon {\bf A}_+\to I({\bf A})$ by $\psi(h) = h^{-1}[\{0,1\}]$. Since $\{0,1\}$ is a prime filter of $\underline{\bf L}$ and $h$ is a $(\meet,\join,\neg)$-morphism (or $(\meet,\join,\neg,\bot,\top)$-morphism, as applicable) we have $\psi_{\bf A}(h)\in A_*$. Moreover, since $h(t)\in\{0,1\}$ always holds, $t\in h^{-1}[\{0,1\}]$ for each $h\in A_+$. This shows that $\psi_{\bf A}$ is well-defined.

$\psi_{\bf A}$ is order-preserving by the same argument offered in the proof of Lemma \ref{lem:orderpreserving}. To see that $\psi_{\bf A}$ is order-reflecting, let $h_1,h_2\in A_+$ with $\psi_{\bf A}(h_1)\subseteq\psi_{\bf A}(h_2)$. Were it the case that $h_1\not\leq h_2$, then there exists $a\in A$ such that $h_2(a)=-1$ and $h_1(a)\neq -1$, or else $h_2(a)=1$ and $h_1(a)\neq 1$.

In the first case, we have that $h_1(a)\in\{0,1\}$. Then $a\in\psi_{\bf A}(h_1)\subseteq\psi_{\bf A}(h_2)$, so $h_2(a)\in\{0,1\}$, a contradiction. In the second case, $h_1(a)\in\{-1,0\}$, so $h_1(\neg a)\in\{0,1\}$. Then $h_2(\neg a)\in\{0,1\}$, but this contradicts $h_2(a) = 1$. It follows that $h_1\leq h_2$, giving that $\psi_{\bf A}$ is order-reflecting.

Finally, to see that $\psi_{\bf A}$ is onto, let $x\in I({\bf A})$ and set $\neg x=\{\neg a : a\in x\}$. Observe that $t\in x$ and the identity $t\leq a\join\neg a$ yields that $a\join\neg a\in x$ for all $a\in A$, whence by primality $a\in x$ or $\neg a\in x$. This gives that $a\in x$ or $a\in\neg x$, and therefore each $a\in A$ is contained in exactly  one of the disjoint sets $x\setminus\neg x$, $x\cap\neg x$, or $\neg x\setminus x$. We may therefore define a map $h\colon A\to \{-1,0,1\}$ by
\[ h(a) = \begin{cases} 
      1 & a\in x\setminus\neg x\\
      0 & a\in x\cap\neg x \\
     -1 & a\in \neg x\setminus x\\
   \end{cases}
\]
By checking cases, one may show that $h$ is a morphism with respect to $\meet$,$\join$,$\neg$, and the lattice bounds (when applicable). This shows that $h\in A_+$. Moreover, $\psi_{\bf A}(h)=h^{-1}[\{0,1\}]=h^{-1}(0)\cup h^{-1}(1)=(x\setminus\neg x)\cup (x\cap\neg x)=x$. Thus $\psi_{\bf A}$ is onto, and hence an order isomorphism.

\end{proof}

\section{The reflection construction}\label{sec:refcon}

The covariant equivalence of Section \ref{sec:twist} provides an entirely algebraic treatment of the relationship between bRS-algebras and Sugihara monoids as well as their bounded analogues. However, the complexity of the construction of a Sugihara monoid from a bRS-algebra is a significant obstacle to understanding the role of twist products in such contexts. Here we exploit the duality of Section \ref{sec:sugidual} to obtain a dramatically simpler presentation of this construction. This amounts to transporting the construction of Section \ref{sec:twist} across the duality to obtain its analogue on dual spaces, which we will call the \emph{reflection construction}. We also obtain a dual presentation of the enriched negative cone construction, giving a complete picture of how the algebraic work of Section \ref{sec:twist} presents on dual spaces. As an added benefit, this illuminates the connection between the duality developed in Section \ref{sec:sugidual} and previous work on duality for Sugihara monoids due to Urquhart \cite{Urquhart}. Because Urquhart presented his duality only for bounded algebras, throughout this section we work with bounded Sugihara monoids.

After introducing some background on Urquhart duality in Section \ref{sec:urqdual}, we construct the dual of the enriched negative cone construction in Section \ref{sec:dualnegcone}, culminating in the definition of the functor in Definition \ref{def:lowerbowtie}. Then in Section \ref{sec:dualtwist}, we construct the dual of the twist product variant from Section \ref{sec:twist}, giving its definition in Definition \ref{def:upperbowtie}. Finally, in Section \ref{sec:equiv} we show that these two constructions give an equivalence of categories between (unpointed) Sugihara spaces and the dual spaces described in the Urquhart duality.

\subsection{The Urquhart duality}\label{sec:urqdual}
In order to articulate the aforementioned constructions, we first recall Urquhart's duality for relevant algebras \cite{Urquhart}. Consider a Priestley space $(X,\leq,\tau)$ and a ternary relation $R$ on $X$. For $x,y\in X$, define $x\odot y = \{z\in X : Rxyz\}$. For subsets $U,V\subseteq X$, define
$$U\ntwt V = \{z\in X : (\exists x,y\in X)(Rxyz\text{ and } x\in U\text{ and } y\in V)\},\text{ and}$$
$$U\ntwi V  = \{x\in X : (\forall y,z\in X)((Rxyz\text{ and }y\in U)\text{ implies } z\in V)\}.$$
Note that here we have repurposed the symbols $\ntwt$ and $\ntwi$ of Section \ref{sec:twist} for ease of notation; context allows us to distinguish between these meanings without difficulty.

Urquhart's duality concerns itself with the category of structured topological spaces and morphisms defined as follows.
\begin{definition}
Let ${\bf X} = (X,\leq, R,\;',I,\tau)$ be a structure such that $(X,\leq,\tau)$ is a Priestley space, $R$ is a ternary relation on $X$, $'\colon X\to X$ is a function, and $I\subseteq X$. We say that ${\bf X}$ is a \emph{relevant space} if it satisfies the following conditions.
\begin{enumerate}
\item Whenever $U$ and $V$ are clopen up-sets of ${\bf X}$, so are the sets $U\ntwt V$ and $U\ntwi V$,
\item If $Rx_1y_1z_1$, $x_2\leq x_1$, $y_2\leq y_1$, and $z_1\leq z_2$, then $Rx_2y_2z_2$,
\item For all $x,y,z\in X$, if it is not that case that $Rxyz$, then there are clopen up-sets $U,V$ of $X$ such that $x\in U$, $y\in V$, and $z\notin U\ntwt V$,
\item The map $'$ is continuous and antitone,
\item $I$ is a clopen up-set and for all $y,z\in X$, $y\leq z$ if and only if there exists $x\in I$ with $Rxyz$.
\end{enumerate}
Given relevant spaces ${\bf X} = (X,\leq_{\bf X}, R_{\bf X}, \;', I_{\bf X},\tau_{\bf X})$ and ${\bf Y} = (Y,\leq_{\bf Y},R_{\bf Y},\;', I_{\bf Y},\tau_{\bf Y})$, a function $\varphi\colon {\bf X}\to {\bf Y}$ is called an \emph{relevant map} if
\begin{enumerate}
\item $\varphi$ is continuous and isotone,
\item If $R_{\bf X}xyx$, then $R_{\bf Y}\varphi(x)\varphi(y)\varphi(z)$,
\item If $R_{\bf Y}xy\varphi(z)$, then there exists $u,v\in X$ such that $R_{\bf X}uvz$, $x\leq\varphi(u)$, and $y\leq\varphi(v)$.
\item If $R_{\bf Y}\varphi(x)yz$, then there exists $u,v\in X$ such that $R_{\bf X}xuv$, $y\leq\varphi(u)$, and $\varphi(v)\leq z$,
\item $\varphi(x')=\varphi(x)'$, and
\item $\varphi^{-1}[I_{\bf Y}]=I_{\bf X}$.
\end{enumerate}
\end{definition}
The relevant algebras for which Urquhart articulated his duality include the bounded Sugihara monoids as a subvariety. Indeed, bounded Sugihara monoids are precisely the idempotent De Morgan monoids. Following Urquhart's correspondence theory for relevant spaces (see \cite[Theorem 4.1]{Urquhart} and the comments thereafter), the relevant spaces ${\bf X}$ corresponding to bounded Sugihara monoids are axiomatized by the conditions that for all $x,y,z\in X$,
\begin{enumerate}
\item $x\odot y = y\odot x$,
\item $x\odot (y\odot z) = (x\odot y)\odot z$,
\item $x = x\odot x$,
\item $x'' = x$, and
\item $z\in x\odot y$ implies $y'\in x\odot z'$.
\end{enumerate}
We call the relevant spaces satisfying the above conditions \emph{Sugihara relevant spaces}, and denote the category of Sugihara relevant spaces with relevant maps by \textsf{SRS}. Specialized to the present inquiry, the main result of \cite{Urquhart} is the following.
\begin{theorem}\label{thm:urquhartdual}
\textsf{SM}$_\bot$ is dually equivalent to \textsf{SRS}.
\end{theorem}
Given a bounded Sugihara monoid ${\bf A} = (A,\meet,\join,\cdot,\to,t,\neg,\bot,\top)$, define for $x,y\in A_*\cup\{A\}$ the complex product
$$x\cdot y = \{c\in A : (\exists a\in x,\exists b\in y)(a\cdot b\leq c)\}.$$
Let $R$ be the ternary relation on $A_*$ given by $Rxyz$ if and only if $x\cdot y\subseteq z$, let $x' = \{a\in A : \neg a\notin x\}$, and let $I({\bf A}) = \{x\in A^* : t\in x\}$ as in Section \ref{sec:sugidual}. Then we denote by ${\bf A}_*$ the Sugihara relevant space $((A,\meet,\join,\bot,\top)_*, R, \;', I({\bf A}))$.

On the other hand, for a Sugihara relevant space ${\bf X} = (X,\leq,R,\;',I,\tau)$, let ${\bf X}^*$ be the bounded Sugihara monoid $((X,\leq,\tau)^*, \ntwt,\ntwi,I,\neg)$, where $\neg$ is given by $\neg U = \{x\in X : x'\notin U\}$. When extended to morphisms in the familiar way, the functors $(-)_*$ and $(-)^*$ witness the equivalence between \textsf{SM}$_{\bot}$ and \textsf{SRS} of the Urquhart duality.

In the next three sections, we introduce functors $(-)_{\bowtie}\colon\textsf{SRS}\to\textsf{SS}$ and $(-)^{\bowtie}\colon\textsf{SS}\to\textsf{SRS}$, named in analogy to their duals in Section \ref{sec:twist}, that give an equivalence of categories between \textsf{SRS} and \textsf{SS}. The construction of each of these functors requires some technical results, which we turn to presently. We start with the functor $(-)_{\bowtie}$.

\subsection{Dual enriched negative cones}\label{sec:dualnegcone}

For a bounded Sugihara monoid ${\bf A}$, recall that $I({\bf A}) = \{x\in A_* : t\in x\}$. Recall also that $\psi_{\bf A}\colon {\bf A_+}\to I({\bf A})$ defined by $\psi_{\bf A}(h) = h^{-1}[\{0,1\}]$ is an order isomorphism between ${\bf A}_+$ and $(I({\bf A}),\subseteq)$ from the proof of Proposition \ref{prop:iofa}. We show that $\psi_{\bf A}$ preserves much more structure.
\begin{lemma}
When $I({\bf A})$ is endowed with the topology inherited as a subspace of ${\bf A}_*$, $\psi_{\bf A}$ is continuous.
\end{lemma}

\begin{proof}
It suffices to check that the inverse image of a basis element is open. Let $a\in A$. The basis elements of $I({\bf A})$ are of the form $\sigma(a) = \{x\in I({\bf A}) : a\in x\}$ and $\sigma(a)^\comp = \{x\in I({\bf A}) : a\notin x\}$. Observe that
\begin{align*}
\psi_{\bf A}^{-1}[\sigma(a)] &= \{h\in A_+ : \psi_{\bf A}(h)\in\sigma(a)\}\\
&=\{h\in A_+ : a\in h^{-1}[\{0,1\}]\}\\
&=\{h\in A_+ : h(a)\in\{0,1\}\}\\
&=\{h\in A_+ : h(a)=0\}\cup \{h\in A_+ : h(a)=1\}
\end{align*}
The above are basis elements of ${\bf A}_+$. Moreover,
\begin{align*}
\psi_{\bf A}^{-1}[\sigma(a)^\comp] &= \{h\in A_+ : \psi_{\bf A}(h)\in\sigma(a)^\comp\}\\
&=\{h\in A_+ : a\notin h^{-1}[\{0,1\}]\}\\
&=\{h\in A_+ : h(a)\notin\{0,1\}\}\\
&=\{h\in A_+ : h(a)=-1\}
\end{align*}
The above is also a basis element, so this gives the result.
\end{proof}

\begin{lemma}
$\psi_{\bf A}$ is a homeomorphism.
\end{lemma}

\begin{proof}
$I({\bf A})$ is a subspace of a Hausdorff space, hence is Hausdorff. ${\bf A}_+$ is compact since it is a Priestley space. This gives that $\psi_{\bf A}$ is a continuous bijection from a compact space to a Hausdorff space, hence a homeomorphism.
\end{proof}

The fact that $\psi_{\bf A}$ is an order isomorphism and a homeomorphism allows us to obtain the following results.

\begin{lemma}\label{lem:IAPriestley}
$I({\bf A})$ is a Priestley space.
\end{lemma}

\begin{proof}
$I({\bf A})$ is compact since $\psi_{\bf A}$ is a homeomorphism. Let $x,y\in I({\bf A})$ with $x\not\subseteq y$. Then since $\psi_{\bf A}$ is an order isomorphism we have $\psi^{-1}_{\bf A}(x)\not\leq\psi^{-1}_{\bf A}(y)$, and since ${\bf A}_+$ is a Priestley space there exists a clopen up-set $U\subseteq A_+$ with $\psi^{-1}_{\bf A}(x)\in U$ and $\psi^{-1}_{\bf A}(y)\notin U$. Then $\psi_{\bf A}[U]$ is a clopen up-set of $I({\bf A})$ and $x\in \psi_{\bf A}[U]$ and $y\notin \psi_{\bf A}[U]$, showing that $I({\bf A})$ is a Priestley space.
\end{proof}

\begin{lemma}\label{lem:IAEsakia}
$I({\bf A})$ is an Esakia space. 
\end{lemma}

\begin{proof}
$\psi_{\bf A}$ is an order isomorphism and a homeomorphism, hence an isomorphism of Priestley spaces. Since $I({\bf A})$ is a Priestley space that is isomorphic to the Esakia space ${\bf A}_+$, it follows that $I({\bf A})$ is an Esakia space too.
\end{proof}
The following collects some information about the operation $'$ of the Urquhart dual of a bounded Sugihara monoid, and is fundamental to the constructions that follow.
\begin{lemma}\label{lem:orderprime}
Let ${\bf A}$ be a bounded Sugihara monoid. Then for all $x\in A_*$,
\begin{enumerate}
\item $x\in I({\bf A})$ or $x'\in I({\bf A})$.
\item $x\subseteq x'$ or $x'\subseteq x$.
\item The larger of $x$ and $x'$ lies in $I({\bf A})$.
\item The following are equivalent.
\begin{enumerate}
\item $x=x'$,
\item $t\in x$ and $\neg t\notin x$,
\item $x,x'\in I({\bf A})$.
\end{enumerate}
\end{enumerate}
\end{lemma}

\begin{proof}
For (1), suppose $t\notin x$. Then as $\neg t\leq t$, it follows that $\neg t\notin x$ as well. Thus $t\in x'$.

For (2), by (1) we may suppose without loss of generality that $t\in x'$. Let $a\in x$. If $a\notin x'$, then $\neg(\neg a)\notin x'$ and hence $\neg a\in x$. Then $a,\neg a \in x$, so $a\meet\neg a\leq t$ gives $t\in x$, a contradiction. It follows that $x\subseteq x'$.

Note that (3) follows immediately from (1) and (2).

For (4), suppose first that $x=x'$. If $t\notin x$, then $t=\neg\neg t\notin x$, so $\neg t\in x'$. Then $\neg t\in x$. But $\neg t\leq t$ gives $t\in x$, so this is impossible. It follows that $t\in x$. Then $t\in x'$ as well. If $\neg t\in x$, then $\neg t\in x'$ as well and this would give $\neg\neg t\notin x$. But this contradicts $t\in x$. Hence $t\in x$ and $\neg t\notin x$.

Next, suppose that $t\in x$ and $\neg t\notin x$. The latter gives that $t\in x'$, so it follows immediately that $x,x'\in I({\bf A})$.

Finally, suppose that $x,x'\in I({\bf A})$. Then $t\in x,x'$, so $t\in x$ and $\neg t\notin x$. Let $a\in x$. If $\neg a\in x$, then $a,\neg a\in x$ implies $a\meet\neg a\leq \neg t\in x$, a contradiction. Hence $\neg a\notin x$, so $a\in x'$ and $x\subseteq x'$. On the other hand, let $a\in x'$. Then $\neg a\notin x$. But $a\join\neg a\geq t$ and $t\in x$ gives $a\join\neg a\in x$, so $a\in x$ by primality. Thus $x'\subseteq x$. This shows that $x=x'$, which completes the proof of the equivalence.
\end{proof}

Let ${\bf A}$ be an unbounded Sugihara monoid, and ${\bf A}_* = (A_*, \subseteq, R, \;', I({\bf A}), \tau)$ its Urquhart dual. Set $D=\{x\in A_* : x = x'\}$, and $\tau_{\bowtie}$ the topology on $I({\bf A})$ induced as a subspace of ${\bf A}_*$. We then obtain the following.

\begin{lemma}\label{lem:IASugi}
$(I({\bf A}),\subseteq, D, \tau_{\bowtie})$ is an unpointed Sugihara space.
\end{lemma}

\begin{proof}
Lemma \ref{lem:IAEsakia} shows that $(I({\bf A}),\subseteq,\tau_{\bowtie})$ is an Esakia space. It thus suffices to show that $(I({\bf A}),\subseteq)$ is a forest and $D$ is a clopen subset of $\subseteq$-minimal elements. The former condition is clear since $\psi_{\bf A}$ is an order isomorphism and ${\bf A}_+$ is a forest. That $D\subseteq I({\bf A})$ follows from Lemma \ref{lem:orderprime}.

To show that each $x\in D$ is minimal, let $y\in I({\bf A})$ with $y\subseteq x=x'$. Then $t\in y$, and $'$ being antitone gives $x=x'\subseteq y'$, so $t\in y'$ as well. It follows that $t\in y,y'$, so $y=y'$ by Lemma \ref{lem:orderprime}. But this gives $x\subseteq y\subseteq x$, so $x=y$. It follows that $D$ is a collection of minimal elements in $I({\bf A})$.

To see that $D$ is clopen, note that $x\in D$ iff $x=x'$ iff $t\in x$ and $\neg t\notin x$ iff $x\in\sigma(t)\cap\sigma(\neg t)^\comp$, so $D=\sigma(t)\cap\sigma(\neg t)^\comp$ is a clopen subset of ${\bf A}_*$, and so too of the subspace $I({\bf A})$.
\end{proof}

\begin{remark}\label{rem:IAiso}
An easy argument shows that if $h\in A_+$ has its image contained in $\{-1,1\}$, then setting $x=\psi_{\bf A}(h)$ gives $x=x'$. On the other hand, if $x=x'\in A_*$, then by $\psi_{\bf A}$ being onto there exists $h\in A_+$ such that $x=\psi_{\bf A}(h)$. Were there $a\in A$ with $h(a)=0$, we would have $h(\neg a)=0$ as well. Moreover, this would give that $a,\neg a\in \psi_{\bf A}(h)=x=x'$. But $a\in x'$ implies that $\neg a\notin x$, a contradiction. This shows that the image of $h$ must lie in $\{-1,1\}$ and hence
$$\psi_{\bf A}[\{h\in A_+ : (\forall a\in A)(h(a)\in\{-1,1\}\}]=\{x\in A_* : x=x'\}$$
Thus $\psi_{\bf A}$ preserves the designated subset $D$, and because $\psi_{\bf A}$ is a bijection this is sufficient to guarantee that it is actually an isomorphism in the category of unpointed Sugihara spaces.
\end{remark}

We may finally describe the dual enriched negative cone functor $(-)_{\bowtie}$.

\begin{definition}\label{def:lowerbowtie}
For a Sugihara relevant space ${\bf X} = (X,\leq,R,\;^\prime,I,\tau)$, let $X_{\bowtie}=I$, $D=\{x\in X : x=x'\}$, and $\tau_{\bowtie}$ be the topology on $X_{\bowtie}$ inherited as a subspace of ${\bf X}$. Define ${\bf X}_{\bowtie}=(X_{\bowtie},\leq,D,\tau_{\bowtie})$. For a morphism $\varphi\colon {\bf X}\to {\bf Y}$ of \textsf{SRS}, define $\varphi_{\bowtie}=\varphi\restriction_{X_{\bowtie}}$.
\end{definition}

The following shows that this definition makes sense on the level of objects. We put off verifying that the definition makes sense for morphisms until Section \ref{sec:equiv}.

\begin{lemma}
Let ${\bf X} = (X,\leq,R,\;',I,\tau)$ be a Sugihara relevant space. Then ${\bf X}_{\bowtie}$ is an unpointed Sugihara space.
\end{lemma}

\begin{proof}
By the Urquhart duality, there exists a bounded Sugihara monoid ${\bf A}$ such that ${\bf A}_*\cong {\bf X}$ as relevant spaces. There hence exists a relevant space isomorphism $\varphi\colon {\bf A}_*\to {\bf X}$. In particular, $\varphi[I({\bf A})]=I$, and the restriction $\varphi\restriction_{I({\bf A})}$ is a continuous order isomorphism. Since $I$ is a subspace of a Hausdorff space, it is itself Hausdorff. Since $I({\bf A})$ is compact by Lemma \ref{lem:IAPriestley}, this gives that $\varphi\restriction_{I({\bf A})}$ is a homeomorphism as well. It follows as before that $I$ is a Priestley space isomorphic to $I({\bf A})$, and hence an Esakia space. That $(I,\leq)$ is a forest also follows from this order isomorphism and the fact that $(I({\bf A}),\subseteq)$ is a forest by Lemma \ref{lem:IASugi}.

It remains only to show that $D\subseteq I$ and that $D$ is a clopen collection of minimal elements. To this end, let $y\in D$. Then since $\varphi$ is a bijection, there exists $x\in A_*$ such that $\varphi(x)=y$. Since $y\in D$, by definition we have $y=y'$. This yields $y'=\varphi(x)$, and as $\varphi$ preserves $'$ this shows that $y=\varphi(x')=\varphi(x)$. It follows from the injectivity of $\varphi$ that $x'=x$, whence $D\subseteq\varphi[\{x\in A_* : x=x'\}]$. Because $\{x\in A_* : x=x'\}\subseteq I({\bf A})$ by Lemma \ref{lem:orderprime}(4), we obtain $D\subseteq I$ as $\varphi[I({\bf A})]=I$. Also, if $x=x'$ in $A_*$, then $\varphi(x)=\varphi(x')=\varphi(x)'$ gives that $\varphi(x)\in D$. This gives that $\varphi[\{x\in A_* : x=x'\}]\subseteq D$, whence that $\varphi[\{x\in A_* : x=x'\}] = D$. As $\{x\in A_* : x=x'\}$ is clopen collection of minimal elements by Lemma \ref{lem:IASugi}, we obtain that $D$ is a clopen collection of minimal elements of $I$ since $\varphi$ is an order isomorphism and homeomorphism. It follows that ${\bf X}_{\bowtie} = (I,\leq,D,\tau_{\bowtie})$ is an unpointed Sugihara space as desired.
\end{proof}

\subsection{Dual twist products}\label{sec:dualtwist}

We next turn our attention to the functor $(-)^{\bowtie}$. Recall that if ${\bf A}$ is a bounded Sugihara monoid and $x,y\in A_*\cup\{A\}$, we defined
$$x\cdot y = \{c\in A : (\exists a\in x,\exists b\in y)(a\cdot b\leq c)\}$$
With this definition, we have the following.

\begin{lemma}
Let ${\bf A}$ be a bounded Sugihara monoid, and let $x,y\in A_*\cup\{A\}$. Then $x\cdot y\in A_*\cup\{A\}$.
\end{lemma}

\begin{proof}
That $x\cdot y$ is a filter is proven in \cite[Lemma 2.1]{Urquhart}, so it suffices to show that $x\cdot y$ is prime or improper. Let $a,b\in A$ with $a\join b\in x\cdot y$. Then there exists $c\in x$, $d\in y$ such that $cd\leq a\join b$. By residuation, we obtain that $d\leq c\to (a\join b)$. But as ${\bf A}$ is semilinear, Proposition \ref{prop2}(3) gives $c\to (a\join b)=(c\to a)\join (c\to b)$. Hence $d\leq (c\to a)\join (c\to b)$, and as $y$ is upward-closed this yields $(c\to a)\join (c\to b)\in y$. Since $y$ is a prime or improper, this shows that $c\to a\in y$ or $c\to b\in y$, whence either $c(c\to a)\in x\cdot y$ or $c(c\to b)\in x\cdot y$. But $c(c\to a)\leq a$ and $c(c\to b)\leq b$, so $x\cdot y$ being upward-closed gives that in either case one of $a\in x\cdot y$ or $b\in x\cdot y$ holds, proving the lemma.
\end{proof}

The above shows that $\cdot$ is a \emph{bona fide} operation on $A_*\cup\{A\}$. A thorough understanding of this operation proves essential to our construction of $(-)^{\bowtie}$, and toward this purpose we prove several technical claims about this operation.

\begin{lemma}\label{lem:timestech}
Let ${\bf A}$ be a bounded Sugihara monoid and let $x,y,z\in A_*\cup \{A\}$. Then the following hold.
\begin{enumerate}
\item $x\cdot y = y\cdot x$.
\item $y\in I({\bf A})$ implies $x\subseteq x\cdot y$.
\item $x\cdot x = x$.
\item $ab\in x$ implies $a\in x$ or $b\in x$.
\item $x\subseteq y$ implies $x\cdot z\subseteq y\cdot z$.
\item $a,b\in x$ implies $ab\in x$.
\end{enumerate}
\end{lemma}

\begin{proof}
For (1), let $c\in x\cdot y$. Then there are $a\in x$, $b\in y$ with $ab\leq c$. But then $ba\leq c$ gives $c\in y\cdot x$. The reverse inclusion follows in the same way.

For (2), let $a\in x$. Then $a=at\in x\cdot y$, so $x\subseteq x\cdot y$.

For (3), let $a\in x$. Then $a=a\cdot a\in x\cdot x$, so $x\subseteq x\cdot x$. On the other hand, if $c\in x\cdot x$ then there exist $a,b\in x$ with $ab\leq c$. Then $a\leq b\to c$ gives $b\to c\in x$ by upward closure, so $b\meet (b\to c)\leq b(b\to c)\leq c$ gives $c\in x$.

For (4), this follows from the primality of $x$ and the fact that $ab\leq a\join b$ holds in every Sugihara monoid.

For (5), let $c\in x\cdot z$. Then there exists $a\in x$, $b\in z$ with $ab\leq c$, so $a\leq b\to c$ gives $b\to c\in x$ by upward closure. Then $b\to c\in y$, so $b(b\to c)\in z\cdot y$ gives $c\in y\cdot z$. Hence $x\cdot z\subseteq y\cdot z$.

For (6), this follows from the identity $a\meet b\leq ab$, which holds in every Sugihara monoid.
\end{proof}

\begin{lemma}\label{lem:xprimemeet}
Let $x\in A_*\cup\{A\}$. Then $x\meet x'$ exists, and $x\meet x'=x\cdot x'$.
\end{lemma}

\begin{proof}
Either $x\subseteq x'$ or $x'\subseteq x$ by Lemma \ref{lem:orderprime}(2), so the meet of $x$ and $x'$ certainly exists and without loss of generality we assume $x'\subseteq x$. Then $t\in x$, so $x'\subseteq x'\cdot x$ by Lemma \ref{lem:timestech}(2). On the other hand, let $c\in x'\cdot x$. Then there exists $a\in x'$ and $b\in x$ with $ab\leq c$. This holds iff $a\cdot\neg c\leq\neg b$. If $\neg c\in x$, then $b\cdot\neg c\leq\neg a$ would give $\neg a\in x$, a contradiction to $a\in x'$. Hence $\neg c\notin x$, so $c\in x'$. Thus $x'\cdot x\subseteq x'$, giving $x\cdot x'=x'=x\meet x'$.
\end{proof}

\begin{lemma}\label{lem:IofAtimes}
If $x,y\in I({\bf A})$, then $x\join y$ exists and $x\join y=x\cdot y$.
\end{lemma}

\begin{proof}
$t\in x,y$ implies $x,y\subseteq x\cdot y$. On the other hand, let $z\in A_*\cup\{A\}$ with $x,y\subseteq z$. Then by monotonicity $x\cdot y\subseteq z\cdot z = z$, so $x\cdot y=x\join y$.
\end{proof}

\begin{lemma}\label{lem:incomptimes}
If $x\mathbin{\|} y$, then $x\join y$ exists and $x\join y=x\cdot y$.
\end{lemma}

\begin{proof}
Let $a\in x\setminus y$ and $b\in y\setminus x$. Then $a\notin y$ gives $\neg a\in y'$, and $b\notin x$ gives $\neg b\in x'$. This yields $a\cdot\neg a\in x\cdot y'$ and $b\cdot\neg b\in y\cdot x' = x'\cdot y$. Note that $a\cdot\neg a = a\cdot (a\to\neg t)\leq\neg t\leq t$, and likewise $b\cdot \neg b\leq \neg t\leq t$. By upward closure, we therefore have $\neg t,t\in x\cdot y',x'\cdot y$. We consider some cases.

For the first case, suppose $x,y\notin I({\bf A})$. Then $x\subseteq x'$ and $y\subseteq y'$ by Lemma \ref{lem:orderprime}, so $x\cdot x'=x$ and $y\cdot y'=y$ by Lemma \ref{lem:xprimemeet}.  From $t\in x\cdot y'$ and Lemma \ref{lem:timestech}(2) we have $y\subseteq x\cdot y'\cdot y$. Since $y'\cdot y=y$ by Lemma \ref{lem:xprimemeet}, this gives $y\subseteq x\cdot y$. By the same token, $t\in x'\cdot y$ gives $x\subseteq x\cdot y$. Thus $x,y\subseteq x\cdot y$. If $x,y\subseteq z$, then $x\cdot y\subseteq z$ follows by monotonicity and idempotence, so $x\cdot y=x\join y$.

For the second case, suppose $x\notin I({\bf A})$ and $y\in I({\bf A})$. Then $x\subseteq x'$ and $y'\subseteq y$. Therefore $x\cdot y'\subseteq x\cdot y$. Since $t\in x\cdot y'$, $t\in x\cdot y$ too. Then $x,y\subseteq x\cdot y$, and $x\cdot y$ must be the least among upper bounds for the same reason as before.

The case where $y\notin I({\bf A})$ and $x\in I({\bf A})$ follows by symmetry, and we already knew the case where $x,y\in I({\bf A})$ from Lemma \ref{lem:IofAtimes}.
\end{proof}

\begin{lemma}\label{lem:sandwichtimes}
If $x\subseteq y\subseteq x'$, then $x\cdot y = x$.
\end{lemma}

\begin{proof}
By monotonicity, $x\cdot x\subseteq x\cdot y\subseteq x\cdot x'$. Since $\cdot$ is idempotent, this implies that $x\subseteq x\cdot y\subseteq x\cdot x'$. But $x\cdot x'=x\meet x'=x$ by Lemma \ref{lem:xprimemeet}, so $x\cdot y=x$.
\end{proof}

\begin{lemma}\label{lem:compar}
Let $x,y\in A_*\cup\{A\}$. If $x$ and $y'$ are comparable, then $x$ and $y$ are comparable.
\end{lemma}

\begin{proof}
Suppose that $x$ and $y'$ are comparable. Without loss of generality we may assume that $x\subseteq y'$, since the case where $y'\subseteq x$ follows from swapping the roles of $x$ and $y$ and the fact that $x=(x')'$. We consider cases.

First, suppose that $x\in I({\bf A})$. Then by Lemma \ref{lem:orderprime}(3) we must have $x'\subseteq x$. Thus $x'\subseteq x\subseteq y'$, so $y\subseteq x$ and $x$ and $y$ are comparable. 

Second, suppose that $y'\notin I({\bf A})$. Then Lemma \ref{lem:orderprime}(3) gives that $y'\subseteq y$, and thus $x\subseteq y'$ gives $x\subseteq y$. Hence $x$ and $y$ are again comparable.

In the only remaining case, $x\notin I({\bf A})$ and $y'\in I({\bf A})$. If $y\in I({\bf A})$, then $y,y'\in I({\bf A})$ gives $y=y'$ by Lemma \ref{lem:orderprime}(4), whence $x\subseteq y$ follows immediately. We may therefore assume further that $y\notin I({\bf A})$. In this situation, we have that $x\subset x'$ and $y\subset y'$, and moreover $x\subset y'$ and $y\subset x'$ hold by hypothesis. By the monotonicity and idempotence of $\cdot$, we therefore obtain that $x\cdot y\subseteq x',y'$. Were it the case that $x\cdot y\in I({\bf A})$, this would yield that $x',y'\in\upset x\cdot y$ in $I({\bf A})$, which would give that $x'$ and $y'$ are comparable since $I({\bf A})$ is a forest. This immediately yields that $x$ and $y$ are comparable as well. On the other hand, if $x\cdot y\notin I({\bf A})$, then we argue by contradiction. If $x$ and $y$ are incomparable, then Lemma \ref{lem:incomptimes} gives that $x\join y$ exists and $x\cdot y=x\join y$. Then $x,y\subseteq x\cdot y$, and if $x\cdot y\notin I({\bf A})$ we have that $x,y\subseteq\downset x\cdot y$ in the image under $'$ of $I({\bf A})$. Since $'$ is a dual order isomorphism of $I({\bf A})$ and $\{z' : z\in I({\bf A})\}$, the latter set is a dual forest, and this is a contradiction. It follows that $x$ and $y$ must be comparable as desired.
\end{proof}

The lemma above has significant consequences, one of which is captured in the following.

\begin{corollary}\label{cor:abstotorder}
Let $x,y\in A_*\cup\{A\}$ with $x$ and $y$ comparable. Then the set $\{x,y,x',y'\}$ is a chain under $\subseteq$.
\end{corollary}

\begin{proof}
Lemma \ref{lem:compar} gives that $x$ and $y'$ are comparable, and likewise that $x'$ and $y$ are comparable. Because any $p\in A_*\cup\{A\}$ is comparable to $p'$ by Lemma \ref{lem:orderprime}(2), we have also that $x'$ and $x$ are comparable and $y'$ and $y$ are comparable. Since $x$ and $y$ being comparable implies that $x'$ and $y'$ are comparable as well, this shows that each of $x,y,x',y'$ are pairwise comparable, which gives the result.
\end{proof}

\begin{lemma}\label{lem:notcontainstimes}
If $x\notin I({\bf A})$, $y\in I({\bf A})$, $x\subseteq y$, and $y\not\subseteq x'$, then $x\cdot y=y$.
\end{lemma}

\begin{proof}
$x''=x\subseteq y$ gives that $x'$ and $y$ are comparable by Lemma \ref{lem:compar}. The fact that $y\not\subseteq x'$ gives that $x'\subset y$. Then $x\subseteq x'\subseteq y$, and by monotonicity and idempotence $x\cdot y\subseteq x'\cdot y\subseteq y$. Since $x'\subset y$, we have also $y'\subset x$. Let $a\in x$ with $a\notin y'$. The latter implies that $\neg a\in y$, so $a\cdot \neg a\in x\cdot y$. Then $t\in x\cdot y$, giving $y\subseteq x\cdot y\cdot y = x\cdot y$. Thus $x\cdot y = y$.
\end{proof}

For a bounded Sugihara monoid ${\bf A}$, define the \emph{absolute value} of $x\in A_*$ by $|x|=x\join x'$. By Lemma \ref{lem:orderprime}, for each $x\in A_*$ we have that this join exists, that $|x|=x$ or $|x|=x'$, and that $|x|\in I({\bf A})$.

\begin{lemma}\label{lem:absmult1}
If $|x|\subset |y|$ and $x\subseteq y$, then $x\cdot y=y$.
\end{lemma}

\begin{proof}
We consider cases. Observe at the outset that $|y|=y'$ cannot occur. If this were the case, then $|x|\subset |y|$ would give that $x'\subseteq |x|\subset y'$, whence that $y\subset x$. This contradicts $x\subseteq y$, and is hence impossible. Thus $|y|=y$. There are two possible cases.

First, suppose that $|x|=x$. Then $x,y\in I({\bf A})$, and Lemma \ref{lem:IofAtimes} gives that $x\cdot y=x\join y=y$.

Second, suppose that $|x|=x'$. If $x=x'$, then the previous case applies, so assume further that $x\neq x'$. Then $x\notin I({\bf A})$ by Lemma \ref{lem:orderprime}(4). Since $|y|=y$, we have also that $y\in I({\bf A})$. Because $x'\subset y$ by hypothesis, we have also that $y\not\subseteq x'$. Thus $x\notin I({\bf A})$, $y\in I({\bf A})$, $x\subseteq y$, and $y\not\subseteq x'$. It hence follows from Lemma \ref{lem:notcontainstimes} that $x\cdot y = y$ as desired.
\end{proof}

\begin{lemma}\label{lem:absmult2}
If $|x|\subset |y|$ and $y\subseteq x$, then $x\cdot y=y$.
\end{lemma}

\begin{proof}
Note that it cannot occur that $|y|=y$ since $y\subseteq x$ would then contradict $x\join x' = |x|\subset |y|$, so we have that $|y|=y'$. Then by hypothesis
$$y\subseteq x\subseteq x\join x'=|x|\subset |y|=y'.$$
It follows by Lemma \ref{lem:sandwichtimes} we obtain $x\cdot y = y$.
\end{proof}

\begin{lemma}\label{lem:absmulteq}
If $|x|=|y|$ and $x\subseteq y$, then $x\cdot y=x=x\meet y$.
\end{lemma}

\begin{proof}
The assumption that $|x|=|y|$ gives that either $x=y$ or $x'=y$. In the first case, $x\cdot y=x\cdot x=x=x\meet y$ by the idempotence of $\cdot$. In the second case, we have that $x\subseteq y\subseteq x'$, and Lemma \ref{lem:sandwichtimes} yields that $x\cdot y=x=x\meet y$.
\end{proof}

The following summarizes the results obtained above.

\begin{lemma}\label{lem:filtermult}
Let ${\bf A}$ be a bounded Sugihara monoid and let $x,y\in A_*\cup\{A\}$. Write $x\mathbin{\|}y$ if $x$ and $y$ are incomparable, and $x\perp y$ if $x$ and $y$ are comparable. Then
\[ x\cdot y = \begin{cases} 
      x\vee y & \text{ if }x,y\in I({\bf A})\text{ or } x\mathbin{\|}y\\
      y & \text{ if }x\perp y\text{ and }|x|\subset |y|\\
      x & \text{ if }x\perp y\text{ and }|y|\subset |x|\\
      x\wedge y & \text{ if }x\perp y\text{ and }|x|=|y|\\
   \end{cases}
\]
\end{lemma}

\begin{proof}
Note that if $x,y\in I({\bf A})$, then $x\cdot y=x\join y$ by Lemma \ref{lem:IofAtimes}. If $x\mathbin{\|}y$, then likewise $x\cdot y = x\join y$ by Lemma \ref{lem:incomptimes}.

If $x\perp y$ and one of $|x|\subset |y|$ or $|y|\subset |x|$ holds, then Lemmas \ref{lem:absmult1} and \ref{lem:absmult2} show that $x\cdot y$ is whichever of $x$ or $y$ has the greatest absolute value. If $x\perp y$ and $|x|=|y|$, then Lemma \ref{lem:absmulteq} gives that $x\cdot y = x\meet y$. This proves the claim.
\end{proof}

Observe that in light of Corollary \ref{cor:abstotorder}, if $x$ and $y$ are comparable, then exactly one of $|x|\subset |y|$, $|x|=|y|$, or $|y|\subset |x|$ holds. Hence the above lemma completely describes the multiplication $\cdot$ on $A_*\cup\{A\}$. With this operation now completely understood, we describe how $(-)^{\bowtie}$ operates on objects.

Let ${\bf X} = (X,\leq,D,\tau)$ be a Sugihara space and let $-D^\comp = \{-x : x\in D^\comp\}$ be a copy of $D^\comp$ with $X\cap -D^\comp=\emptyset$. Set $X^\bowtie = X\cup -D^\comp$. We extend our use of the formal symbol $-$ to define a unary operation on $X^{\bowtie}$ by stipulating that $-(-x)=x$ for $-x\in-D^\comp$ and $-x=x$ for $x\in D$. We also extend the order $\leq$ to a partial order $\leq^{\bowtie}$ on $X^{\bowtie}$ via the conditions
\begin{enumerate}
\item If $x,y\in X$, then $x\leq^{\bowtie} y$ if and only if $x\leq y$,
\item If $-x,-y\in -D^\comp$, then $-x\leq^{\bowtie} -y$ if and only if $y\leq x$,
\item If $-x\in -D^\comp$ and $y\in X$, then $-x\leq^{\bowtie} y$ if and only if $x$ and $y$ are comparable with respect to $\leq$.
\end{enumerate}
For a bounded Sugihara monoid ${\bf A}$, define a map $\Gamma_{\bf A}\colon A_*\to I({\bf A})^{\bowtie}$ by

\[ \Gamma_{\bf A}(x) = \begin{cases} 
      x & \text{ if } x\in I({\bf A})\\
      -(x') & \text{ if } x\notin I({\bf A})\\
   \end{cases}
\]
Lemma \ref{lem:orderprime} gives that one of $x\in I({\bf A})$ or $x'\in I({\bf A})$ holds for all $x\in A_*$, and $x=x'=-x$ if both hold. This guarantees that the above map is well-defined.

\begin{lemma}\label{lem:gamorderiso}
$\Gamma_{\bf A}$ is an order isomorphism.
\end{lemma}

\begin{proof}
To see that $\Gamma_{\bf A}$ is order-preserving, let $x,y\in A_*$ with $x\subseteq y$. In the event that $x,y\in I({\bf A})$, then the result is immediate. If $x,y\notin I({\bf A})$, then we may obtain that $\Gamma_{\bf A}(x)=-(x')\leq^{\bowtie} -(y')=\Gamma_{\bf A}(y)$ as $y'\subseteq x'$. If $x\notin I({\bf A})$ and $y\in I({\bf A})$, then there is $z\in I({\bf A})$ with $x=z'$. Since $x$ and $y$ are $\subseteq$-comparable, so too must be $y$ and $x'=z$. In this event, $-z\leq^{\bowtie} y$ gives $\Gamma_{\bf A}(x)\leq^{\bowtie}\Gamma_{\bf A}(y)$.

Next, to see that $\Gamma_{\bf A}$ reflects the order, let $x,y\in A_*$ with $\Gamma_{\bf A}(x)\leq^{\bowtie} \Gamma_{\bf A}(y)$. If $x,y\in I({\bf A})$, then it immediately follows that $x\subseteq y$. If $x,y\notin I({\bf A})$, then there exist $u,v\in I({\bf A})$ with $x=u'$ and $y=v'$ and $\Gamma_{\bf A}(x)=-u$ and $\Gamma_{\bf A}(y)=-v$. then we have $-u\leq^{\bowtie}-v$. By definition, this holds iff $v\subseteq u$, so $x=u'\subseteq v'=y$. In the final case, suppose that $x\notin I({\bf A})$ and $y\in I({\bf A)}$. Then there exists $u\in I({\bf A})$ with $x=u'$, and $\Gamma_{\bf A}(x)=-u$ and $\Gamma_{\bf A}(y)=y$. By definition $-u\leq^{\bowtie}y$ holds iff $u$ and $y$ are $\subseteq$-comparable. If $u\subseteq y$, then $u'\subseteq u\subseteq y$ gives that $x\subseteq y$. If $y\subseteq u$, then $x=u'\subseteq y'\subseteq y$ gives the result. It follows that $\Gamma_{\bf A}$ is order-reflecting.

Since $\Gamma_{\bf A}$ is order-preserving and order-reflecting, it suffices to see that it is onto in order to see that it is an order isomorphism. Let $x\in I({\bf A})^{\bowtie}$. If $x\in I({\bf A})$, then $\Gamma_{\bf A}(x)=x$. If $x\notin I({\bf A})$, then there exists $y\in I({\bf A})$ such that $x=-y$. Then $\Gamma_{\bf A}(y')=-y=x$. This gives the result.
\end{proof}

\begin{lemma}\label{lem:gamprimeiso}
For each $x\in A_*$ we have $\Gamma_{\bf A}(x')=-\Gamma_{\bf A}(x)$.
\end{lemma}

\begin{proof}
Let $x\in A_*$. If $x\in I({\bf A})$ and $x'\notin I({\bf A})$, then we may obtain that $\Gamma_{\bf A}(x')=-(x'')=-x=\Gamma_{\bf A}(x)$. If $x,x'\in I({\bf A})$, then by Lemma \ref{lem:orderprime} we have $x=x'$. It follows that $\Gamma_{\bf A}(x')=x'=x=\Gamma_{\bf A}(x)$. For the final case, if $x\notin I({\bf A})$ and $x'\in I({\bf A})$, then $\Gamma_{\bf A}(x')=x'=-(-(x'))=-\Gamma_{\bf A}(x)$. This proves the claim.
\end{proof}
Taken together, Lemmas \ref{lem:gamorderiso} and \ref{lem:gamprimeiso} show $(A_*,\subseteq, ')$ and $(I({\bf A}), \subseteq^{\bowtie}, -)$ are isomorphic for a bounded Sugihara monoid ${\bf A}$. We extend this isomorphism to associated topological structures.

Let $\tau^{\bowtie}$ be the disjoint union topology on $X\cup -D^\comp$, where the topology on $-D^\comp$ is induced by considering it as a (copy of a) subspace of ${\bf X}$.

\begin{lemma}
$\Gamma_{\bf A}$ is continuous.
\end{lemma}

\begin{proof}
Let $U\cup V\subseteq I({\bf A})^{\bowtie}$ be open, where each of the sets $U\subseteq I({\bf A})$ and $V\subseteq -\{x\in I({\bf A}) : x=x'\}^\comp$ are open. Since $U$ is an open subset of a clopen subspace of ${\bf A}_*$, it is open in ${\bf A}_*$ as well. Moreover, $V$ being open in the set $-\{x\in I({\bf A}) : x=x'\}^\comp$ means precisely that $\{x\in I({\bf A}) : -x\in V\}$ is open in the clopen subspace $\{x\in I({\bf A}) : x\neq x'\}$ of ${\bf A}_*$, hence in ${\bf A}_*$ as well. Because the function $'\colon A_*\to A_*$ is continuous, we have also that the inverse image $\{x' : -x\in V\}$ of $\{x\in I({\bf A}) : -x\in V\}$ under $'$ is open as well. We hence have
\begin{align*}
\Gamma_{\bf A}^{-1}[U\cup V] &= \Gamma_{\bf A}^{-1}[U]\cup \Gamma_{\bf A}^{-1}[V]\\
&= U\cup \{x'\in A_* : -x\in V\}
\end{align*}
is open, which gives the result.
\end{proof}

\begin{lemma}\label{lem:bowtiecomphaus}
Let $(X,\leq,D,\tau)$ be an unpointed Sugihara space. Then $(X^{\bowtie},\tau^{\bowtie})$ is a compact Hausdorff space.
\end{lemma}

\begin{proof}
The subset $D$ is clopen by definition, so $D^\comp$ is a closed subspace of the compact Hausdorff space $(X,\tau)$. It follows that $D^\comp$, and hence its copy $-D^\comp$, is a compact Hausdorff space. Since $(X^{\bowtie},\tau^{\bowtie})$ is the disjoint union of two compact Hausdorff spaces, the result follows.
\end{proof}

\begin{lemma}\label{lem:gamhomeo}
$\Gamma_{\bf A}$ is a homeomorphism.
\end{lemma}

\begin{proof}
We have $(I({\bf A}), \subseteq,D,\tau)$, where $D=\{x\in I({\bf A}) : x\neq x'\}$ and $\tau$ is the topology inherited from ${\bf A}_*$, is an unpointed Sugihara space by Lemma \ref{lem:IASugi}. Lemma \ref{lem:bowtiecomphaus} hence shows that $I({\bf A})^{\bowtie}$ is a compact Hausdorff space. ${\bf A}_*$ is a Priestley space, and hence is compact, so $\Gamma_{\bf A}$ is a continuous bijection from a compact space to a Hausdorff space. It follows that $\Gamma_{\bf A}$ is a homeomorphism.
\end{proof}

Let ${\bf X} = (X,\leq,D,\tau)$ be an unpointed Sugihara space. The duality of Section \ref{sec:sugidual} shows that there exists a bounded Sugihara monoid ${\bf A}$ such that ${\bf X} \cong {\bf A}_+$. Moreover, by Remark \ref{rem:IAiso} we have that ${\bf A}_+$ is isomorphic to $I({\bf A})$ considered as an unpointed Sugihara space. As a consequence, for some bounded Sugihara monoid ${\bf A}$ we have that $(X^{\bowtie},\lesb,-)$ is isomorphic to $(I({\bf A})^{\bowtie}, \subseteq^{\bowtie}, -)$, and hence via $\Gamma_{\bf A}$ to $(A_*,\subseteq,\; ')$. Because the multiplication $\cdot$ on $A_*\cup\{A\}$ is determined entirely by the ordering and the involution $'$, so too is its restriction to a partial multiplication on $A_*$. Consequently, for each object ${\bf X} = (X,\leq,D,\tau)$ of \textsf{SS} we may define a partial multiplication $\cdot$ on $X^{\bowtie}$ by 
\[ x\cdot y = \begin{cases} 
      x\vee y & \text{ if }x,y\in X\text{ or } x\mathbin{\|}y\text{, provided the join exists}\\
      y & \text{ if }x\perp y\text{ and }|x|\slesb |y|\\
      x & \text{ if }x\perp y\text{ and }|y|\slesb |x|\\
      x\wedge y & \text{ if }x\perp y\text{ and }|x|=|y|\\
      \text{undefined} & \text{ otherwise}
   \end{cases}
\]
where $|x|=x$ if $x\in X$, and $|-x|=x$ if $-x\in -D^\comp$. The foregoing remarks along with Lemma \ref{lem:filtermult} show that this definition makes sense, and we may moreover define a ternary relation $R$ on $X^{\bowtie}$ by $Rxyz$ if and only if $x\cdot y$ exists and $x\cdot y\lesb z$. With these definitions, we finally arrive at our construction of the functor $(-)^{\bowtie}$.

\begin{definition}\label{def:upperbowtie}
For an unpointed Sugihara space ${\bf X} = (X,\leq,D,\tau)$, let $X^{\bowtie}$, $\lesb$, $-$, $R$, and $\tau^{\bowtie}$ be as above. Define ${\bf X}^{\bowtie}=(X^{\bowtie},\lesb,R,-,X,\tau^{\bowtie})$. For a morphism $\varphi\colon (X,\leq_X,D_X,\tau_X)\to (Y,\leq_Y,D_Y,\tau_Y)$ of \textsf{SS}, define $\varphi^{\bowtie}\colon{\bf X}^{\bowtie}\to {\bf Y}^{\bowtie}$ by
\[ \varphi^{\bowtie}(x) = \begin{cases} 
     \varphi(x) & \text{ if } x\in X,\\
     -\varphi(-x) & \text{ if } x\in -D_X^\comp
    \end{cases}
\]
\end{definition}

We will shortly show that $(-)^{\bowtie}$ produces a Sugihara relevant space when given an unpointed Sugihara space. We will show that it makes sense on the level of morphisms and provides a reverse functor for $(-)_{\bowtie}$ in Section \ref{sec:equiv}. While $(-)_{\bowtie}$ is a dual version of the enriched negative cone construction, $(-)^{\bowtie}$ is a dual version of the twist product variant appearing in Section \ref{sec:twist}. For reasons illustrated in the following example, we call it the \emph{reflection construction}.

\begin{figure}
\begin{center}
\begin{tikzpicture}
    \node[label=left:\tiny{$h_0$}] at (0,-0.5) {$\bullet$};
    \node[label=left:\tiny{$h_1$}] at (-0.5,-1)  {$\bullet$};
    \node[label=right:\tiny{$h_2$}] at (0.5,-1.25) {\textcircled{$\bullet$}};
    \node[label=left:\tiny{$-h_1$}] at (-0.5,-1.5) {$\bullet$};
    \node[label=left:\tiny{$-h_0$}] at (0,-2) {$\bullet$};

    \draw (0,-0.5) -- (-0.5,-1);
    \draw (0,-0.5) -- (0.5,-1.25);
    \draw (0.5,-1.25) -- (0,-2);
    \draw (-0.5,-1) -- (-0.5,-1.5);
    \draw (-0.5,-1.5) -- (0,-2);
    \draw [dashed] (-2,-1.25) -- (2,-1.25);
\end{tikzpicture}
\end{center}
\caption{Hasse diagram for $({\bf E}_+)^{\bowtie}$}
\label{fig:HasseEreflection}
\end{figure}

\begin{example}\label{ex:Ereflection}
The dual of the bounded Sugihara monoid ${\bf E}_\bot$ was described in Example \ref{ex:Edual}. Figure \ref{ex:Ereflection} shows the result of applying the reflection construction of this section to the dual of ${\bf E}_\bot$. Observe that the elements aside from $h_2$ (which is the sole element of the designated subset) are copied and reflected across an axis determined by the designated subset. One can easily check that this is isomorphic to the Urquhart dual of ${\bf E}_\bot$.
\end{example}

The following verifies that our definition of $(-)^{\bowtie}$ makes sense on the level of objects.

\begin{lemma}
Let ${\bf X}=(X,\leq,D,\tau)$ be an unpointed Sugihara space. Then ${\bf X}^{\bowtie}$ is a Sugihara relevant space.
\end{lemma}

\begin{proof}
By the duality of Section \ref{sec:sugidual}, there exists an unbounded Sugihara monoid ${\bf A}$ such that ${\bf X}\cong {\bf A}_+$. By Remark \ref{rem:IAiso}, the map $\psi_{\bf A}$ witnesses that ${\bf A}_+$ is isomorphic to $(I({\bf A}),\subseteq,D_I,\tau_I)$, where $D_I=\{x\in A_* : x=x'\}$ and $\tau_I$ is the topology on $I({\bf A})$ induced from the topology on ${\bf A}_*$. It follows that ${\bf X}$ is isomorphic to $(I({\bf A}),\subseteq,D_I,\tau_I)$ in the category of unpointed Sugihara spaces, whence that there is a map $\varphi\colon (X^{\bowtie},\lesb,-,\tau^{\bowtie})\to (I({\bf A})^{\bowtie},\subseteq^{\bowtie},-,\tau_I^{\bowtie})$ that is an order isomorphism, homeomorphism, and preserves $-$. Because $\Gamma_{\bf A}$ is an order isomorphism by Lemma \ref{lem:gamorderiso}, a homeomorphism by Lemma \ref{lem:gamhomeo}, and preserves the involution by Lemma \ref{lem:gamprimeiso}, we have that $\delta=\Gamma_{\bf A}^{-1}\circ \varphi$ is an order isomorphism, homeomorphism, and preserves the involution. Because ${\bf A}_*$ is a Sugihara relevant space, in order to show that ${\bf X}^{\bowtie}$ is as well it suffices to show that $\delta[X]=I({\bf A})$ and that for any $x,y,z\in X^{\bowtie}$, $Rxyz$ if and only if $R\delta(x)\delta(y)\delta(z)$.

Note that $\Gamma_{\bf A}$ and $\varphi$ being bijections gives that
$$\delta[X]=(\Gamma_{\bf A}^{-1}\circ \varphi)[X]=\Gamma^{-1}[I({\bf A})]=I({\bf A})$$
It remains only to show that $\delta$ is an isomorphism with respect to $R$, so let $x,y,z\in X^{\bowtie}$. Then by definition $x\cdot y$ exists and $x\cdot y\lesb z$. But since $\cdot$ is defined in terms of $-$ and the order $\lesb$ and $\delta$ preserves this structure, $x\cdot y\lesb z$ must hold exactly when $\delta(x)\cdot\delta(y)\subseteq\delta(z)$ holds in ${\bf A}_*$, i.e., exactly when $R\delta(x)\delta(y)\delta(z)$ holds. It follows that ${\bf X}^{\bowtie}$ is a Sugihara relevant space isomorphic to ${\bf A}_*$.
\end{proof}

\subsection{An equivalence between \textsf{SS} and \textsf{SRS}}\label{sec:equiv}
We turn our attention to verifying that $(-)_{\bowtie}$ and $(-)^{\bowtie}$ really extend to functors in the manner previously described, and provide an equivalence between \textsf{SS} and \textsf{SRS}. We first verify that our definitions make sense for morphisms.

\begin{lemma}
Let $\varphi\colon {\bf X}\to {\bf Y}$ be a morphism of \textsf{SRS}. Then $\varphi_{\bowtie}$ is a morphism of \textsf{SS}.
\end{lemma}

\begin{proof}
Note that since $\varphi$ is a relevant map, we have that $\varphi^{-1}[Y_{\bowtie}]=X_{\bowtie}$ by definition. This implies that $\varphi[X_{\bowtie}]= \varphi[\varphi^{-1}[Y_{\bowtie}]]\subseteq Y_{\bowtie}$, so $\varphi\restriction_{X_{\bowtie}}$ has image in $Y_{\bowtie}$ and $\varphi_{\bowtie}$ is well-defined.

$\varphi_{\bowtie}$ is a continuous isotone map because it is the restriction of a continuous isotone map. To see that $\varphi_{\bowtie}$ is an Esakia map, suppose that $x\in X_{\bowtie}$, $z\in Y_{\bowtie}$ with $\varphi_{\bowtie}(x)\leq z$. Then since $\varphi(x),z\in Y_{\bowtie}$, the definition of $\cdot$ provides that $\varphi(x)\cdot z = \varphi(x)\join z = z$ and $R_Y\varphi(x)zz$. As $\varphi$ is a relevant map, this gives that there exists $u,v\in X$ with $R_Xxuv$, $z\leq\varphi(u)$, and $\varphi(v)\leq z$. That $z\leq\varphi(u)$ and $z\in Y_{\bowtie}$ give $\varphi(u)\in Y_{\bowtie}$. Note that $\varphi(u)\in Y_{\bowtie}$ implies that $u\in\varphi^{-1}[Y_{\bowtie}]=X_{\bowtie}$ since $\varphi$ is a relevant map. Since $x,u\in X_{\bowtie}$, the definition of $\cdot$ gives $x\cdot u = x\join u$. But $R_Xxuv$ gives that $x\cdot u\leq v$, so $x,u\leq x\join u\leq v$. It follows by monotonicity that $\varphi(v)\leq z\leq \varphi(u)\leq\varphi(v)$, so $x\leq v$ and $z=\varphi(v)$. This yields that $\varphi_{\bowtie}$ is a $p$-morphism.

Finally, note that if $x\in X$ with $x=x'$, then $\varphi_{\bowtie}(x) = \varphi_{\bowtie}(x)'$ as $\varphi$ preserves $'$. On the other hand, if $x\neq x'$, then we may assume without loss of generality that $x\in X_{\bowtie}$ and $x'\notin X_{\bowtie}=\varphi^{-1}[Y_{\bowtie}]$. Then $\varphi(x)\in Y_{\bowtie}$ and $\varphi(x')\notin Y_{\bowtie}$, so $\varphi(x)\neq\varphi(x)'$. This yields the result.
\end{proof}

Given a \textsf{SS}-morphism $\varphi\colon {\bf X}\to {\bf Y}$, the function $\varphi^{\bowtie}$ is a relevant map. For comprehensibility we divide the proof into pieces.

\begin{lemma}\label{lem:pbowiso}
Let $\varphi\colon {\bf X}\to {\bf Y}$ be a morphism of \textsf{SS}. Then $\varphi^{\bowtie}$ is isotone.
\end{lemma}

\begin{proof}
Suppose that $x\lesb y$. We consider cases.

First, if $x,y\in X$, then $\pbow(x) =\varphi(x)\leq\varphi(y)=\pbow(y)$ follows from the isotonicity of $\varphi$.

Second, if $x,y\notin X$, then $x\lesb y$ implies $-y\leq -x$. The isotonicity of $\varphi$ gives $-\pbow(y)=\varphi(-y)\leq\varphi(-x)=-\pbow(x)$, yielding $\pbow(x)\lesb\pbow(y)$.

Third, suppose that $x\notin X$ and $y\in X$. Then $x\notin X$ gives that $-x\in X$, and $x\lesb y$ gives that $-x$ and $y$ are $\leq$-comparable. Since $\varphi$ is isotone, this gives that either $-\pbow(x)=varphi(-x)$ and $\pbow(y)=\varphi(y)$ are $\leq$-comparable as well. Note that $\pbow(x)\notin Y$ by the definition of $\pbow$ since $x\notin X$. Hence by the definition of $\lesb$ we have that $\pbow(x)\lesb\pbow(y)$. This proves the lemma.
\end{proof}

\begin{lemma}\label{lem:pbowprime}
Let $\varphi\colon {\bf X}\to {\bf Y}$ be a morphism of \textsf{SS}. Then for any $x\in X^{\bowtie}$, $\pbow(-x)=-\pbow(x)$.
\end{lemma}

\begin{proof}
We consider cases. First, if $x\in X\setminus D_X$, then $-x\in -D_X^\comp$ gives that $\pbow(-x)=-\varphi(-(-x))=-\varphi(x)=-\pbow(x)$. Second, if $x\in D_X$ then we obtain $\pbow(-x)=\pbow(x)=-\pbow(x)$. Third, if $x\in -D_X^\comp$, then $-x\in X\setminus D_X$ gives that $\pbow(-x)=\varphi(-x)=-(-\varphi(-x))=-\pbow(x)$, which gives the result.
\end{proof}

\begin{lemma}\label{lem:abspres}
Let $\varphi\colon {\bf X}\to {\bf Y}$ be a morphism of \textsf{SS}. Then for any $x\in X^{\bowtie}$, $\pbow(|x|)=|\pbow(x)|$.
\end{lemma}

\begin{proof}
Let $x\in X^{\bowtie}$. Then either $-x\lesb x$ or $x\lesb -x$. Since $\pbow$ preserves the ordering $\lesb$ by Lemma \ref{lem:pbowiso} and preserves $-$ by Lemma \ref{lem:pbowprime}, we have that $-\pbow(x)\lesb \pbow(x)$ in the first case, and $\pbow(x)\lesb -\pbow(x)$ in the second case. In the first case, we therefore have $\pbow(x)\join -\pbow(x) = \pbow(x) =\pbow(|x|)$, and in the second case we have $\pbow(x)\join -\pbow(x) = -\pbow(x)=\pbow(-x)=\pbow(|x|)$. In either event, the result follows.
\end{proof}

\begin{lemma}
Let $\varphi\colon {\bf X}\to {\bf Y}$ be a morphism of \textsf{SS}. Then $\varphi^{\bowtie}$ preserves the ternary relation $R$.
\end{lemma}

\begin{proof}
Let $x,y,z\in X^{\bowtie}$ with $R_Xxyz$. Then $x\cdot y$ exists and $x\cdot y\leq^{\bowtie} z$. We consider two cases.

First, suppose that $x\cdot y = x\join y$. Then $x\join y\leq^{\bowtie} z$, so $x\leq^{\bowtie} z$ and $y\leq^{\bowtie} z$. Since $\pbow$ preserves the order, $\pbow(x),\pbow(y)\lesb\pbow(z)$. Since $\cdot$ is order-preserving and idempotent, this gives $\pbow(x)\cdot\pbow(y)\lesb\pbow(z)$, hence $R_Y\pbow(x)\pbow(y)\pbow(z)$.

Second, suppose that $x\cdot y\neq x\join y$. Then the definition of the partial multiplication $\cdot$ shows that $x\cdot y$ is one of $x$ or $y$ and $x\perp y$. Without loss of generality we may assume that $x\lesb y$ and (since $x\cdot y\neq x\join y$) that $x\cdot y = x$. In this situation, the definition of $\cdot$ gives that $|y|\lesb |x|$. Note Lemma \ref{lem:pbowiso} shows that $\pbow(x)\lesb\pbow(y)$, so $\pbow(x)\cdot\pbow(y)$ must exist by the definition of $\cdot$. Moreover, the fact that $|y|\lesb |x|$ together with Lemmas \ref{lem:pbowiso} and \ref{lem:abspres} give that $|\pbow(y)|\lesb |\pbow(x)|$. The definition of $\cdot$ then shows that $\pbow(x)\cdot\pbow(y)$ is either $\pbow(x)\meet\pbow(y)$ or whichever of $\pbow(x)$ and $\pbow(y)$ has greater absolute value, but this gives $\pbow(x)\cdot\pbow(y)=\pbow(x)$ in either cases. Because $x=x\cdot y\lesb z$, we hence have $\pbow(x)\cdot\pbow(y)=\pbow(x)\lesb\pbow(z)$, which gives $R_Y\pbow(x)\pbow(y)\pbow(z)$ as desired.
\end{proof}

\begin{lemma}
Let $\varphi\colon {\bf X}\to {\bf Y}$ be a morphism of \textsf{SS}. Then if $R_Yxy\varphi^{\bowtie}(z)$, there exists $u,v\in X^{\bowtie}$ such that $R_Xuvz$, $x\lesb\varphi^{\bowtie}(u)$, and $y\lesb\varphi^{\bowtie}(v)$.
\end{lemma}

\begin{proof}
Suppose that $R_Yxy\pbow(z)$. Then $x\cdot y$ exists and $x\cdot y\lesb\pbow(z)$. We consider two cases.

First, suppose that $x\cdot y = x\join y$. Then $x\lesb\pbow(z)$ and $y\lesb\pbow(z)$. Taking $u=v=z$ gives the result as $R_Xzzz$.

Second, suppose that $x\cdot y\neq x\join y$. Then from the definition of $\cdot$ we have that $x\perp y$ and $x\cdot y$ is one of $x$ or $y$. We may assume without loss of generality that $x\lesb y$, that $x\cdot y = x$ (for if $x\cdot y=y$, then $x\cdot y = x\join y$, a contradiction), and that $|y|\lesb |x|$. Because $x,y\in Y$ would give that $x\cdot y=x\join y$ by the definition of $\cdot$, we may further assume that $x\notin Y$ and hence that $|x|=-x$ (for otherwise $x\lesb y$ and $Y$ being upward-closed would give $x,y\in Y$). Note that in this situation the hypothesis that $x=x\cdot y\lesb\pbow(z)$ gives that $\pbow(-z)\lesb -x$. It follows that $\pbow(|z|)$ must be comparable to $-x$ by Corollary \ref{cor:abstotorder} (as transferred along the obvious isomorphism), and we have either $\pbow(|z|)\lesb -x$ or $-x\lesb\pbow(|z|)$.

If $\pbow(|z|)\lesb -x$, then $\varphi(|z|)\leq -x$ and $\varphi$ being a p-morphism gives that there exists $u\in X$ such that $|z|\leq u$ and $\varphi(u)=-x$. Then $-u\lesb -|z|\lesb z$ and $y\lesb |y|\lesb |x|=-x\lesb\pbow(u)$, so $x\lesb\pbow(-u)$, $y\lesb\pbow(u)$, and $(-u)\cdot u= -u\lesb z$ gives the result.

If $-x\lesb\pbow(|z|)$, then $|y|\lesb |x|=-x$ gives that $y\lesb\pbow(|z|)$. Observing that $z\cdot |z| = z\meet |z|=z$, we obtain that $x\lesb\pbow(z)$, $y\lesb\pbow(|z|)$, and $R_Xz|z|z$, giving the result.
\end{proof}

\begin{lemma}
Let $\varphi\colon {\bf X}\to {\bf Y}$ be a morphism of \textsf{SS}. Then if $R_Y\varphi^{\bowtie}(x)yz$, there exists $u,v\in X^{\bowtie}$ such that $R_Xxuv$, $y\lesb\varphi^{\bowtie}(u)$, and $\varphi^{\bowtie}(v)\lesb z$.
\end{lemma}

\begin{proof}
The fact that $R_Y\pbow(x)yz$ gives that $\pbow(z)\cdot y$ exists and $\pbow(x)\cdot y\lesb z$. We again consider cases.

For the first case, suppose that $\pbow(x)\cdot y = \pbow(x)\join y\lesb z$. Then $\pbow(x)\lesb z$ and $y\lesb z$. If $\pbow(x)\in Y$ (Subcase 1.1), then the p-morphism condition gives that there exists $u\in X$ with $x\leq u$ and $\varphi(u)=\pbow(u)=z$. Then $y\lesb\pbow(u)$, $\pbow(u)\lesb z$, and $R_Xxuu$ since $x\cdot u\lesb u$ follows from $x\lesb u$ by monotonicity and idempotence.

If $\pbow(x)\notin Y$ (Subcase 1.2), then we may assume that $\pbow(x)$ and $y$ are incomparable (i.e., since we are in the case where $\pbow(x)\cdot y=\pbow(x)\join y$). Moreover, $-\pbow(x)=\pbow(-x)\in Y$ and $-z\lesb\pbow(-x)$, $-z\lesb -y$. Were $-z\in Y$, this would contradict the fact that $Y$ is a forest, so $-z\notin Y$ and hence $z\in Y$. The fact that $-z$ and $\pbow(-x)$ are comparable gives that $z$ and $\pbow(-x)$ are comparable.

In the event that $z\lesb\pbow(-x)$ (Subcase 1.2.1), then $y\lesb\pbow(-x)$ and $\pbow(x)\lesb -z\lesb z$. The result follows in this situation from the fact that $-x\cdot x = x$ and hence $R_Xx(-x)x$.

In the situation that $\pbow(-x)\lesb z$ (Subcase 1.2.2), we note that $\pbow(x)\notin Y$ gives that $\pbow(-x)\in Y$ and $-x\in X$. Then $\varphi$ being a p-morphism gives that there exists $u\in X$ with $-x\leq u$ and $\varphi(u)=\pbow(u)=z$. Then since $x\notin X$, we have that $x\lesb -x\lesb u$ and this gives $x\cdot u\lesb u$. Since $y\lesb z = \pbow(u)$, $\pbow(u)\lesb z$, the fact that $R_Xxuu$ gives the result. This completes the first case.

For the remaining cases, we may assume that $\pbow(x)$ and $y$ are comparable and that not both of $\pbow(x)$ and $y$ are contained in $Y$. For the second case, assume that $|\pbow(x)|=|y|$, and thus that $\pbow(x)\cdot y = \pbow(x)\meet y$.

Suppose that $\pbow(x)\lesb y$ (Subcase 2.1). Then $\pbow(x)\cdot y = \pbow(x)\lesb z$. From $|\pbow(x)|=|y|$, we have $\pbow(x)=y$ or $\pbow(x)=-y$. If $\pbow(x)=y$, then $R_Xxxx$ gives the result. If $\pbow(x)=-y$, then $\pbow(-x)=y$ and $R_Xx(-x)x$ gives the result.

Now suppose that $y\lesb\pbow(x)$ (Subcase 2.2). Then $\pbow(x)\cdot y = y\lesb z$. Again, $|\pbow(x)|=|y|$ gives $\pbow(x)=y$ or $\pbow(x)=-y$. The former gives the result from $R_Xxxx$. The latter gives $\pbow(-x)=y\lesb z$, so $R_Xx(-x)(-x)$ gives the result. This yields the second case.

For the third case, suppose that $|y|<|\pbow(x)|$. Then $\pbow(x)\cdot y = \pbow(x)\lesb z$. If $y\lesb\pbow(x)$ (Subcase 3.1), this case may be concluded with $R_Xxxx$. On the other hand, if $\pbow(x)\lesb y$ (Subcase 3.2), we may assume that $\pbow(x)\notin Y$, hence that $\pbow(-x)\in Y$. Then $\pbow(-x)=|\pbow(x)|$, so $y\lesb |y|\lesb \pbow(-x)$. Then $R_Xx(-x)x$ gives the result and the third case.

For the fourth case, suppose that $|\pbow(x)|<|y|$. Then $\pbow(x)\cdot y = y\lesb z$. If $\pbow(x),y\notin Y$ (Subcase 4.1), then $|\pbow(x)|=-\pbow(x)\lesb -y = |y|$. This gives $\pbow(-x)\leq -y$ and the p-morphism condition implies that there exists $u\in Y$ with $-x\leq u$ and $\pbow(u)=\varphi(u)=-y$, whence $\pbow(-u)=y\lesb z$. Then $-u\lesb x$, and the fact that $-u,x\notin X$ gives that $x\cdot (-u) = -u$ since the value of $x\cdot (-u)$ is either the meet or the one with the larger absolute value. Hence $R_Xx(-u)(-u)$ and $y=\pbow(-u)\lesb z$ give the result. In the only remaining case, $\pbow(x)\in Y$ and $y\notin Y$ (Subcase 4.2). Then $|\pbow(x)|=\pbow(x)\lesb -y=|y|$. Since $\varphi$ is a p-morphism, this implies that there exists $u\in X$ with $x\leq u$ and $\pbow(u)=\varphi(u)=-y$. Then $y=\pbow(-u)$ and $y\lesb z$ hence yields $\pbow(-u)\lesb z$. Since $x\lesb u$, by monotonicity of $\cdot$ we have $x\cdot (-u) \lesb u\cdot (-u) = u\meet -u\lesb -u$. This gives $R_Xx(-u)(-u)$, and since $y\lesb\pbow(-u)$ and $\pbow(-u)\lesb z$, this settles the fourth case. This completes the proof.
\end{proof}

\begin{lemma}
Let $\varphi\colon {\bf X}\to {\bf Y}$ be a morphism of \textsf{SS}. Then $\varphi^{\bowtie}$ is continuous.
\end{lemma}

\begin{proof}
Let $U\cup V\subseteq {\bf Y}^{\bowtie}$ be open, where $U\subseteq Y$ and $V\subseteq\-D_Y^\comp$ are open. Note that the map $-\colon {\bf Y}^{\bowtie}\to{\bf Y}^{\bowtie}$ is a continuous bijection of compact Hausdorff spaces, and is therefore a homeomorphism. By definition, $(\pbow))^{-1}[V]$ is exactly the set $\{x\in Y^{\bowtie} : -\varphi(-x)\in V\}$. This is precisely $\{-x\in Y^{\bowtie} : \varphi(-x)\in V\}$, so it is the inverse image of $V$ under the continuous composite map $\varphi\circ -$. and hence the inverse image of $V$ under this map is open. Since $(\varphi^{\bowtie})^{-1}[U\cup V] = (\varphi^{\bowtie})^{-1}[U]\cup(\varphi^{\bowtie})^{-1}[V]$, the result follows.
\end{proof}

\begin{lemma}
Let $\varphi\colon {\bf X}\to {\bf Y}$ be a morphism of \textsf{SS}. Then $\varphi^{\bowtie}$ is a relevant map.
\end{lemma}

\begin{proof}
Previous lemmas show that $\varphi^{\bowtie}$ is a continuous, isotone map that preserves and is a p-morphism with respect to the ternary relation $R$. We also have that $(\varphi^{\bowtie})^{-1}[Y] = \varphi^{-1}[Y]=X$, and $\varphi^{\bowtie}(-x)=-\varphi^{\bowtie}(x)$ by Lemma \ref{lem:pbowprime}. This proves the result.
\end{proof}

\begin{lemma}
$(-)_{\bowtie}\colon\textsf{SRS}\to\textsf{SS}$ is functorial.
\end{lemma}

\begin{proof}
Let $\varphi\colon{\bf Y}\to {\bf Z}$ and $\psi\colon{\bf X}\to{\bf Y}$ be morphisms of \textsf{SRS}. We must show that $(\varphi\circ\psi)_{\bowtie} = \varphi_{\bowtie}\circ\psi_{\bowtie}$. Let $x\in X_{\bowtie}$. Then $(\varphi\circ\psi)_{\bowtie}(x) = \varphi(\psi(x)) = \varphi_{\bowtie}(\psi_{\bowtie}(x))$ follows immediately since $(-)_{\bowtie}$ acts by restriction. That $(-)_{\bowtie}$ preserves the identity morphism is obvious.
\end{proof}

\begin{lemma}
$(-)^{\bowtie}\colon\textsf{SS}\to\textsf{SRS}$ is functorial.
\end{lemma}

\begin{proof}
Given Sugihara spaces ${\bf X} = (X,\leq_{\bf X},D_{\bf X},\tau_{\bf X})$, ${\bf Y} = (Y,\leq_{\bf Y},D_{\bf Y},\tau_{\bf Y})$, and ${\bf Z} = (Z,\leq_{\bf Z},D_{\bf Z},\tau_{\bf Z})$, let $\varphi\colon{\bf Y}\to {\bf Z}$ and $\psi\colon{\bf X}\to{\bf Y}$ be morphisms of \textsf{SS}. Let $x\in X^{\bowtie}$. Then $x\in X$ or $x\in\{-y : y\notin D_{\bf X}\}$. In the former case, we immediately obtain that $(\varphi\circ\psi)^{\bowtie}(x)=(\varphi\circ\psi)(x)=\varphi(\psi(x))=\varphi^{\bowtie}(\psi^{\bowtie}(x))$ from the definition. If $x=-y$ where $y\notin D_{\bf X}$, then $(\varphi\circ\psi)^{\bowtie}(x)=-(\varphi\circ\psi)(y) = -\varphi(\psi(y))$. On the other hand, $\psi^{\bowtie}(x) = -\psi(y)$ is not in $Y$, and hence $\varphi^{\bowtie}(-\psi(y))=-\varphi(\psi(y))$. This shows that $(\varphi\circ\psi)^{\bowtie}=\varphi^{\bowtie}\circ\psi^{\bowtie}$ in each case. That $(-)^{\bowtie}$ preserves the identity morphism is obvious, so this gives the result
\end{proof}

\begin{lemma}
Let ${\bf X} = (X,\leq,R,\;',I,\tau)$ be a Sugihara relevant space. Then $({\bf X}_{\bowtie})^{\bowtie}\cong {\bf X}$.
\end{lemma}

\begin{proof}
Define a map $\theta_{\bf X}\colon ({\bf X}_{\bowtie})^{\bowtie}\to {\bf X}$ by
\[ \theta_{\bf X}(x) = \begin{cases} 
      x & \text{ if } x\in I\\
      (-x)' & \text{ if } x\notin I
   \end{cases}
\]
Since $x\notin I$ implies that $-x\in I$ is an element of ${\bf X}$, this map is well-defined. We will show that $\theta_{\bf X}$ is an isomorphism in \textsf{SRS}. Following \cite{Urquhart}, it suffices to show that $\theta_{\bf X}$ is an order isomorphism, homeomorphism, preserves the involution, is an isomorphism with respect to $R$, and satisfies $\theta_{\bf X}[I]=I$.

To see that $\theta_{\bf X}$ is an order isomorphism, first suppose that $x,y\in ({\bf X}_{\bowtie})^{\bowtie}$ with $x\lesb y$. If $x,y\in X_{\bowtie}$, then this means that $\theta_{\bf X}(x)=x\leq y=\theta_{\bf X}(y)$. If $x,y\notin X_{\bowtie}$, then $-x,-y\in X_{\bowtie}$ and $x\lesb y$ means $-y\leq -x$, hence $(-x)'\leq (-y)'$. Then $\theta_{\bf X}(x)=(-x)'\leq (-y)'=\theta_{\bf X}(y)$. Finally, if $x\notin X_{\bowtie}$ and $y\in X_{\bowtie}$, then $x\lesb y$ gives that $-x$ and $y$ are $\leq$-comparable. If $-x\leq y$, then $(-x)'\leq -x\leq y$, and if $y\leq -x$, then $(-x)'\leq y'\leq y$. In either case, $\theta_{\bf X}(x)\leq\theta_{\bf X}(y)$. This shows that $\theta_{\bf X}$ preserves the order.

To show that it reflects the order, let $x,y\in ({\bf X}_{\bowtie})^{\bowtie}$ with $\theta_{\bf X}(x)\leq \theta_{\bf X}(y)$. If $x,y\in X_{\bowtie}$, then $x\lesb y$ is immediate. If $x,y\notin X_{\bowtie}$, then we have $(-x)'\leq (-y)'$, whence $-y\leq -x$. In this case, $-x,-y\in X_{\bowtie}$, so it follows that $x\lesb y$ from the definition. If $x\in X_{\bowtie}$ and $y\notin X_{\bowtie}$, then $x=\theta_{\bf X}(x)\leq\theta_{\bf X}(y)=\lesb (-y)'$. But $y\notin X_{\bowtie}$ implies that $(-y)'\notin X_{\bowtie}$, so this contradicts the fact that $X_{\bowtie}$ is an upset and hence cannot occur. For the final case, suppose that $x\notin X_{\bowtie}$ and $y\in X_{\bowtie}$. Then $(-x)'\leq y$ by hypothesis. Since $y$ and $-x$ are comparable, we obtain also that $-x$ and $y$ are comparable with $-x,y\in X_{\bowtie}$. By the definition of $\lesb$, this entails $x=-(-x)\lesb y$. This yields that $\theta_{\bf X}$ is order-reflecting.

To see that $\theta_{\bf X}$ is an order isomorphism, we show that it is onto. Let $x\in X$. If $x\in I$, then $x\in (X_{\bowtie})^{\bowtie}$ as well and $\theta_{\bf X}(x)=x$. If $x\notin I$, then $x'\in I$ and hence $-(x')\in (X_{\bowtie})^{\bowtie}$ and $-(x')\notin X_{\bowtie}$. Then $\theta_{\bf X}(-(x'))=(-(-(x')))'=x''=x$. This gives that $\theta_{\bf X}$ is an order isomorphism.

We turn to showing that $\theta_{\bf X}$ is a homeomorphism. The above shows that $\theta_{\bf X}$ is a bijection, so since $({\bf X}_{\bowtie})^{\bowtie}$ and ${\bf X}$ are compact Hausdorff spaces, it suffices to show that $\theta_{\bf X}$ is continuous. Let $W\subseteq X$ be open, and set $U=W\cap I$ and $V=W\cap I^\comp$. Since $I$ is open by definition, both $U$ and $V$ are open as well. By definition, $\theta_{\bf X}^{-1}[U]=U$. Observe that $\theta_{\bf X}(x)\notin I$ implies that $x\notin I$ because $x\in I$ would gives $\theta_{\bf X}(x)=x$. Using this fact, we obtain
\begin{align*}
\theta_{\bf X}^{-1}[V] &= \{x\in (X_{\bowtie})^{\bowtie} : \theta_{\bf X}(x)\in V\}\\
&= \{x\in (X_{\bowtie})^{\bowtie} : (-x)'\in V\}\\
\end{align*}
Now $'\colon {\bf X}\to{\bf X}$ and $-\colon ({\bf X}_{\bowtie})^{\bowtie}\to ({\bf X}_{\bowtie})^{\bowtie}$ are continuous bijections by definition, and the above is precisely the inverse image of $V$ under the composition of $-$ and $'$. It follows that $V$ is an open subset of $({\bf X}_{\bowtie})^{\bowtie}$ disjoint from $X_{\bowtie}$, whence $\theta_{\bf X}^{-1}[W]=\theta_{\bf X}^{-1}[U]\cup\theta_{\bf X}^{-1}[V]$ is open. It follows that $\theta_{\bf X}$ is a homeomorphism.

To see that $\theta_{\bf X}$ preserves the involution, let $x\in ({\bf X}_{\bowtie})^{\bowtie}$. If $-x\notin {\bf X}_{\bowtie}$, then $x\in {\bf X}_{\bowtie}$ and $\theta_{\bf X}(-x)=(-(-x))'=x'=\theta_{\bf X}(x)'$. If $-x\in {\bf X}_{\bowtie}$ with $-x=x$, then by definition $x=x'$ and $\theta_{\bf X}(-x)=-x=x=x'=\theta_{\bf X}(x)'$. If $-x\in {\bf X}_{\bowtie}$ with $-x\neq x$, then $x\notin X_{\bowtie}$ and $\theta_{\bf X}(-x)=-x=(-x)''=\theta_{\bf X}(x)'$. This gives the preservation of the involution.

That $\theta_{\bf X}[I]=I$ is immediate from $\theta_{\bf X}(x)=x$ for $x\in I$, so it remains only to show that $\theta_{\bf X}$ is an isomorphism with respect to $R$. But this follows immediately since $R$ is completely determined by the meet, join, and involution, and $\theta_{\bf X}$ is an involution-preserving order isomorphism. This gives the result.
\end{proof}

\begin{lemma}
Let ${\bf X}$ be an unpointed Sugihara space. Then $({\bf X}^{\bowtie})_{\bowtie}\cong {\bf X}$.
\end{lemma}

\begin{proof}
Let $i_{\bf X}\colon ({\bf X}^{\bowtie})_{\bowtie}\to {\bf X}$ be the identity map. Then $i_{\bf X}$ is obviously an isomorphism of \textsf{SS}, and the result follows.
\end{proof}

\begin{theorem}
$(-)_{\bowtie}$ and $(-)^{\bowtie}$ witness an equivalence of categories between \textsf{SRS} and \textsf{SS}.
\end{theorem}

\begin{proof}
The lemmas above yield this result provided that we show that the maps $\theta_{\bf X}$ and $i_{\bf X}$ are natural isomorphisms. This is obvious in the latter case, so we need only check the naturality of $\theta_{\bf X}$. Let $\varphi\colon {\bf X}\to {\bf Y}$ be a morphism of \textsf{SRS}. We must show that $\varphi\circ\theta_{\bf X} = \theta_{\bf Y}\circ (\varphi_{\bowtie})^{\bowtie}$, so let $x\in (X_{\bowtie})^{\bowtie}$. If $x\in X_{\bowtie}$, then providing $x$ as an input yields $\varphi(x)$ on both sides of this equation. If $x\notin X_{\bowtie}$, then both sides become $\varphi(-x)'$. This gives the result, and yields the equivalence.
\end{proof}

\section{Conclusion}

The foregoing analysis reveals a rich web of pairwise equivalences among various categories associated to ${\bf R}$-mingle. Although each of these equivalences is of interest in its own right, their mutually-supporting structure provides insight above and beyond that afforded by any of them individually. The Sugihara monoids have two features that allow for this sort of analysis. First, they have reducts among the normal $i$-lattices, granting access to the Davey-Werner duality and its connection to twist product constructions. Second, they are semilinear, which (among many other consequences) allows for the characterization of the ternary relation of the Urquhart duality in terms of the partial multiplication on prime filters only. Due to the powerful consequences of these properties, we expect that a similar analysis to that conducted here is possible for other classes of semilinear residuated lattices with normal $i$-lattice reducts.

\bibliographystyle{plain}
\bibliography{bibCatModRM}

\end{document}